\documentclass[10pt]{article}
\usepackage{amsmath,amsthm,amsfonts,amssymb,bm,bbm,enumerate, graphicx, mathtools,color,comment}
\usepackage[english]{babel}
\usepackage[utf8]{inputenc}
\usepackage[toc,page]{appendix}
\usepackage[capitalise]{cleveref}
\usepackage{array, tabularx, graphics, multirow, tcolorbox}

\newcommand{\pr}[1]{\mathbb{P}\!\left(#1\right)}
\newcommand{\E}[1]{\mathbb{E}\!\left[#1\right]}
\newcommand{\Var}[1]{\text{Var}\!\left(#1\right)}

\newcommand{\prstart}[2]{\mathbb{P}_{#2}\!\left(#1\right)}

\newcommand{\prcond}[3]{\mathbb{P}_{#3}\!\left(#1\;\middle\vert\;#2\right)}
\newcommand{\econd}[2]{\mathbb{E}\!\left[#1\;\middle\vert\;#2\right]}

\newcommand{\Pb}{\mathbb{P}}
\newcommand{\Pbt}{\tilde{P}}

\newcommand{\prb}[1]{\mathbf{P}\!\left(#1\right)}
\newcommand{\Eb}[1]{\mathbf{E}\!\left[#1\right]}

\newcommand{\pt}[1]{\tilde{P}\!\left(#1\right)}

\newcommand{\ptstart}[2]{\tilde{P}_{#2}\!\left(#1\right)}

\newcommand{\estartpt}[2]{\tilde{E}_{#2}\!\left[#1\right]}

\newcommand{\dGH}[1]{d_{GH}\!\left(#1\right)}
\newcommand{\dGHP}[1]{d_{GHP}\!\left(#1\right)}

\newtheorem{theorem}{Theorem}[section]
\newtheorem{lemma}[theorem]{Lemma}
\newtheorem{prop}[theorem]{Proposition}
\newtheorem{corollary}[theorem]{Corollary}
\newtheorem{defn}[theorem]{Definition}
\newtheorem{remark}[theorem]{Remark}
\newtheorem{claim}[theorem]{Claim}

\parskip \medskipamount
\def\N{\mathbb{N}}
\def\Z{\mathbb{Z}}

\def\R{\mathbb{R}}

\def\bP{\mathbb{P}}
\def\bPb{\mathbf{P}}

\def\F{\mathcal{F}}

\renewcommand{\epsilon}{\varepsilon}

\def\T{\mathcal{T}}
\def\Levy{L\'{e}vy }
\def\Ito{It\^o }

\def\cadlag{c\`{a}dl\`{a}g }

\def\d{\tilde{d}}

\def\diam{\textsf{Diam}}

\def\Sr2{S^{(\frac{1}{2}r)}_{\sigma}}

\def\Xi{X^{\infty}}
\def\dGHP{d_{GHP}}

\def\Height{\textsf{Height}}

\def\dt{d}
\def\Tgc{\T_{\gamma}^c}
\def\dbt{d^{(\alpha)}}
\def\dbtn{d_n^{(\alpha)}}

\def\dbztn{d_n^{(0)}}
\def\dzt{d^{(1)}}
\def\phib{\phi^{(\alpha)}}
\def\phiz{\phi^{(1)}}
\def\Bphi{B_{\phi}}
\def\Rphi{R_{\phi}}
\def\Mphi{M_{\phi}}
\def\nuphi{\nu_{\phi}}

\def\he{\hat{\epsilon}}

\def\al{\alpha}

\def\d1i{d_{i-1}}

\def\h1i{H_{i-1}}

\def\T{\mathcal{T}}
\def\Tgphi{\T_{\gamma, \alpha}}
\def\Tgphiz{\T_{\gamma, 0}}
\def\Tg{\T_{\gamma}^{\infty}}

\newcommand{\Ti}{T_{\infty}}

\newcommand{\AD}{A_{\Delta}}
\newcommand{\BD}{B_{\Delta}}
\newcommand{\mgood}{m-\text{good} }
\newcommand{\mgoodd}{m-\text{good}}
\newcommand{\mbad}{m-\text{bad} }
\newcommand{\typical}{$(m,r, c, C, \lambda)$-{typical} }

\allowdisplaybreaks
\begin{document}

\title{Scaling limit of linearly edge-reinforced random walks on critical Galton-Watson trees}
\author{George Andriopoulos\thanks{NYU-ECNU Institute of Mathematical Sciences, NYU Shanghai, Shanghai, 200062, China. Email: {ga73@nyu.edu}} \hspace{1cm} Eleanor Archer\thanks{Department of Mathematical Sciences, Tel Aviv University, Tel Aviv 69978, Israel. Email: {eleanora@mail.tau.ac.il}.}}
\maketitle
\begin{abstract}
    We prove an invariance principle for linearly edge reinforced random walks on $\gamma$-stable critical Galton-Watson trees, where $\gamma \in (1,2]$ and where the edge joining $x$ to its parent has rescaled initial weight $d(O, x)^{\alpha}$ for some $\alpha \leq 1$. This corresponds to the recurrent regime of initial weights. We then establish fine asymptotics for the limit process. In the transient regime, we also give an upper bound on the random walk displacement in the discrete setting, showing that the edge reinforced random walk never has positive speed, even when the initial edge weights are strongly biased away from the root.
\end{abstract}



\section{Introduction}


Linearly edge-reinforced random walks (LERRW) are a classical model of self-interacting processes introduced by Coppersmith and Diaconis in 1986 \cite{CoppersmithDi87} (for a generalisation, we refer to the recent work of \cite{bacallado2021edge}). Given a rooted graph $G$ endowed with initial edge weights, and a \textit{reinforcement parameter} $\Delta > 0$, the model is defined as follows. The process $(X_n)_{n \geq 0}$ on $G$ is started at the root and given $X_n = v$, the next edge traversed by $X$ is chosen from the edges incident to $v$ with probability proportional to their weights. After $X$ has crossed the chosen edge, its weight is subsequently increased by $\Delta$ and the process repeats with the updated edge weights.


The non-Markovian nature of the LERRW makes it difficult to analyse. However, a remarkable result of Diaconis and Freedman \cite{DiaconisFreedman1980} (for the recurrent case) and then Merkl and Rolles \cite{merklrolles2007} (extension to the transient case) shows that the LERRW can be represented as a random walk in random environment (RWRE). This makes LERRW more tractable and for example leads to applications in Bayesian statistics \cite{diaconis2006bayes, bacallado2009bayesian, bacallado2011bayesian, bacallado2016bayesian}. On trees this RWRE representation can be understood by noticing that the weights of the edges incident to each of the vertices evolve according to a P\'olya urn model, independently for each vertex, see \cite{robin1988phase}. We will make this connection precise in Section \ref{sctn:LERRW and RWRE connection trees}.




The aim of this paper is to study LERRW on critical Galton-Watson trees by constructing its scaling limit for a range of initial weights and obtaining almost sure asymptotics for the limiting diffusion. Since critical Galton-Watson trees have a rich geometry, in particular having fractal properties, unbounded degrees, and non-uniform volume growth, we hope that these results are interesting in their own right as well as offering insight into the behaviour of LERRW on related critical random graphs such as uniform planar triangulations, high dimensional uniform spanning trees and critical percolation clusters.

Our first result, an invariance principle, extends an earlier work of Lupu, Sabot and Tarr\`es, who constructed the scaling limit of LERRW on $\Z$ as a diffusion on $\R$ \cite{lupu2018scaling}. In the Galton-Watson tree setting, in the special case where the initial weights are constant and identical for all edges and the critical Galton-Watson tree has finite variance, the scaling limit was previously constructed in \cite{andriopoulos2021invariance}.


We assume that the underlying Galton-Watson tree is critical with offspring distribution in the domain of attraction of a $\gamma$-stable law for some $\gamma \in (1,2]$. In the case $\gamma < 2$ this immediately entails that the offspring distribution $\xi$ satisfies $\xi ([k, \infty)) =k^{-\gamma} L(k)$ for some slowly-varying function $L$. If such a $\xi$ is fixed and $T_n$ is a $\xi$-Galton-Watson tree with $n$ vertices endowed with uniform mass measure, it is well-known that there exists a sequence $(a_n)_{n=1}^{\infty}$ such that $n^{-1}a_nT_n$ converges in distribution to the \textit{stable tree} of Le Gall and Le Jan \cite{galljan1998branching} (see also \cite[Chapter 1]{duquesne2002asterisque}). We denote the compact stable tree by $\Tgc$ and its root by $O$. In the case $\gamma = 2$ this is simply Aldous' Continuum Random Tree (see \cite{aldous1991tree} and \cite{aldous1993continuum}). We give the formal constructions in Section \ref{sctn:stable trees bckgrnd}.  


To define our LERRW model, we fix a reinforcement parameter $\Delta > 0$, label the root of our Galton-Watson tree $O_n$ and consider a class of initial weights $(\alpha_e^{(n)})_{e \in E(T_n)}$ parametrised by $\alpha \le 1$ whereby, if $\overleftarrow{x}$ denotes the parent vertex of $x$, 
\begin{equation}\label{scheme2}
\alpha^{(n)}_{\{\overleftarrow{x},x\}}=
(n a_n^{-1})^{1-\alpha} (d_n(O_n, x)+n a_n^{-1})^{\alpha}
\end{equation}
where $d_n$ denotes the graph distance on $T_n$, cf. \cite[Section 2.3]{takei2020almost}. Note that $na_n^{-1}$ is the natural scale for branch lengths in $\Tgc$; the addition of of this term in is just to ensure that the weights and inverse weights are integrable at $0$. Our first result is an invariance principle for this LERRW process.

\begin{theorem}\label{thm:scaling lim intro}
Let $X^{(n)}$ be a discrete-time LERRW on $T_n$, started at $O_n$, with initial weights as in \eqref{scheme2} and reinforcement parameter $\Delta$. Denote its law by $P_{O_n}^{(\alpha^{(n)})}$. For every $\alpha \leq 1$, there exists a stochastic process $X$ defined on $\Tgc$ and started at $O$, with law $P_{O}$, such that the following convergence holds jointly with respect to the Gromov-Hausdorff topology, and the topology of weak convergence of measures on the \cadlag path space on $T_n$ (itself endowed with the Skorohod-$J_1$ topology):
\[
\left(a_nn^{-1} T_n, P_{O_n}^{(\alpha^{(n)})}\left(\left(n^{-1} a_n  X^{(n)}_{\lfloor 2n^2 a_n^{-1} t \rfloor}\right)_{t\ge 0}\in \cdot \right)\right) \xrightarrow{(d)} \left(\Tgc,  P_{O}\left(\left(X_t\right)_{t\ge 0}\in \cdot\right)\right). 
\]
\end{theorem}

The law of $X$ will be defined explicitly in Section \ref{sctn:scaling limit candidate}. 

To establish scaling limits for tree-indexed random walks, one generally needs tighter control on the input laws compared to the $\Z^d$ case, cf. \cite[Theorem 2]{janson2005convergence} and \cite[Theorem 1]{marzouk2018snake} which require finiteness of higher moments in order to prove convergence of tree-indexed random walks to tree-indexed Brownian motion (known as snake processes). This is because all the extra branches present in the tree present more opportunities for large displacements of the random walk, leading to larger fluctuations without the higher moment assumption. However, we will see in the LERRW setting that the random environment becomes increasingly concentrated  as $n \to \infty$ without making any higher moment assumptions.

The analogous result previously proved by Lupu, Sabot and Tarr\`es for the LERRW in dimension one \cite{lupu2018scaling} used a high-level representation of a LERRW as a RWRE that was established in \cite{davis2002}. They encode the corresponding random environment in a random walk process and show that this converges to a time-changed Brownian motion on $\R$ under the appropriate rescaling. The law of the limiting LERRW diffusion is analogously encoded by this Brownian motion. We will use the same strategy to prove Theorem \ref{thm:scaling lim intro}. In contrast to \cite{lupu2018scaling} and \cite{andriopoulos2021invariance}, we directly use the observation of Pemantle that the random environment can be represented by independent Dirichlet random variables at each vertex \cite{robin1988phase}, rather than the (more complicated) mixing measure obtained in \cite{davis2002}, which makes our proof more elementary.

In the second half of this paper, we obtain asymptotics for the limiting diffusion appearing in Theorem \ref{thm:scaling lim intro}. To understand the long-time asymptotics of LERRW and its scaling limit it is more natural to study these processes on Galton-Watson trees conditioned to survive and their continuum counterparts, rather than on compact trees. In the discrete setting Kesten \cite{kesten1986limit} showed that this conditioning can be achieved by conditioning a Galton-Watson tree to have a single branch that survives to infinity, known as the backbone, and attaching finite Galton-Watson trees to the backbone. We denote this tree by $T_{\infty}$. This construction also has a natural analogue in the continuum known as sin-trees. These were constructed by Aldous in the case $\gamma = 2$ \cite[Section 2.5]{aldous1991continuum} and Duquesne in the case $\gamma < 2$ \cite{DuqSinTree}. We denote the infinite stable tree by $\Tg$. We give the formal constructions in Sections \ref{sctn:infinite trees bckgrnd} and \ref{sctn:Tgamma props}.

It seems intuitively clear that increasing $\Delta$ should make the LERRW ``more recurrent'' whereas increasing $\alpha$ should make it ``more transient''. Since the infinite tree $T_{\infty}$ contains a unique path to infinity, it follows directly from the corresponding result of Takeshima \cite{takeshima2000behavior} (see also \cite{akahori2019phase}) for LERRW on the infinite half-line that the LERRW is almost surely recurrent if $\sum_{n \geq 1} n^{-\alpha} = \infty$, and almost surely transient otherwise. However, although the value of $\Delta$ cannot change the transience and recurrence properties of $X$, we will see in our next theorems that both $\alpha$ and $\Delta$ affect the time it takes for $X$ to escape a ball. In the case $\alpha < 1$, which we refer to as ``strongly recurrent'', the reinforcement gives an exponential bias towards the root, so that $X$ only moves away from the root at logarithmic speed. In what follows, we denote by $\mathbf{P}$ the probability measure for Kesten's tree $\Ti$ and for the infinite stable tree $\Tg$.

\begin{theorem}\label{thm:almost sure limsup}[Strongly recurrent regime $\alpha < 1$].
Let $(X_t)_{t \geq 0}$ be the limiting diffusion of \cref{thm:scaling lim intro} on the infinite stable tree $\Tg$ when $\alpha < 1$. If $\Delta>0$, we have that $\mathbf{P}\times P_{O}$-almost surely, 
\[
\limsup_{t \rightarrow \infty} \frac{d(O, X_t)}{(\log t)^{\frac{1}{1-\alpha}}} = \left(\frac{1-\alpha}{\Delta}\right)^{\frac{1}{1-\alpha}},
\]
where by $d$ we denote the tree metric on $\Tg$.
\end{theorem}

\begin{remark}
Takei \cite{takei2020almost} also proved similar results for the discrete-time LERRW on $\Z$. Lupu, Sabot and Tarr\`es also proved a result analogous to Theorem \ref{thm:almost sure limsup} in the case $\alpha=0$ for the limiting diffusion on $\R$ \cite[Proposition 4.9]{lupu2018scaling}; this corresponds to taking the initial occupation profile function $L_0$ equal to 1 everywhere, except possibly a compact interval. Our long-time asymptotics results cover a larger class of initial occupation profiles, and we additionally have to take care of fluctuations in the underlying tree $\Tg$.
\end{remark}

In the critical case $\alpha = 1$, the slow-down effect from the reinforcement is almost balanced by the increasing initial edge weights, so we lose the exponential attraction towards the root and no longer see the slow movement. Instead the attraction is of polynomial order and this is reflected in the results of \cref{thm:almost sure limsup critical}. Because the reinforcement effect and the volumes of balls in the tree both grow on a polynomial scale, we see that there are two regimes depending on which one has the dominant effect.

\begin{theorem}\label{thm:almost sure limsup critical}[Critical regime $\alpha = 1$].
Let $(X_t)_{t \geq 0}$ be the limiting diffusion of \cref{thm:scaling lim intro} on the infinite stable tree $\Tg$ when $\alpha =1$.
\begin{enumerate}[(i)]
\item If $\Delta \leq \frac{2 \gamma-1}{\gamma-1}$, then for any $\beta > \frac{1}{2}$, we have that $\mathbf{P}\times P_{O}$-almost surely, 
\[
0 < \limsup_{t \rightarrow \infty} \frac{d(O, X_t)}{t^{\frac{\gamma - 1}{2\gamma - 1}}e^{-(\log t)^{\beta}}} \qquad \text{and} \qquad \limsup_{t \rightarrow \infty} \frac{d(O, X_t)}{t^{\frac{\gamma - 1}{2\gamma - 1}}e^{(\log t)^{\beta}}} < \infty.
\]
\item If $\Delta > \frac{2 \gamma-1}{\gamma-1}$, then for any $\beta > \frac{1}{2}$, we have that $\mathbf{P}\times P_{O}$-almost surely, 
\[
0 < \limsup_{t \rightarrow \infty} \frac{d(O, X_t)}{t^{1/\Delta}e^{-(\log t)^{\beta}}}  \qquad \text{and} \qquad  \limsup_{t \rightarrow \infty} \frac{d(O, X_t)}{t^{1/\Delta}e^{(\log t)^{\beta}}} < \infty.
\]
\end{enumerate}
\end{theorem}

In the transient regime when $\alpha > 1$, the resistance techniques used to prove Theorem \ref{thm:scaling lim intro} no longer provide the machinery to identify the LERRW scaling limit since the rescaled resistances degenerate in this regime. However, we can work in the discrete setting and obtain an upper bound on the displacement of the random walk. 

\begin{theorem}\label{thm:almost sure limsup transient}[Transient regime $\alpha > 1$].
Let $(Y_n)_{n\ge 0}$ be a discrete-time LERRW on Kesten's tree $\Ti$ with offspring distribution in the domain of attraction of a $\gamma$-stable law, with initial weights $\alpha_{\{\overleftarrow{x},x\}}= d_{\Ti}(O_{\infty}, x)^{\alpha}$. Then, almost surely, for any $A > \frac{1}{2}$,
\[
d_{\Ti}(O_{\infty}, Y_n) \leq n^{\frac{2\gamma - 1}{2\gamma}} e^{(\log n)^{A}},
\]
for all sufficiently large $n$. Consequently, $Y$ does not have positive speed.
\end{theorem}

It is somewhat more delicate to obtain lower bounds on the displacement since this requires uniform control of the environment. However, we believe that it should also be possible to get a comparable lower bound using discrete analogues of the arguments used to obtain this control in the recurrent case.

Despite the bias of our initial weights away from the root, the result of Theorem \ref{thm:almost sure limsup transient} contrasts strongly to those obtained on $b$-regular trees with constant initial weights by Collevecchio \cite{collevecchio2006limit}, who proved that the speed is positive when $b$ is large enough or under an appropriate moment condition for the LERRW. This was later extended to all $b \geq 2$ by Aidekon \cite{Aidekon2008TransientRW}. In our case, Theorem \ref{thm:almost sure limsup transient} holds because the LERRW is slowed down by the time spent in dead ends of the tree; this effect is particularly pronounced because of the fractal properties of the tree (i.e. many dead ends en route to and from other dead ends). By contrast, in regular trees there are no dead ends, and exponential volume growth, which gives a much stronger bias away from the root.

In the same spirit, due to the uniqueness of the path to infinity in $\Ti$, the LERRW on $\Ti$ can be viewed as a LERRW on $\Z^+$ slowed down by extra excursions into subtrees hanging off this path. This can be quantified: on $\Z^+$, with the same initial weights, it is possible to show that $\sup_{m \leq n} |Y_m| \geq n^{\frac{1}{2} - o(1)}$ (see Remark \ref{remark:result on Z}). Since $\gamma > 1$, the LERRW is really slower on the trees we consider (note however that it is not true that $\Ti \to \Z^+$ in any sense as $\gamma \downarrow 1$, in fact due to the heavy tails, the Galton-Watson trees undergo a condensation phenomenon, e.g. see \cite{KortCondensation}).

Here we briefly record some notation that we will use throughout the paper.

\begin{center}
\begin{tabular}{ |c|c|c|c| } 
\hline
$T_n$ & critical Galton-Watson tree with offspring distribution satisfying \eqref{eqn:dom of att def} \\ 
$d_n$ & graph distance on $T_n$ \\
$\mu_n$ & uniform probability measure on the vertices of $T_n$\\
$R_n$ & distorted graph distance on $T_n$ \eqref{eqn:discrete V, res, mu def}\\
$\nu_n$ & distorted measure on the vertices of $T_n$ \eqref{eqn:discrete V, res, mu def}\\
$O_n$ & root of $T_n$\\
$\Ti$ & critical Galton-Watson tree conditioned to survive, with offspring distribution satisfying \eqref{eqn:dom of att def} \\
$d_{\Ti}$ & graph distance on $\Ti$\\
$R$ & distorted graph distance on $\Ti$\\
$\nu$ & distorted measure on the vertices of $\Ti$\\
$O_{\infty}$ & root of $\Ti$\\
$\Tgc$ & compact stable tree \\
$\Tg$ & infinite stable tree \\
$\mu$ & canonical measure on $\Tgc$ or $\Tg$ \eqref{eqn:natmeas2}\\
$\nuphi$ & distorted measure on $\Tgc$ or $\Tg \eqref{eqn:distorted meas def}$ \\
$d$ & canonical metric on $\Tgc$ or $\Tg$ \eqref{stablemet}\\
$\Rphi$ & distorted metric on $\Tgc$ or $\Tg$
\eqref{eqn:distorted res def}\\
$O$ & root of $\Tgc$ or $\Tg$\\
$a_n$ & scaling factor for $T_n: n^{-1}a_n T_n \overset{(d)}{\to}\Tgc$ \\
$\mathbf{P}$ & probability measure for the mm-spaces $T_n, \Ti, \Tgc, \Tg$ \\
$\alpha^{(n)}_e$ & initial weight for LERRW for $e \in E(T_n)$ \\
$\Delta$ & reinforcement parameter for LERRW on $T_n$ \\
$\mathbf{b}^{(n)}_v$ & parameter for RWDE at $v \in T_n$ \\
$\mathbb{P}^{(\mathbf{b})}$ & law of Dirichlet random environment on $T_n$ with (vector) parameter $\mathbf{b}$\\
$X^{(n)}$ & discrete-time LERRW on $T_n$ \\
$P_{O_n}^{(\alpha)}$ & quenched law of $X^{(n)}$ with initial weights $\alpha$ \\
$Z^{(n)}$ & continuous-time constant speed RWDE on $T_n$ \\
$\tilde{P}_{O_n,\omega}$ & quenched law of $Z^{(n)}$ \\
$\tilde{P}_{O_n}^{(\mathbf{b})}$ & annealed law of $Z^{(n)}$ \\
$X$ & scaling limit of $X^{(n)}$ as in Theorem  \ref{thm:scaling lim intro}\\
$\tilde{P}_{O,\phi}$ & quenched law of $X$ (over the environment) \\
${P}_{O}$ & annealed law of $X$ (over the environment) \\
$Y$ & discrete-time LERRW on $\Ti$\\
$\dbt(s,t)$ & distance defined on $\R$ equal to $\begin{cases}
\dt(0,s)^{1-\alpha} + \dt(0,t)^{1-\alpha} - 2\dt(0, s\wedge t)^{1-\alpha}, &\text{ if } \alpha < 1, \\
\log (\dt(0,s)+1) + \log (\dt(0,t)+1) - 2\log (\dt(0, s \wedge t)+1), &\text{ if } \alpha = 1,
\end{cases}$\\
$\phib$ & Gaussian process defined such that $\E{|\phib(s) - \phib(t)|^2} = \dbt(s,t)$ \\
$\Pb$ & law of $\phib$ \\
\hline
\end{tabular}
\end{center}

\subsection{Organisation}
The paper is organised as follows. In \cref{sctn:tree background} we give background on critical Galton-Watson trees and the topologies used in this paper. In \cref{sctnP:LERRW RWRE connection} we make the connection between LERRW and RWRE precise, and construct a candidate for the scaling limit of LERRW. In \cref{sctn:scaling lim} we prove that the resistance metrics and measures characterising the associated RWRE converge to those characterising the aforementioned limit candidate, which proves \cref{thm:scaling lim intro}. The second half of the paper is devoted to the proofs of Theorems \ref{thm:almost sure limsup} - \ref{thm:almost sure limsup transient}: in \cref{sctn:Tgamma props} we establish some properties of $\Tg$ and its associated Gaussian potential, then in Sections \ref{sctn:limsup asymp recurrent} - \ref{sctn:trans results} we respectively prove Theorems \ref{thm:almost sure limsup} - \ref{thm:almost sure limsup transient}.

\subsection{Acknowledgements}
We would like to thank David Croydon for introducing us to this topic, and Andrea Collevecchio for useful comments on the manuscript. The research of EA was supported by JSPS and ERC starting grant 676970 RANDGEOM.

\section{Preliminaries}\label{sctn:tree background}
\subsection{Critical Galton-Watson trees}
To prove the convergence we will view our Galton-Watson trees as plane trees using the following Ulam-Harris labelling notation for discrete trees. To define these, we follow \cite{Neveu} and first introduce the set
\[
\mathcal{U}=\bigcup_{n=0}^{\infty} {\N}^n.
\]
By convention, ${\N}^0=\{ \emptyset \}$. If $u=(u_1, \ldots, u_n)$ and $v=(v_1, \ldots, v_m) \in \mathcal{U}$, we let $uv= (u_1, \ldots, u_n, v_1, \ldots, v_m)$ denote the concatenation of $u$ and $v$.

\begin{defn} A plane tree $T$ is a finite subset of $\mathcal{U}$ such that
\begin{enumerate}[(i)]
\item $\emptyset \in T$.
\item If $v \in T$ and $v=uj$ for some $j \in \N$, then $u \in T$.
\item For every $u \in T$, there exists a number $k_u(T) \geq 0$ such that $uj \in T$ if and only if $1 \leq j \leq k_u(T)$.
\end{enumerate}
\end{defn}

We let $\mathbb{T}$ denote the set of all plane trees. If $T \in \mathbb{T}$ and $u \in T$ we also let $\tau_u(T)=\{v \in \mathcal{U}: uv \in T \}$ denote the subtree emanating from $u$. In this paper we are interested in random plane trees, more specifically \textit{Galton-Watson} trees. To define these, first let $\xi$ be a probability measure on $\Z^{\geq 0}$. We will refer to $\xi$ as the \textit{offspring distribution}.

\begin{defn}\label{def:GW tree}
A Galton-Watson tree with offspring distribution $\xi$ is a plane tree $T$ with law $\Pb_{\xi}$ satisfying the following properties:
\begin{enumerate}[(i)]
\item $\prstart{k_{\emptyset}=j}{\xi}=\xi(j)$ for all $j \in \Z^{\geq 0}$.
\item For every $j \geq 1$ with $\xi(j)>0$, the shifted trees $\tau_1(T), \ldots, \tau_j(T)$ are independent under the conditional probability $\prcond{\cdot}{k_{\emptyset}=j}{\xi}$, with law $\bP_{\xi}$.
\end{enumerate}
\end{defn}

In other words, a Galton-Watson tree with offspring distribution $\xi$ is a branching process with a single root $\emptyset$, where the trees emanating from each vertex are independently distributed according to $\bP_{\xi}$. It is shown in \cite[Section 3]{Neveu} that for any probability measure $\xi$ on $\Z^{\geq 0}$, there is a unique probability measure $\bP_{\xi}$ on $\mathbbm{T}$ satisfying the above two properties. 

We will restrict to Galton-Watson trees with a critical aperiodic offspring distribution in the domain of attraction of a $\gamma$-stable law for some $\gamma \in (1,2]$, by which we mean that $\sum_{k=0}^{\infty} k \xi(k) = 1$ and there exists an increasing sequence $a_n \uparrow \infty$ such that, if $(\xi^{(i)})_{i=1}^{\infty}$ are i.i.d. copies of $\xi$, then
\begin{equation}\label{eqn:dom of att def}
\frac{\sum_{i=1}^n \xi^{(i)} - n}{a_n} \xrightarrow{(d)} Z_{\gamma},
\end{equation}
as $n \rightarrow \infty$, where $Z_{\gamma}$ is a $\gamma$-stable random variable, i.e. can be normalised so that $\E{e^{-\lambda Z_{\gamma}}} = e^{-\lambda^{\gamma}}$ for all $\lambda > 0$. In the finite variance case we always have $\gamma=2$. It is shown in \cite[Section 8.3.2]{bingham1989variation} that necessarily $a_n = n^{\frac{1}{\gamma}}L(n)$ for some slowly-varying function $L$. Equivalently, $\xi([n, \infty)) = n^{-\gamma} L(n)$ (but not necessarily with the same $L$). Throughout the paper we will make the assumption that $\gamma \in (1,2]$, and let $(a_n)_{n=1}^{\infty}$ be the sequence appearing in (\ref{eqn:dom of att def}).

A Galton-Watson tree $T$ conditioned to have $n$ vertices can be coded by a random walk $(W^{T}_m)_{0 \leq m \leq n}$ called the Lukasiewicz path; this is defined by setting $W^{T}_0 = 0$, then for $m \geq 1$ listing the vertices $u_0, u_1, \ldots, u_{n-1}$ in lexicographical order and setting $W^T_{m+1} = W^T_m + k_{u_m}(T)-1$. It is not hard to see that $W^{T}_m \geq 0$ for all $0 \leq m \leq n-1$, and $W^T_n = -1$.

A discrete tree $T$ conditioned to have $n$ vertices can be also characterised by its contour function. To introduce the contour function, we imagine that the tree is embedded in the plane in such a way that each edge has length one. Consider a particle that is placed at the root of the tree at time $i=0$ and then traverses the tree, moving continuously along the edges at unit speed (respecting the left-right order of the vertices), until all vertices have been explored and the particle has returned back to the root. Then, we denote by $C^T(i)$ the distance to the root of the position of the particle at time $i$. More specifically, letting $x_i$ denote the $i$-th visited vertex by the particle, set
\[
C^T(i)=d_T(x_0,x_i), \qquad 0\le i< 2 n.
\]
By convention, we also set $x_{2n}=x_0$, and $C^T(2n)=0$. Naturally, we extend $C^T$ by linear interpolation between integer times. Note that the particle visits every vertex apart from the root a number of times given by its degree.

\subsection{Gromov-Hausdorff topology and correspondences}\label{sctn:GHP conv background}
In this paper we will be interested in pointed Gromov-Hausdorff-Prokhorov (GHP) convergence between pointed compact metric measure (mm) spaces. Accordingly, let $\mathbb{K}^c$ denote the set of quadruples $(X,d,\mu,O)$ such that $(X,d)$ is a compact metric space, $\mu$ is a finite Borel measure of full support, and $O$ is a distinguished point, which will play the role of the root. Given two pointed mm-spaces $(X,d,\mu,O)$ and $(X',d',\mu',O')$, the pointed GHP distance between them is defined as
\begin{equation}\label{eqn:GHP def}
\dGHP((X,d,\mu,O),(X',d',\mu',O')) = \inf \left\{d_H^F (\phi(X), \phi'(X')) \vee d_P^F(\phi_* \mu, \phi_*' \mu')\vee \delta(\phi(O),\phi'(O')) \right\},
\end{equation}
where the infimum is taken over all isometric embeddings $\phi: X \rightarrow F$, $\phi': X' \rightarrow F$ into some common metric space $(F, \delta)$, where $d_H^F$ denotes the Hausdorff distance between two subsets of $F$ and where $d_P$ denotes the Prokhorov distance between finite Borel measures on $F$. For a definition of the latter, see \cite[Chapter 1]{billingsley2013convergence}. If the second term is omitted from the right-hand side of \eqref{eqn:GHP def}, this is known simply as the pointed Gromov-Hausdorff (GH) distance, and is denoted by $d_{GH}$. We say that two spaces $(X,d,\mu,O)$ and $(X',d',\mu',O')$ in $\mathbb{K}^c$ are equivalent if there exists a measure and a root-preserving isometry between them. It is well-known, see \cite[Theorem 2.3]{abraham2013note}, that the pointed GHP distance defines a metric on the space of equivalence classes of $\mathbb{K}^c$, and that this is a Polish metric space.

To prove Theorem \ref{thm:scaling lim intro}, we will use the helpful notion of \textit{correspondences}. A correspondence $\mathcal{R}$ between two metric spaces $(M,R)$ and $(M',R')$ is a subset of $M \times M'$ such that for every $x \in M$, there exists $y \in M'$ with $(x,y) \in \mathcal{R}$, and similarly for every $y \in M'$, there exists $x \in M$ with $(x,y) \in \mathcal{R}$. 
%
It is straightforward to show (e.g. see \cite[Theorem 7.3.25]{burago2001course}) that
\begin{align*}
\dGH{(M,R,O), (M',R',O')} = \frac{1}{2} \inf_{\mathcal{R}} \sup_{(x,x'), (y, y') \in \mathcal{R}} |R(x,y) - R'(x',y')|,
\end{align*}
where the infimum is taken over all correspondences $\mathcal{R}$ between $(M,R)$ and $(M',R')$ that contain the pair of distinguished points $(O,O')$. The quantity on the right-hand side above is known as the distortion of $\mathcal{R}$, and is denoted by $\textsf{dis} (\mathcal{R})$. 

In this paper, we will prove GHP convergence by first proving GH convergence using correspondences, and then prove Prokhorov convergence between the measures on the canonical embedding induced by the correspondence. 

\subsection{Stable trees} \label{sctn:stable trees bckgrnd}

If $T_n$ denotes a discrete Galton-Watson tree conditioned to have $n$ vertices with critical aperiodic offspring distribution $\xi$ in the domain of attraction of a $\gamma$-stable law, endowed with the graph metric $d_n$, it is well-known that it admits a compact metric space scaling limit, known as the \textit{$\gamma$-stable tree}. We denote this by $\T_{\gamma}^c$. More precisely, it is shown in \cite[Theorem 3.1]{duq2003limit} that
\begin{equation}\label{eqn:stable tree scaling limit def}
n^{-1} a_n T_n \rightarrow \T_{\gamma}^c
\end{equation}
in the GH topology as $n \rightarrow \infty$, where $a_n$ is defined as in (\ref{eqn:dom of att def}). The limiting space $\T_{\gamma}^c$ can be formally defined from a $\gamma$-stable \Levy excursion which plays the role of a continuum Lukasiewicz path (e.g. see \cite{duquesne2005probabilistic} for details). Given a spectrally positive $\gamma$-stable \Levy excursion $X^{(\gamma)}$, we define the height function $H^{(\gamma)}$ to be the continuous modification of the process
\[
H^{(\gamma)}(t)=\lim_{\epsilon\to 0} \frac{1}{\epsilon} \int_0^t \mathbbm{1} \{X_s^{(\gamma)}<\inf_{r\in [s,t]} X_r^{(\gamma)}+\epsilon \} ds.
\]
The limit above exists in probability, see \cite[Lemma 1.1.3]{duquesne2002asterisque}. 

For $s, t\in [0,1]$, let $m_{H^{(\gamma)}}(s,t)=\inf_{r\in [s\wedge t,s\vee t]} H^{(\gamma)}(r)$ and $d: [0,1]\times [0,1]\to \mathbb{R}_{+}$ be defined by
\begin{equation} \label{stablemet}
\dt(s,t)=H^{(\gamma)}(s)+H^{(\gamma)}(t)-2 m_{H^{(\gamma)}}(s,t).
\end{equation}
It is obvious that $\dt$ is symmetric and satisfies the triangle inequality. One can introduce the equivalence relation $s\sim t$ if and only if $H^{(\gamma)}(s)=H^{(\gamma)}(t)=m_{H^{(\gamma)}}(s,t)$, which is equivalent to $\dt(s,t)=0$. The quotient space $([0,1]/\sim,\dt)$ is the $\gamma$-stable tree $\mathcal{T}_{\gamma}^c$, which can be proven to be almost surely a compact metric space \cite[Theorem 2.1]{duquesne2005probabilistic}. 
In addition, this construction provides a natural way to define a canonical (non-atomic) probability measure associated with it, $\mu$, which has full support.  
Denote by $p_{H^{(\gamma)}}: [0,1]\to \mathcal{T}_{\gamma}^c$ the canonical projection. For every $A\in \mathcal{B}(\mathcal{T}^c_{\gamma})$, we let
\begin{equation}\label{eqn:natmeas2}
\mu(A)=\ell(\{t\in [0,1]: p_{H^{(\gamma)}}(t)\in A\})
\end{equation}
denote the image measure on $\mathcal{T}^c_{\gamma}$ of Lebesgue measure $\ell$ on $\mathbb{R}$ by the canonical projection $p_{H^{(\gamma)}}$. 

 We briefly outline how to prove \eqref{eqn:stable tree scaling limit def} since we will use a similar strategy in Section \ref{sctn:scaling lim}. Using Skorohod's Representation theorem, we can assume that we are working on a probability space under which the contour function of $T_n$, which is normalised by setting:
\[
C^{(n)}(t)=\left(n^{-1} a_n C^n(2 n t); 0\le t\le 1\right)
\]
where $C^n$ is the contour function of $T_n$, converges almost surely (with respect to the uniform norm) to the height function constructed by a spectrally positive $\gamma$-stable \Levy excursion. The convergence in distribution 
\begin{equation} \label{contourconv}
C^{(n)}\xrightarrow{(d)} H^{(\gamma)}
\end{equation}
was originally shown by Duquesne in \cite[Theorem 3.1]{duquesne2003cont}. We construct the related correspondence
\begin{equation} \label{corresp}
\mathcal{R}_n=\left\{(x_{\lfloor 2 n t\rfloor},p_{H^{(\gamma)}}(t)); 0\le t\le 1\right\},
\end{equation}
where $x_i$ is the $i$-th visited vertex in the exploration of the outline of $T_n$ and $p_{H^{(\gamma)}}(t)$ is the equivalence class of $t$ with respect to the relation $\sim$. It is straightforward to show that the distortion of this correspondence between $n^{-1} a_n T_n$ and $\Tgc$ is upper bounded by $2 ||C^{(n)}-H^{(\gamma)}||$, where $||C^{(n)}-H^{(\gamma)}||$ stands for the uniform norm of $C^{(n)}-H^{(\gamma)}$. Since \eqref{contourconv} holds, the convergence of the metric spaces follows. In fact, if $\mu_n$ denotes the uniform probability measure on its vertices, it is the case that
\begin{equation} \label{eqn:stable tree scaling limit def 1}
(T_n,n^{-1} a_n d_n,\mu_n)\to (\mathcal{T}^c_{\gamma},d,\mu)
\end{equation}
in distribution with respect to the GHP distance between compact mm-spaces, see \cite[Theorem 4.2]{gall2006trees}, which is a corollary of the result originally proved in \cite{duq2003limit}. Although the uniform probability measure $\mu_n$ on the vertices of $T_n$ was not considered in \cite{duq2003limit}, it is not difficult to extend the result in \eqref{eqn:stable tree scaling limit def} to include it since the work regarding the convergence in \eqref{eqn:stable tree scaling limit def 1} has already been done.

\subsection{Infinite critical trees}\label{sctn:infinite trees bckgrnd}
A critical Galton-Watson tree is almost surely finite, so to study the long-time asymptotics of LERRW on critical trees and more specifically to prove Theorems \ref{thm:almost sure limsup} - \ref{thm:almost sure limsup transient}, we will instead consider the model on a Galton-Watson tree \textit{conditioned to survive.} In the discrete setting such a model is naturally described by \textit{Kesten's tree}, denoted $T_{\infty}$ and defined as follows.

\begin{defn}\label{def:Kesten's tree}\cite{KestenIICtree}.
Let $\xi$ be a critical offspring distribution, and define its size-biased version $\xi^*$ by
\[
\xi^*(n) = {n \xi (n)}.
\]
The \textbf{Kesten's tree} $T_{\infty}$ associated to the probability distribution $\xi$ is a two-type Galton-Watson tree distributed as follows:
\begin{itemize}
\item Individuals are either normal or special.
\item The root of $T_{\infty}$ is special.
\item A normal individual produces only normal individuals according to $\xi$.
\item A special individual produces individuals according to the size-biased distribution $\xi^*$. Of these, one of them is chosen uniformly at random to be special, and the rest are normal.
\end{itemize}
\end{defn}
Almost surely, the special vertices form a unique one-ended infinite backbone of $T_{\infty}$. Aldous in \cite{AldousFringeSinTree} coined the term \textit{sin-trees} for such trees, since they have a single infinite spine. Although a critical Galton-Watson tree is almost surely finite, Kesten \cite[Lemma 1.14]{KestenIICtree} showed that $T_{\infty}$ arises as the local limit of a critical Galton-Watson tree with offspring distribution $\xi$ as its height goes to infinity.

Kesten's construction has been imitated in the continuum by Duquesne in \cite{DuqSinTree}, who constructs continuum sin-trees and shows that these arise as the appropriate local limit of compact continuum trees conditioned on being large. Duquesne's construction involves defining two height functions from two independent \Levy processes in the same way as done for the compact case. In the stable case, we denote the infinite stable sin-tree by $\Tg$. It is also possible to show that $\Tg$ is the scaling limit of a discrete tree constructed from a critical aperiodic $\gamma$-stable offspring distribution as constructed in Definition \ref{def:Kesten's tree} (e.g. by following a similar strategy to \cite{archer2020infinite}). As in the discrete case, $\Tg$ has a single infinite path to infinity, to which compact stable trees are grafted.

\section{Connection between LERRW and a RWRE on trees}\label{sctnP:LERRW RWRE connection}

The fine mesh limit of the LERRW in dimension one, which was called linearly reinforced motion (LRM), was introduced in \cite{lupu2018scaling}. In fact, the authors construct the continuous space limit of the vertex-reinforced jump process (VRJP) out of a convergent Bass-Burdzy flow \cite{bass1999flow}. To obtain LRM as the continuous space limit of the LERRW, they then use the close relation between the VRJP and the LERRW that was established in \cite{sabot2015edge}, namely that the LERRW can be represented in terms of a VRJP with independent gamma conductances. The principles used in \cite{lupu2018scaling} are broadly the same as those used in this paper. 

In the setting of \cite{lupu2018scaling} an appropriate potential related to the VRJP converges. This yields a characterisation of the limiting LRM as a diffusion in random potential that contains a Wiener term and a drift (the motion drifts towards the places it has already visited many times), or equivalently describes the LRM as a scale-transformed and time-changed diffusion in a random environment, cf. \cite{sinai1982limit, brox1986one, seignourel2000discrete, carmona1997mean, pacheco2016sinai}. 

In our setting, we work directly on the trees and view them as electrical networks equipped with a so-called resistance metric and a measure that we will specify in Section \ref{sctn:RWRE conductances}. The latter provide the natural scale functions and speed measures of a RWRE associated with the LERRW in a representation of Pemantle, see Section \ref{sctn:LERRW and RWRE connection trees}. Then we are able to use the main result of \cite{croydon2016scaling} to yield convergence of the RWRE as a consequence of the convergence of the resistance metric and the speed measure. Since the resistance metric and the speed measure are expressed in terms of the potential of the RWRE, their convergence can be deduced from the convergence of the aforementioned potential, see Sections \ref{sctn:potential convergence} and \ref{sctn:GHP convergence}.

The limiting resistance metric and speed measure on the limiting stable tree are distortions of its canonical metric and uniform mass measure. These distortions are expressed in terms of an exponential that includes a tree-indexed Gaussian term and a drift. The Gaussian term corresponds to a tree-indexed process that is essentially a Brownian snake on the stable tree. This yields a characterisation of the limiting diffusion $X$ as a diffusion in random potential on the stable tree, see Section \ref{sctn:scaling limit candidate}. We note here that the notion of a ``scale change'' on general separable real trees was formalised in \cite{athreya2013brownian}.

\subsection{Random walk in a random environment}\label{sctn:RWRE conductances}

In order to study a LERRW on $T_n$ we will use a representation of a LERRW as a RWRE. We briefly introduce the formalism of a RWRE here.

In a rooted metric tree $(T,d_T,O)$, we define for $u,v\in T$ the path intervals
\[
[u,v]=\{z\in T: d_T(u,z)+d_T(z,v)=d_T(u,v)\},
\]
and
\[
(u,v)=[u,v]\setminus \{u,v\}, \quad [u,v)=[u,v]\setminus \{v\}, \quad (u,v]=[u,v]\setminus \{u\}.
\]
We say that $u$ and $v$ are connected by an edge, denoted $u\sim v$, if and only if
$
u\neq v \text{ and } [u,v]=\{u,v\}.
$
We also define a partial order on $T$ by setting $u\preceq v$ ($u$ is an ancestor of $v$) if and only if $u\in [O,v]$. We write $u \prec  v$ if $u \preceq v$ and $u \neq v$. Finally, the unique $z\in T$ for which $[O,u]\cap [O,v]=[O,z]$ is written $u\wedge v$. We call $u\wedge v$ the most recent common ancestor to $u$ and $v$. 

For a fixed tree $T$ with root $O$, a RWRE on $T$ can be constructed by assigning each edge $e \in E(T)$ a random conductance. At each time, if a random walk is currently at vertex $v$ it moves to one of the neighbours of $v$ with a probability proportional to the conductance of the edge joined to $v$. Rather than defining the full set of conductances $(c(e))_{e \in E(T)}$, we can equivalently assign a single conductance to a single special ``root edge'' $e_{\text{root}}$ and then sample a sequence of random variables $(\mathbf{W}_v)_{v \in T}$, so that 
\[
\mathbf{W}_v = (W_v^u)_{u \sim v}
\]
gives the ratios of the conductances of edges emanating from $v$. We will follow this approach as it allows us to formalise the connection to P\'olya urns outlined in the introduction.

We now assume that an environment $\omega = \left(c(e_{\text{root}}), \mathbf{W} = \left(\mathbf{W}_v\right)_{v \in T}\right)$ is fixed. For each non-root vertex $v$, let $\overleftarrow{v}$ denote the parent of $v$. Let  
\begin{align}\label{eqn:rhov def}
\rho_v(\omega)=\frac{W_{\overleftarrow{v}}^{\overleftarrow{\overleftarrow{v}}}}{W_{\overleftarrow{v}}^v},
\end{align}
which in our case represents the transition probability from $\overleftarrow{v}$ to the ancestor $\overleftarrow{\overleftarrow{v}}$ of $\overleftarrow{v}$, divided by the transition probability from $\overleftarrow{v}$ to the descendant $v$ of $\overleftarrow{v}$. 

We add a new vertex $\overleftarrow{O}$, which we call the base and attach it to the root with a new edge. This edge will play the role of $e_{\text{root}}$ and we call the new tree the planted tree. To keep the notation simple, even if the statements are expressed in terms of this new planted tree, say $T^*$, we still phrase them in terms of $T$. It follows from \eqref{eqn:rhov def} that the conductance of an edge $\{\overleftarrow{v},v\}$ is given by
\begin{equation}\label{eqn:conductance from rho}
 c(\{\overleftarrow{v},v\}) = c(\{\overleftarrow{O}, O\}) \cdot e^{- \sum_{O\prec u\preceq v} \log \rho_u(\omega)}.
\end{equation}
This motivates the definition of the potential on $T$ as 
\[
V_{\omega}(x)=
\begin{cases}
\displaystyle
\sum_{O\prec v\preceq x} \log \rho_v(\omega), &\text{ for } x\neq O,
\\
0, &\text{ for } x=O.
\end{cases}
\]
By defining the RWRE from conductances in this way, we have in fact defined it via an electrical network and can therefore take advantage of electrical network theory (e.g. see \cite[Chapter 9]{levin2017markov} for an introduction). The conductance of an edge $\{\overleftarrow{v}, v\}$ is given by \eqref{eqn:conductance from rho}, and consequently its electrical resistance is given by
\begin{equation*}
R(\{\overleftarrow{v},v\}) = c(\{\overleftarrow{v},v\})^{-1} = c(\{\overleftarrow{O}, O\})^{-1} \cdot e^{V_{\omega}(v)}.
\end{equation*}
Moreover, since our graph is a tree, the electrical resistance between any two vertices is given by
\begin{equation}\label{eqn:resistance from rho}
R(u,v) = \sum_{e \in E_{u,v}}R(e) = c(\{\overleftarrow{O}, O\})^{-1} \cdot \sum_{x \in [u, v] \setminus \{u \wedge v\} }e^{V_{\omega}(x)},
\end{equation}
where in the first sum, $E_{u,v}$ is the set of edges contained in $[u,v]$. Clearly, $R$ defines a metric on $T$ (this also follows from a more general result of Tetali \cite{Tetali}). We will take the measure
\begin{equation}\label{eqn:measure def}
\mu(x)= 
\begin{cases}
\displaystyle c(\{\overleftarrow{O}, O\}) \cdot \left( e^{-V_{\omega}(x)} \mathbbm{1}_{\{x\neq O\}}+\sum_{y \sim x: \overleftarrow{y}=x} e^{-V_{\omega}(y)} \right), &\text{ for } x\neq \overleftarrow{O},
\\
0, &\text{ for } x=\overleftarrow{O}.
\end{cases}
\end{equation}
This is the stationary reversible measure for a stochastic process associated with $(T,R,O)$. This stochastic process has generator which acts on $L^2(T,\mu)$ as follows:
\[
(\Delta f)(x)=
\frac{1}{\mu(x)} \sum_{y\in T: y\sim x} c(\{x,y\}) (f(y)-f(x))
\]
(see \cite[Section 2]{croydon2018introduction} for more details on the correspondence between graphs equipped with edge conductances and a measure, and stochastic processes). This stochastic process is a continuous-time random walk $Z$ having $\textsf{exp}(1)$ holding time at each non-root vertex, and at each jump time, the random walk traverses an edge incident to its previous state, chosen with probability proportional to the conductance of each of the available edges. At the base, $Z$ jumps from $\overleftarrow{O}$ to $O$ with rate $c(\{\overleftarrow{O},O\})/\mu(\overleftarrow{O})=\infty$, i.e. the expected holding time at $\overleftarrow{O}$ is 0. The random walk transitions from the base to the root with probability 1. Moreover, the overall rate at which $Z$ jumps from $O$ to its collection of children is $1$.

\begin{remark}
The electrical resistance between $x$ and $y$ can be alternatively defined by
\[
R(x,y)^{-1}=\inf \{\mathcal{E}(f,f): f(x)=0,f(y)=1\},
\]
where 
\[
\mathcal{E}(f,g)=\frac{1}{2} \sum_{\substack{x, y\in T: \\ x\sim y}} c(\{x,y\}) (f(x)-f(y))(g(x)-g(y))
\]
is a quadratic form on $T$. In fact, $\mathcal{E}$ is a regular Dirichlet form on $L^2(T, \mu)$, and corresponds to electrical energy. For such functions $f$ and $g$, we note that 
\[
\mathcal{E}(f,g)=-\sum_{x\in T} (\Delta f)(x) g(x) \mu(x).
\]
We will not directly use the machinery of Dirichlet forms in this paper.
\end{remark}

The RWRE that we use in this paper will be defined by a collection of Dirichlet random variables taking the role of $(\mathbf{W}_v)_{v \in T}$.

\begin{defn}[Dirichlet distribution]\label{def:Dirichlet dist}
For a finite set $I$ and positive real numbers $(b_i)_{i\in I}\in (0,+\infty)^{I}$, the Dirichlet distribution $\mathcal{D}((b_i)_{i\in I})$ on
\[
\Sigma_I=\bigg\{(u_i)_{i\in I}\in (0,1)^{|I|}: \sum_{i\in I} u_i=1\bigg\}
\]
has density
\[
\frac{\Gamma(\sum_{i\in I} b_i)}{\prod_{i\in I} \Gamma(b_i)} \prod_{i\in I} u_i^{b_i-1} \mathbbm{1}_{\Sigma_I}((u_i)_{i\in I}) \prod_{j \in I \setminus \{j_0\}} du_j,
\]
for any choice of $j_0 \in I$.  
\end{defn} 
The Dirichlet distribution can alternatively be defined as the law of a normalised Gamma vector. One can indeed check that the following lemma holds.

\begin{lemma} \label{gammadir}
\cite[Section 0.3.2]{pitman2006combinatorial}.
Let $(W_i)_{i\in I}$ be independent random variables such that, for $i\in I$,
\begin{equation} \label{eqn:Gamma weights}
W_i\sim \Gamma(b_i,1), \text{ i.e. with density } \frac{1}{\Gamma(b_i)} w^{b_i-1} e^{-w} \mathbbm{1}_{(0,+\infty)}(w) dw.
\end{equation}
Then,
\begin{equation} \label{eqn:normalised D weights}
(U_i)_{i\in I}=\frac{1}{\sum_{i\in I} W_i}\cdot (W_i)_{i\in I}\sim \mathcal{D}((b_i)_{i\in I}).
\end{equation}
Furthermore, $(U_i)_{i\in I}$ is independent of $\sum_{i\in I} W_i$.
\end{lemma}

\subsection{Representation of LERRW as a random walk in random conductances}\label{sctn:LERRW and RWRE connection trees}
 
In this section we outline the connection between LERRW, P\'olya urns and Dirichlet random variables on trees. This connection was first used in the context of (non-directed) edge reinforced random walk on trees by Pemantle \cite{robin1988phase} where, due to the absence of cycles, the process evolves independently at each vertex.

Given a tree $T$, let $\alpha=(\alpha_e)_{e\in E(T)}\in (0,+\infty)^{E(T)}$ be a sequence of positive initial weights on the edges, and let $O$ denote the root of $T$. We consider edges to be undirected, so that $e = \{x,y\} = \{y,x\}$ denotes the edge joining two vertices $x$ and $y$.

\begin{defn}\label{def:LERRW def}
The discrete-time linearly edge reinforced random walk (LERRW) on $T$ with initial weights $\alpha$ and starting at $O$ is the process on $T$ with law $P_{O}^{(\alpha)}$ defined by: $P_{O}^{(\alpha)}$-a.s., $X_0=O$ and, for all $n\ge 0$, for all edges $e\in E(T)$,
\[
P_{O}^{(\alpha)}\left(X_{n+1}=v|X_0,...,X_n\right)=\frac{N_{\{w, v\}}(n)}{\sum_{y:y \sim w} N_{\{w, y\}}(n)} \mathbbm{1}_{\{X_n=w\}},
\]
where $N_e(n)=\alpha_e+\#\{0\le k\le n-1: (X_k,X_{k+1})=e\}\Delta$. We call $\Delta > 0$ the \textbf{reinforcement parameter}. 
\end{defn} 

In other words, at time $n$ the walk jumps along a neighbouring edge $e$ chosen with probability proportional to its current weight $N_e(n)$. This weight is initially equal to $\alpha_e$ and then increases by $\Delta$ each time the edge $e$ is traversed.

Consider a vertex $v \in T$ and suppose $v$ has $\# v$ offspring in $T$. Let $e_0(v)$ denote the edge joining $v$ to its parent, and ${e_1(v)},...,{e_{\# v}(v)}$ denote the edges joining $v$ to each of its offspring. When the LERRW arrives at $v$ for the first time, it must have arrived via the parent, so the weights of these edges will be respectively given by the components of the vector
\[
(\al_{e_0(v)}+\Delta,\al_{e_1(v)},...,\al_{e_{\# v}(v)}).
\]
Moreover, since $T$ is a tree, if the LERRW exits $v$ by edge $e_i(v)$, it must also return to $v$ via $e_i(v)$, so that at the next visit to $v$ the weight of $e_i(v)$ will have increased by $2\Delta$ and all other weights will have remained the same. Moreover, this holds independently of what happens to the LERRW between its successive traversals of $e_i(v)$. The same logic applies on subsequent visits to $v$, with the weights updating each time.

It follows that the decisions of the LERRW process are ruled by independent P\'olya urns, one per vertex, where outgoing edges play the role of colours, and $\alpha_{e_i(v)}$ determines the initial number of balls of colour $i$. Since the asymptotic proportion of colours in such an urn model follows the Dirichlet distribution, conditional on which the draws are independent and identically distributed as Bernoulli random variables with success probability given by the asymptotic proportion of the drawn colour, the reinforced walk may equivalently be obtained by assigning a Dirichlet random variable $\omega_{(x,\cdot)}$ to each vertex $x$, which then defines the transition probabilities of the walk every time it is at $x$. This is the definition of the random walk in Dirichlet random environment (RWDE) that we define below. To get this equivalent representation of LERRW we have to average over all the Dirichlet random variables.

Given $v \in T$, let the positive initial weights from $v$ to its neighbouring vertices $\overleftarrow{v},v_1,...,v_{\# v}$ be denoted by the vector 
\[
\al_v=(\al_{e_0(v)},\al_{e_1(v)},...,\al_{e_{\# v}(v)}),
\]
with $\al_{e_0(v)}$ being the positive weight to the parent of $v$. 
We define the set of environments on $T$ as
\[
\Omega_T=\bigg\{\left(\mathbf{\omega}_v = (\omega_v^e)_{v\in e}\right)_{v \in T}\in \prod_{v \in T} \Sigma_{\deg v}\bigg\}, \hspace{5mm} \text{ where } \Sigma_j=\bigg\{(u_i)_{i \leq j}\in (0,1)^{j}: \sum_{i\leq j} u_i=1\bigg\} \text{ for } j \geq 1.
\]
(So for $v \in T$, $\Sigma_{\deg v}$ is the ($\deg v$)-dimensional simplex). We shall denote by $\omega$ a random environment sampled from $\Omega_T$.

\begin{defn}
For $\omega = (\omega_v)_{v \in T} \in \Omega_T$, the quenched random walk in random environment $\omega$ starting at $O$ is the Markov chain on $T$ starting at $O$ and with transition probabilities $(\omega_v)_{v \in T}$. We denote its law by $\tilde{P}_{O,\omega}$. Thus, $\tilde{P}_{O,\omega}$-a.s., $X_0=O$ and for all $n\ge 0$, for all edges $e\in E(T)$,
\[
\tilde{P}_{O,\omega}((X_n,X_{n+1})=e|X_0,...,X_n)=\omega_v^e \mathbbm{1}_{\{X_n=v\}}.
\]
\end{defn}

We define the Dirichlet distribution on $T$ with parameters $\left(\mathbf{b}_v = (b_v^e)_{v\in e}\right)_{v \in T}$ as the product distribution:
\[
\mathbb{P}^{(\mathbf{b})}=\prod_{v\in T} \mathcal{D}\left(\mathbf{b}_v\right).
\]
where each $\mathcal{D}\left(\mathbf{b}_v\right)$ is as in Definition \ref{def:Dirichlet dist}.

Now let us consider the joint law $\tilde{P}_{O}^{(\mathbf{b})}$ of $(\omega,X)$ on $\Omega_T\times V(T)^{\mathbb{N}}$ such that $\omega\sim \mathbb{P}^{(\mathbf{b})}$ and the conditional distribution of $X$ given $\omega$ is $\tilde{P}_{O,\omega}$. Its law is thus
\[
\tilde{P}_{O}^{(\mathbf{b})}\left(X\in \cdot\right)=\int \tilde{P}_{O,\omega}(X\in \cdot) \mathbb{P}^{(\mathbf{b})}(d \omega).
\]

Recall from Definition \ref{def:LERRW def} that $P_{O}^{(\mathbf{\al})}$ denotes the law of a LERRW started at $O$ with initial edge weights $\mathbf{\al} = (\al_e)_{e \in E(T)}$. The next result can be found in \cite[Section 3]{robin1988phase} and formalises the representation via P\'olya urns.

\begin{lemma}\cite[Lemma 2]{robin1988phase}.\label{lem:Dirichlet rep of LERRW}
Fix a tree $T$. If $\mathbf{\al}=(\al_e)_{e\in E(T)}$ is a collection of positive initial weights for a LERRW, then for any $k\ge 0$ and any $O=x_0, x_1,...,x_k\in T$, we have that
\[
P_{O}^{(\mathbf{\al})}(X_1=x_1,...,X_k=x_k)
= \tilde{P}_{O}^{(\mathbf{b})}(X_1=x_1,...,X_k=x_k),
\]
where $\mathbf{b}=(\mathbf{b}_v)_{v\in T}$ is a collection of vectors of positive weights such that if $v \in T$ with $\# v$ offspring,
\[
\mathbf{b}_v=\left((\al_{e_0(v)}+\Delta)/(2 \Delta),\al_{e_1(v)}/(2 \Delta),\cdot \cdot \cdot,\al_{e_{\# v}(v)}/(2 \Delta)\right).
\] 
\end{lemma}

\begin{remark}
On general (non-tree) graphs this representation of non-directed LERRW is not available since a random walk may return to a vertex $v$ via a different edge from that which it used to leave $v$. Therefore, the urn models at each vertex do not update independently of each other. However, on general graphs, non-directed LERRWs can still be seen as random walks in an explicit correlated random environment. For instance, see \cite{sabot2015edge}. Additionally, it is possible to represent a directed LERRW on general graphs as a RWDE, see for example \cite[Lemma 2]{sabot2017overview}.
\end{remark}

In this paper we want to consider a LERRW on the \textit{random} tree $T_n$. We consider the quenched law of LERRW on $T_n$, by which we mean that we first sample $T_n$ and then run a LERRW on $T_n$ started at $\overleftarrow{O_n}$ according to Definition \ref{def:LERRW def}. In order to obtain a non-trivial scaling limit, we will consider a LERRW on $T_n$ with initial weights given by \eqref{scheme2} and reinforcement parameter $\Delta$. In light of Lemma \ref{lem:Dirichlet rep of LERRW}, we will construct the scaling limit of LERRW on $T_n$ by instead constructing the scaling limit of the corresponding RWDE on $T_n$, which can be achieved by representing it as an electrical network endowed with a resistance metric and a measure as outlined in Section \ref{sctn:RWRE conductances}. This is equivalent to sampling a random environment $\omega = (c(e_{\text{root}}), \mathbf{W}^{(n)})$ and defining a resistance metric $R_n$ and a measure $\nu_n$ from $\omega$ as in \eqref{eqn:resistance from rho} and \eqref{eqn:measure def}. The edge $\{\overleftarrow{O_n}, O_n\}$ will play the role of $e_{\text{root}}$.

To sample the random environment, we therefore define the parameters
 $\mathbf{b}^{(n)}=(\mathbf{b}^{(n)}_v)_{v\in T_n}$ such that, if $v \in T_n$ has $\# v$ offspring,
\begin{equation}\label{eqn:rescaled Dirichlet parameters}
\mathbf{b}^{(n)}_v=\left((\al^{(n)}_{e_0(v)}+\Delta)/(2 \Delta), \al^{(n)}_{e_1(v)}/(2 \Delta), \ldots, \al^{(n)}_{e_{\# v}(v)}/(2 \Delta)\right).
\end{equation}
We will sample the Dirichlet distribution with parameter $\mathbf{b}_v^{(n)}$ using Lemma \ref{gammadir}. Since we are only interested in the ratios of the Dirichlet weights, we can work directly with the non-normalised Gamma weights, so we assume that our random environment $\omega_n = (c(\{\overleftarrow{O_n}, O_n\}), \mathbf{W}^{(n)})$ is obtained by sampling $(\mathbf{W}^{(n)}_v)_{v \in T_n}$ according to \eqref{eqn:Gamma weights}, and without loss of generality take $c(\{\overleftarrow{O_n}, O_n\})=1$ for simplicity.

In the setup of Section \ref{sctn:RWRE conductances}, we therefore have that
\begin{align}\label{eqn:discrete V, res, mu def}
&\rho_v(\omega_n) =\frac{W_{\overleftarrow{v}}^{\overleftarrow{\overleftarrow{v}}}}{W_{\overleftarrow{v}}^v}, 
\hspace{10mm}
V_{\omega_n}(x)=
\sum_{O_n\prec v\preceq x} \log \rho_v(\omega_n) \mathbbm{1}\{x\neq O_n, \overleftarrow{O_n}\},
\nonumber \\
&R^{\omega_n}(u,v)= \sum_{x \in [u, v] \setminus \{u \wedge v\}} e^{V_{\omega_n}(x)},
\nonumber \\ \nonumber \\
&\nu^{\omega_n}(x)=
\begin{cases}
e^{-V_{\omega_n}(x)} \mathbbm{1}\{x\neq O_n\}+ \displaystyle \sum_{y \sim x: \overleftarrow{y}=x} e^{-V_{\omega_n}(y)} &\text{ if } x\neq \overleftarrow{O_n},
\\ 
0 &\text{ if } x=\overleftarrow{O_n}.
\end{cases}
\end{align}
For simplicity, we henceforth suppress the dependence on $\omega_n$ and write $\rho_v, V^{(n)}, R_n$ and $\nu_n$ instead of $\rho_v(\omega_n)$, $V_{\omega_n}$, $R^{\omega_n}$ and $\nu^{\omega_n}$ respectively. Moreover, we let $Z^{(n, \text{pl})}$ denote the stochastic process associated with the triple $(T_n^*, R_n, \nu_n)$, which was called the planted tree, see Section \ref{sctn:RWRE conductances}. Since $Z^{(n, \text{pl})}$ is defined on the planted tree, we define $Z^{(n)}$ to be its trace on the unplanted tree $T_n$; that is, we set 
\[
Z^{(n)}(t) = \begin{cases}
Z^{(n, \text{pl})}(t), &\text{ if } Z^{(n, \text{pl})}(t) \neq \overleftarrow{O_n}, \\
O_n, &\text{ if } Z^{(n, \text{pl})}(t) = \overleftarrow{O_n}.
\end{cases}
\]
It follows from our definitions that $Z^{(n)}(t)$ is a continuous-time random walk on $T_n$, with $\textsf{exp}(1)$ holding times at each vertex, and that at each jump time, the transition probabilities from a vertex are given by the Dirichlet weights of \eqref{eqn:normalised D weights}.

Finally, in light of Lemma \ref{lem:Dirichlet rep of LERRW}, in order to connect this with LERRW we need to consider the law of $Z^{(n)}$ annealed over the Dirichlet weights. We denote this process $Z^{(n)}$ and its law $\tilde{P}_{O_n}^{(\mathbf{b}^{(n)})}$, so that
\begin{equation} \label{lerrwannealed}
\tilde{P}^{(\mathbf{b}^{(n)})}_{O_n}\left(Z^{(n)}\in \cdot\right)=\int \tilde{P}_{O_n,\omega_n}({Z}^{(n)}\in \cdot) \mathbb{P}^{(\mathbf{b}^{(n)})}(d \omega_n).
\end{equation}
It follows from Lemma \ref{lem:Dirichlet rep of LERRW} that under $\tilde{P}^{(\mathbf{b}^{(n)})}_{O_n}$, $Z^{(n)}$ has the law of a \textit{quenched} (with respect to the randomness of $T_n$) continuous-time LERRW on $T_n$, with initial weights \eqref{scheme2} and reinforcement parameter $\Delta$.

\subsection{Candidate for the scaling limit} \label{sctn:scaling limit candidate}

To construct the scaling limit of $(T_n, R_n, \nu_n)$ we need to imitate the definitions of \eqref{eqn:discrete V, res, mu def} in the continuum.  
To this end, we take $\Tgc$ as in Section \ref{sctn:stable trees bckgrnd}, denote its root by $O$ and for all $\sigma, \sigma' \in \Tgc$ we set (cf. \eqref{stablemet}):
\[
\dbt(\sigma,\sigma') =
\begin{cases}
(\dt(O, \sigma)+1)^{1-\alpha} + (\dt(O,\sigma')+1)^{1-\alpha} - 2(\dt(O, \sigma\wedge \sigma')+1)^{1-\alpha}, &\text{ if } \alpha < 1, \\
\log (\dt(O, \sigma)+1) + \log (\dt(O, \sigma')+1) - 2\log (\dt(O, \sigma \wedge \sigma')+1), &\text{ if } \alpha = 1,
\end{cases}
\]
where $\sigma\wedge \sigma'$ denotes the most recent common ancestor of $\sigma$ and $\sigma'$. This is a metric on $\Tgc$.
{\begin{defn}[Snake process]\label{def:snake}
Let $(\phib (\sigma))_{\sigma\in \Tgc}$ be the $\mathbb{R}$-valued Gaussian process with law $\mathbb{P}$ whose distribution, given $\Tgc$ is characterised by 
\begin{equation} \label{snake}
\mathbb{E} \phib(\sigma)=0, \qquad \textnormal{Cov}(\phib(\sigma),\phib(\sigma'))=\dbt(O,\sigma\wedge \sigma').
\end{equation}
\end{defn}}

Since 
\[
\mathbb{E} {|\phib(\sigma) - \phib(\sigma')|^2} =\dbt(\sigma,\sigma'),
\]
which is a metric on $\Tgc$, it can be seen by the same arguments as in \cite[Section 6]{duquesne2005probabilistic} that such a process $\phib$ exists, and that it has a continuous modification, with bounded sample paths, $\bPb \times \mathbb{P}$-almost surely. In the continuum, $\phib$ will play the same role as the potential $V_{\omega_n}$ in \eqref{eqn:discrete V, res, mu def}.

To capture that our diffusions are processes on ``natural scale'' it is desirable to introduce the notion of a length measure on $\Tgc$ which extends Lebesgue measure on $\mathbb{R}$. For real trees it was first presented in \cite{evans2006rayleigh} and later extended to any separable 0-hyperbolic metric space in \cite{athreya2017invariance}. Denote the skeleton of $(\Tgc,d)$ by
\[
(\Tgc)^0=\bigcup_{x\in \Tgc} (O,x).
\]
Recalling that $\Tgc$ is a separable pointed metric space, observe that if $D\subset \Tgc$ is a dense countable subset, then the previous definition is still the same when the union is taken over points in $D$. In particular, $(\Tgc)^0\in \mathcal{B}(\Tgc)$ and $\mathcal{B}(\Tgc)|_{(\Tgc)^0}=\sigma(\{(x,y); x, y\in D\})$. Therefore, by Carath\'eodory's Extension Theorem, there exists a unique $\sigma$-finite measure $\lambda^{\gamma}$ on $\Tgc$, called the length measure, such that $\lambda^{\gamma}(\Tgc\setminus (\Tgc)^0)=0$ and
\begin{equation}\label{eqn:length measure prop}
\lambda^{\gamma}((O,x])=\dt(O,x), \qquad x\in \Tgc.
\end{equation}
Alternatively, the length measure is the trace onto the skeleton of the $\gamma$-stable tree of the one-dimensional Hausdorff measure on it. 

Now fix the reinforcement parameter $\Delta>0$ and let $\mathcal{W}$ be the space of continuous functions $\Tgc \to \R$ vanishing at the root, and let $\Omega$ be the space of processes $\R_+ \to \Tgc$. Given a realisation $\phib \in \mathcal{W}$, we firstly define
\begin{equation}\label{eqn:distorted res def}
R_{\phi}(x,y)=\begin{cases}
\displaystyle \int_{[x,y]} (\dt(O,z)+1)^{-\alpha} e^{\sqrt{\frac{4 \Delta}{1-\alpha}} \phib(z)+\frac{\Delta }{1-\alpha} [(\dt(O,z)+1)^{1-\alpha}-1]} \lambda^{\gamma}(\text{d} z), &\text{ if } \alpha <1, \\
\displaystyle \int_{[x,y]} e^{\sqrt{4 \Delta} \phiz(z)+(\Delta-1) \log(\dt(O,z)+1)} \lambda^{\gamma}(\text{d} z), &\text{ if } \alpha = 1,
\end{cases}
\end{equation}
for all $x, y\in \Tgc$, and secondly define $\nu_{\phi}$ to be the measure which is absolutely continuous with respect to the measure $\mu$ defined in \eqref{eqn:natmeas2} with density given by

\begin{equation}\label{eqn:distorted meas def}
(\text{d} \nu_{\phi}/\text{d} \mu)(x)= \begin{cases}
 \displaystyle (\dt(O,x)+1)^{\alpha} e^{-\left[\sqrt{\frac{4 \Delta}{1-\alpha}} \phib(x)+\frac{\Delta}{1-\alpha}[(\dt(O,x)+1)^{1-\alpha}-1]\right]},  &\text{ if } \alpha < 1,
 \\
 \displaystyle e^{-\left[\sqrt{4 \Delta}\phiz(x)+(\Delta-1) \log(\dt(O,x)+1) \right]},  &\text{ if } \alpha = 1.
\end{cases}
\end{equation}
We let $\Tgphi$ denote the random mm-space $(\Tgc, \Rphi, \nuphi)$, let $X=(X_t)_{t\ge 0}$ denote the process canonically associated with it via the resistance form of \cite[Definition 2.1]{croydon2016scaling}, and let $\tilde{P}_{O,\phi}$ denote the quenched law of $X$ when started from $O$. Finally, given $\Tgc$, we denote by $P_{O}$ the corresponding annealed probability measure of the process on $\mathcal{W}\times \Omega$ defined by 
\begin{equation}\label{eqn:annealed X}
P_{O} \left( (X_t)_{t \geq 0} \in \cdot \right)=\int \tilde{P}_{O,\phi} \left( (X_t)_{t \geq 0} \in \cdot \right) \mathbb{P}(d \phib).
\end{equation}

\begin{remark}
We can also view $((X_t)_{t\ge 0},\tilde{P}_{O,\phi})$ as an $\nu_{\phi}$-Brownian motion on $(\Tgc,\Rphi)$ as characterised in \cite[Proposition 1.9]{athreya2013brownian}. In this setting, $(X_t)_{t\ge 0}$ is a diffusion process with regular Dirichlet form 
\[
\mathcal{E}_{\phi}(u,v)=
\begin{cases}
\displaystyle \frac{1}{2} \int_{\Tgc} \nabla_{\Rphi} u(z)\cdot \nabla_{\Rphi} v(z) (\dt(O,z)+1)^{-\alpha} e^{\sqrt{\frac{4 \Delta}{1-\alpha}} \phib(z)+\frac{\Delta}{1-\alpha} [(\dt(O,x)+1)^{1-\alpha}-1]} \lambda^{\gamma}(\textnormal{d} z), &\text{ if } \alpha<1,
\\
\displaystyle \frac{1}{2} \int_{\Tgc} \nabla_{\Rphi} u(z)\cdot \nabla_{\Rphi} v(z) e^{\sqrt{4 \Delta} \phiz(z)+(\Delta-1) \log(\dt(O,z)+1)} \lambda^{\gamma}(\textnormal{d} z), &\text{ if } \alpha=1,
\end{cases}
\]
for every $u, v\in \mathcal{D}(\mathcal{E}_{\phi})$, 
where $\mathcal{D}(\mathcal{E}_{\phi})=\{u\in \mathcal{A}: \nabla u\in L^2(\lambda)\}\cap L^2(\nu_{\phi})\cap \mathcal{C}_{\infty}(\Tgc)$. As usual, by $\mathcal{C}_{\infty}(\Tgc)$ we refer to the space of continuous functions which vanish at infinity. A function $u$ belongs to $\mathcal{A}$ if and only if $u$ is locally absolutely continuous. We need to stress that the gradients of $u$ and $v$ correspond to gradients of functions from $(\Tgc,\Rphi)$, i.e. $\nabla_{\Rphi} f$, of $f\in \mathcal{D}(\mathcal{E}_{\phi})$, is the function, which is unique up to 
\[
\lambda^{(\Tgc,\Rphi)}(\textnormal{d}z)=
\begin{cases}
(\dt(O,x)+1)^{-\alpha} e^{\sqrt{\frac{4 \Delta}{1-\alpha}}\phib(z)+\frac{\Delta}{1-\alpha} [(\dt(O,x)+1)^{1-\alpha}-1]} \lambda^{\gamma}(\textnormal{d}z),
&\text{ if } \alpha<1,
\\
e^{\sqrt{4 \Delta} \phiz(z)+(\Delta-1) \log(\dt(O,z)+1)} \lambda^{\gamma}(\textnormal{d}z), &\text{ if } \alpha=1,
\end{cases}
\]
-zero sets that satisfies
\[
\int_{[x_1,x_2]} \nabla_{\Rphi} f(z) \lambda^{(\Tgc,\Rphi)}(\textnormal {d}z)=f(x_2)-f(x_1), \qquad x_1, x_2\in \Tgc, \qquad f\in \mathcal{D}(\mathcal{E}_{\phi}).
\]

We will not directly use the theory of Dirichlet forms in this paper (though it is hidden behind the result we apply from \cite{croydon2016scaling}).
\end{remark}

\begin{remark}[Connection to 1d results]

The analogue of \cite[Proposition 1.9]{lupu2018scaling} (which is written with a scaling factor of $2^n$ rather than $na_n^{-1}$) equally applies to the LERRW on a single branch of $n^{-1}a_n T_n$ with initial weights 
\[
w_0^{(n)}(x, x-n^{-1}a_n) = \frac{1}{2} na_n^{-1} L_0(x) L_0(x-n^{-1}a_n).
\]
Comparing with \eqref{scheme2}, we see that we can formulate our process in the setting of \cite[Proposition 1.9]{lupu2018scaling} by taking $L_0(x)^2 \approx 2\Delta^{-1} (|x|+1)^{\alpha}$. In this case, $S_0(x)$ as defined by \cite[Equation (1.5)]{lupu2018scaling} satisfies $S_0(x) \approx \frac{\Delta [(|x|+1)^{1-\alpha}-1]}{2(1-\alpha)}$ if $\alpha < 1$, and $S_0(x) \approx \frac{\Delta \log (|x|+1)}{2}$ if $\alpha=1$. 

According to \cite[Proposition 1.9]{lupu2018scaling}, such a rescaled LERRW on $\mathbb{Z}$ converges to a diffusion in the random potential
\[
2\sqrt{2} W(S_0(x))+2 |S_0(x)|-2 \log (L_0(x)) - \log \Delta,
\]
cf. \cite[Equation (1.8)]{lupu2018scaling}, where $W$ is a standard Brownian motion on $\R$ (the final term of $\log \Delta$ above appears since we have rescaled everything by $\Delta$ to fit into the framework of \cite{lupu2018scaling}). In particular, we can identify $W((|x|+1)^{1-\alpha}-1)$ with $\phib(|x|)$ and $W(\log (|x|+1))$ with $\phiz(|x|)$, with a slight abuse of notation. Substituting the above values of $L_0(x)$ and $S_0(x)$ we obtain a limiting potential of the form
\begin{align*}
&2\sqrt{2} W\left(\frac{\Delta[(|x|+1)^{1-\alpha}-1]}{2(1-\alpha)}\right)+\frac{\Delta [(|x|+1)^{1-\alpha}-1]}{1-\alpha} -\log (2(|x|+1)^{\alpha}) \\
&\qquad \qquad = \sqrt{\frac{4\Delta}{1-\alpha}}\phib(|x|)+\frac{\Delta[(|x|+1)^{1-\alpha}-1]}{1-\alpha}-\alpha \log (|x|+1) - \log 2
\end{align*}
when $\alpha<1$, and
\begin{align*}
2\sqrt{2} W\left(\frac{\Delta\log (|x|+1)}{2} \right) + \Delta \log (|x|+1) - \log (2(|x|+1))  &= \sqrt{4\Delta}\phiz(|x|) + (\Delta - 1) \log (|x|+1) - \log 2
\end{align*}
when $\alpha = 1$, which is consistent with \eqref{eqn:distorted res def} up to the constant $\log 2$ (but note that the effect of adding a constant cancels out when inserted into both \eqref{eqn:distorted res def} and \eqref{eqn:distorted meas def}).
\end{remark}

\section{Scaling limit of $(T_n, R_n, \nu_n)$ and the LERRW $X^{(n)}$} \label{sctn:scaling lim}

Our goal in this section is to prove the following proposition. The space $(T_n, R_n, \nu_n)$ is as defined just below \eqref{eqn:discrete V, res, mu def}.

\begin{prop} \label{prop:mm space limit}
Under the joint law $\bPb \times \Pb$ and with initial weights as in \eqref{scheme2}, the following convergence holds with respect to the GHP topology as $n \to \infty$:

\[
(T_n, (n a_n^{-1})^{-1} R_n, (2 n)^{-1} \nu_n,O_n) \overset{(d)}{\to} (\Tgc, \Rphi,\nu_{\phi},O).
\]

\end{prop}

Throughout the section, we will work pointwise on the probability space $(\mathbf{\Omega}, \mathbf{\F}, \bPb)$ (on which we defined $(T_n, d_n, \mu_n)$), and where the convergence of \eqref{eqn:stable tree scaling limit def} holds almost surely. This means that most of the statements that follow should be written conditionally on $T_n$ or on $\Tgc$. To make the arguments clearer to follow, we have not written this explicitly in the statements or the proofs, and instead ask the reader to keep this in mind throughout. There is one specific case (Proposition \ref{prop:Kolmogorov}) where we need to restrict to a certain set $A_{n, \epsilon} \subset \mathbf{\Omega}$, and in this case we make it explicit.

We also let $V^{(n, \alpha)}( 2nt) = V_{\omega_n}(x_{\lfloor 2nt \rfloor})$, where $V_{\omega_n}$ is as defined in \eqref{eqn:discrete V, res, mu def}. At times we will abuse notation and write $V^{(n,\alpha)}(x)$ in place of $V^{(n,\alpha)}(2nt)$, where $t$ is such that $x = x_{\lfloor 2nt \rfloor}$.

We will prove Proposition \ref{prop:mm space limit} in two main steps. We first apply Skorohod Representation and assume that \eqref{contourconv} holds almost surely on $(\mathbf{\Omega}, \mathbf{\F}, \bPb)$. On this space, the set of all pairs $(x_{\lfloor 2nt \rfloor},p_{H_{(\gamma)}}(t))$ defines a correspondence between $n^{-1} a_n T_n$ and $\Tgc$ with distortion going to $0$; see \eqref{corresp} and Section \ref{sctn:GHP conv background} for details.

We first prove the following claim.
    \begin{claim} \label{claim:Vbranch} 
    Take $\alpha\le 1$. We have for almost every $\omega \in \mathbf{\Omega}$ that
    \begin{align} \label{wlimit}
    \left( V^{(n,\alpha)}( 2nt ) \right)_{t \in [0,1]} \overset{(d)}{\to} 
    \begin{cases}
     \left( \frac{\Delta [(\dt(O,t)+1)^{1-\alpha}-1]}{1-\alpha} - \alpha \log \left(\dt(O, t) + 1\right)  +  \sqrt{\frac{4\Delta}{1-\alpha}} \phib (t)\right)_{t \in [0,1]} &\text{ if } \alpha<1,
    \\ \\
    \left((\Delta-1) \log(\dt(O,t)+1)+\sqrt{4 \Delta} \phiz(t)\right)_{t\in [0,1]} &\text{ if } \alpha=1,
    \end{cases}
    \end{align}
    with respect to the topology of uniform convergence on $C([0,1])$.
    \end{claim}
    
    Then, we apply Skorohod Representation a second time to work on a probability space where the convergence of \eqref{wlimit} also holds almost surely. On this space, we show that the rescaled resistances and measures defined in \eqref{eqn:discrete V, res, mu def} converge to the limit candidates suggested in \eqref{eqn:distorted res def} and \eqref{eqn:distorted meas def}.

\subsection{The limiting potential} \label{sctn:potential convergence}

In this subsection we establish \eqref{wlimit} above. 

\subsubsection{Proof of Claim \ref{claim:Vbranch}}

Throughout this section, recall that $(\mathbf{\Omega}, \mathbf{\F}, \bPb)$ is a probability space where the convergence of \eqref{contourconv} holds almost surely. This also implies that the convergence of \eqref{eqn:stable tree scaling limit def} holds almost surely (see \cite[Theorem 4.2]{gall2006trees}). For $n \geq 1$ and $x \in T_n$ we define the shorthand
\[
\Delta_n=
\begin{cases}
\Delta (n a_n^{-1})^{-(1-\alpha)} &\text{ if } \alpha<1,
\\
\Delta &\text{ if } \alpha=1,
\end{cases} 
\qquad |x|=d_n(O_n,x).
\]

Note though that the reinforcement parameter for the process defined in Section \ref{sctn:LERRW and RWRE connection trees} is still $\Delta$, not $\Delta_n$. If $x\in T_n$, then the Dirichlet weights for the model have parameters given by \eqref{eqn:rescaled Dirichlet parameters}, so that by \eqref{scheme2} the weight at $x$ is given by
\[
\mathbf{b}_x^{(n)}=
\displaystyle \left( \frac{(|x| + n a_n^{-1})^{\alpha} + \Delta_n}{2\Delta_n},  \frac{(|x|+1+n a_n^{-1})^{\alpha}}{2\Delta_n}, \frac{(|x|+1+n a_n^{-1})^{\alpha}}{2\Delta_n}, \ldots, \frac{(|x|+1+n a_n^{-1})^{\alpha}}{2\Delta_n} \right).
\]  
For each $x \in T_n\setminus \{O_n\}$, as the ratio of two independent Gamma random variables, we observe that
\[
\rho_{x} :=\frac{W_{\overleftarrow{x}}^{\overleftarrow{\overleftarrow{x}}}}{W_{\overleftarrow{x}}^x}\sim \frac{\Gamma(b_{\overleftarrow{x}}^{(\overleftarrow{x},\overleftarrow{\overleftarrow{x}})},1)}{\Gamma(b_{\overleftarrow{x}}^{(\overleftarrow{x},x)},1)}
\sim 
\displaystyle \beta' \left( \frac{\left((|x|-1) + n a_n^{-1} \right)^{\alpha} + \Delta_n}{2\Delta_n},  \frac{(|x|+n a_n^{-1})^{\alpha}}{2\Delta_n} \right), 
\]
and $(\rho_{x})_{x \in T_n\setminus \{O_n\}}$ is a sequence of independent random variables. Here $\beta'(a,b)$ denotes the beta prime distribution with positive parameters $a$ and $b$, and has probability density function
\[
x^{a-1}(1+x)^{-a-b}/B(a,b), \qquad x>0,
\]
where $B$ is the beta function.

We start by giving some preliminary claims about the expectation and variance of the quantities $\log \rho_{x}$ summed along a branch. The proofs are elementary but are included in the appendix.

\begin{claim}\label{claim:log sum exp}
\begin{enumerate}[(i)]
    \item For all $x \in T_n$ we have that
    $|\E{\log \rho_x}| \leq \frac{2\alpha}{\left(|x| + n a_n^{-1} \right)} + \frac{4\Delta_n}{\left(|x| + n a_n^{-1} \right)^{\alpha}}$.
    \item For almost every $\omega\in \mathbf{\Omega}$, it holds uniformly over $t \in [0,1]$ as $n \to \infty$ that:
    \begin{align*}
\sum_{O_n\prec x \preceq x_{\lfloor 2nt \rfloor}} &\E{\log \rho_{x} }\to \frac{ \Delta [(\dt(O,t)+1)^{1-\alpha}-1]}{1-\alpha} - \alpha \log \left(\dt(O, t) + 1\right) \hspace{1cm} &\text{if }  \alpha<1,
\\
\sum_{O_n\prec x \preceq x_{\lfloor 2nt \rfloor}} &\E{\log \rho_{x}}\to (\Delta-1) \log(\dt(O,t)+1),  \hspace{1cm} &\text{if } \alpha = 1.
\end{align*}
\end{enumerate}
\end{claim}
\begin{proof}
See Lemma \ref{lem:log sum exp app} in the Appendix.
\end{proof}

\begin{claim}\label{claim:log var}
For almost every $\omega \in \mathbf{\Omega}$, it holds uniformly over $t \in [0,1]$ as $n\to \infty$ that:
\begin{align*}
\sum_{O_n\prec x \preceq x_{\lfloor 2nt \rfloor}} &\textnormal{Var} (\log \rho_{x}) \to \frac{4 \Delta [(\dt(O,t)+1)^{1-\alpha}-1]}{1-\alpha} \hspace{1cm} &\text{if } \alpha < 1,
\\
\sum_{O_n\prec x \preceq x_{\lfloor 2nt \rfloor}} &\textnormal{Var} (\log \rho_{x}) \to 4 \Delta \log(\dt(O,t)+1) \hspace{1cm} &\text{if } \alpha = 1.
\end{align*}
\end{claim}
\begin{proof}
See Lemma \ref{cor:log var app} in the Appendix.
\end{proof}

\begin{claim}\label{claim:log mgf}
For all $\alpha \leq 1$ and almost every $\omega \in \mathbf{\Omega}$, it holds for all $x \in T_n$ and all $1 \leq k \leq (na_n^{-1})^{1/2}$ that
\begin{align*}
    \E{e^{k\log \rho_{x}}}  
    \leq \exp \left\{\left(1+3\Delta\right) \left( \frac{2k\alpha}{|x|+n a_n^{-1}} + \frac{3k^2\Delta_n}{(|x|+n a_n^{-1})^{\alpha}} \right)  \right\}.
\end{align*}
\begin{align*}
    \E{e^{k\log (\rho_{x}^{-1})}}  
    \leq \exp \left\{\left(1+3\Delta\right) \left( \frac{2k\alpha}{|x|+n a_n^{-1}} + \frac{3k^2\Delta_n}{(|x|+n a_n^{-1})^{\alpha}} \right)  \right\}.
\end{align*}
\end{claim}
\begin{proof}
See Lemma \ref{claim:log mgf app} in the Appendix.
\end{proof}

We will prove the convergence of Claim \ref{claim:Vbranch} in two steps. Firstly we consider the recentred processes defined by
    \begin{equation} \label{Mdiff}
    M_i^{(n,\alpha)} = 
    \sum_{O_n\prec x \preceq x_i} \left( \log \rho_x - \E{ \log \rho_x}\right)
    \end{equation}
for $0 \leq i \leq 2n$, and extended to non-integer time indices by interpolation. We then prove that $\left( M_{2nt}^{(n,\alpha)}\right)_{t \in [0,1]}$ converges to an appropriate snake process as follows. Firstly, by applying a martingale CLT along finitely many branches of $T_n$, we deduce that the finite dimensional marginals of $M^{(n,\alpha)}$ converge to those of the snake process. For this part of the proof, the martingales in question are indexed by the branch lengths, and not by the time interval $[0,1]$. Then, to extend this finite-dimensional convergence to full convergence in $C([0,1])$, we verify Kolmogorov's tightness condition, this time considering the whole process indexed by the time interval $[0,1]$.

We also set 
\[
A(t) =     \begin{cases}
    \displaystyle \frac{4 \Delta [(\dt(O,t)+1)^{1-\alpha}-1]}{1-\alpha}
    &\text{ if } \alpha<1,
    \\
    \displaystyle 4 \Delta \log(\dt(O,t)+1) &\text{ if } \alpha=1.
    \end{cases}
\]

\begin{prop}[Finite dimensional convergence]
For each $n \geq 1$, let $(M_m^{(n, \alpha)})_{m\ge 0}$ be as in \eqref{Mdiff}, and for each $1 \leq m \leq 2n$ let 
$Z_m^{(n, \alpha)} = M_m^{(n,\alpha)} - M_{m-1}^{(n,\alpha)}$.
\begin{enumerate}
    \item Almost surely on $\mathbf{\Omega}$, it holds for each $t\in (0,1)$ that
    \[
 \sum_{m:x_m \preceq x_{\lfloor 2nt \rfloor}}  \E{(Z_m^{(n,\alpha)})^2} \to
A(t),
    \]
    as $n \to \infty$.
    \item Almost surely on $\mathbf{\Omega}$, it holds $\forall \varepsilon>0, \forall t \in (0,1)$, 
    \[
    \sum_{m:x_m \preceq x_{\lfloor 2nt \rfloor}} \E{(Z_m^{(n,\alpha)})^2 \mathbbm{1}\{|Z_m^{(n,\alpha)}|>\varepsilon\}}\to 0,
    \]
    as $n\to \infty$.
    \item Almost surely on $\mathbf{\Omega}$, for any $k \geq 1$ and any sequence $0 < t_1 < \ldots < t_k < 1$, it holds that
    \begin{equation}\label{eqn:fd conv}
    (    M_{2nt_i}^{(n,\alpha)})_{1 \leq i \leq k} \to \begin{cases} \left(\sqrt{\frac{4\Delta}{1-\alpha}} \phib (t_i)\right)_{1 \leq i \leq k} &\text{ if } \alpha < 1, \\
    \left(\sqrt{4\Delta} \phiz (t_i)\right)_{1 \leq i \leq k} &\text{ if } \alpha = 1,
    \end{cases}
    \end{equation}
    jointly in distribution as $n \to \infty$.
\end{enumerate}
\end{prop}
\begin{proof}
\begin{enumerate}
\item This is \cref{claim:log var}.
\item First note that, for any $t \in (0,1)$,
\begin{align}\label{eqn:Z squared calc}
\begin{split}
    \sum_{m:x_m \preceq x_{\lfloor 2nt \rfloor}} \E{(Z_m^{(n,\alpha)})^2 \mathbbm{1}\{|Z_m^{(n,\alpha)}|>\varepsilon\}} 
    &=
    \sum_{m:x_m \preceq x_{\lfloor 2nt \rfloor}} \int_{\epsilon^2}^{\infty} \pr{(Z_x^{(n,\alpha)})^2>y} dy
    \\
    &=
    \sum_{m:x_m \preceq x_{\lfloor 2nt \rfloor}} \int_{\epsilon^2}^{\infty} \pr{|\log \rho_{x_m} - \E{\log \rho_{x_m}}|> \sqrt{y}} dy.
    \end{split}
\end{align}
Also note that it follows from Claim \ref{claim:log sum exp}$(i)$ and Claim \ref{claim:log mgf} that almost surely on $\mathbf{\Omega}$, we have for all $1 \leq k \leq (na_n^{-1})^{1/2}$ and all $x \in T_n$,
\begin{align}\label{eqn:recentred mgf}
\E{e^{k(\log \rho_{x} - \E{\log \rho_{x}})}} \vee \E{e^{k(\log \rho_{x}^{-1} - \E{\log \rho_{x}^{-1}})}}  \leq \exp \left\{3\left(1+3\Delta\right) \left( \frac{2k\alpha}{|x|+n a_n^{-1}} + \frac{3k^2\Delta_n}{(|x|+n a_n^{-1})^{\alpha}} \right)  \right\}.
\end{align}

Taking $k = (na_n^{-1})^{1/2}$, the latter expression in \eqref{eqn:Z squared calc} is upper bounded by 
        \begin{align*}
        &n\sup_{x: |x| \geq 1} \int_{\epsilon^2}^{\infty}1 \wedge \left\{ \E{e^{(na_n^{-1})^{\frac{1}{2}} \left( \log \rho_{x} - \E{\log \rho_{x}}\right)}} e^{-(na_n^{-1})^{\frac{1}{2}}\sqrt{y}} \right\}dy 
        \\
        &\quad + n \sup_{x: |x|\ge 1} \int_{\epsilon^2}^{\infty} 1 \wedge \left\{\E{e^{(na_n^{-1})^{\frac{1}{2}} \left( \log \rho_x^{-1}-\E{\log \rho_x^{-1}}\right)}} e^{-(na_n^{-1})^{\frac{1}{2}}\sqrt{y}} \right\}dy 
        \\
        &\leq 2n \sup_{x: |x| \geq 1} \int_{\epsilon^2}^{\infty} 1 \wedge \left\{ \exp\left\{ 3\left(1+3\Delta\right) \left( \frac{2\alpha}{(n a_n^{-1})^{1/2}} + 3\Delta \right) -(na_n^{-1})^{\frac{1}{2}}\sqrt{y} \right\}  \right\}dy.
        \end{align*}
For any $\epsilon>0$, there exists $N_{\epsilon} < \infty$ such that the minimum in the final integrand is equal to the second expression for all $y \in (\epsilon^2, \infty)$. By dominated convergence, this integral therefore tends to $0$ as $n \to \infty$, as required. (In fact, for fixed $\epsilon>0$, this upper bound is also uniform over $t \in (0,1)$).

\item Without loss of generality we can assume that  $t_1, \ldots, t_k \in \frac{1}{2n}\Z$; the general case follows since we interpolate and the contribution of a single jump goes to zero in probability as $n \to \infty$ (for example by Markov's inequality). Note that, for fixed $t \in (0,1)$, parts 1 and 2 exactly verify the conditions given in \cite[Corollary 3.1 and Theorem 3.2]{hallheyde2014book} to imply that 
\begin{equation}\label{eqn:fd convergence one t}
 M_{2nt}^{(n,\alpha)} \to N(0, A(t)),
\end{equation}
as $n \to \infty$. Therefore, given such a sequence $0 < t_1 < \ldots < t_k < 1$, for each $k \geq 1$ we can define an \textit{augmented sequence} obtained by adding the time indices corresponding to all the most recent common ancestors of pairs of vertices in the original sequence. Given this augmented sequence, we can then sum the contributions along the relevant branch segments between vertices of the form $x_{t_i}$ and $x_{t_j}$, where $t_i$ and $t_j$ are such that there is no $\ell \leq k$ with $t_i \preceq t_{\ell} \preceq t_j$. Since the process evolves independently along distinct branch segments and the sum of independent Gaussians is Gaussian, the finite-dimensional result then follows from \eqref{eqn:fd convergence one t}.
\end{enumerate}
\end{proof}

In order to strengthen the above convergence, we will verify Kolmogorov's tightness condition. We first give a preliminary lemma. For this, for $\alpha < 1$ we first define $\dbtn (s,t)$ to be equal to
{\small 
\begin{align}\label{eqn:def distance discrete}
 (na_n^{-1})^{-(1-\alpha)}\left[(d_n(O_n, x_{\lfloor 2nt \rfloor})+na_n^{-1})^{1-\alpha} + (d_n(O_n, x_{\lfloor 2ns \rfloor})+na_n^{-1})^{1-\alpha} 
 - 2(d_n(O_n, x_{\lfloor 2ns \rfloor \wedge \lfloor 2nt \rfloor})+na_n^{-1})^{1-\alpha}\right].
\end{align}
}

Note that $\dbtn (s,t) = \dbt(s,t) + o(1)$, where the $o(1)$ is uniform over all $s, t \in [0,1]$ by \eqref{eqn:stable tree scaling limit def 1}. Also let $D_n = (na_n^{-1})^{-1} \diam (T_n)$. We first give a useful lemma.

\begin{lemma}\label{lem:alpha distance UB}
Almost surely on $\mathbf{\Omega}$, we have for all $n \geq 1$ and all $s,t \in [0,1]$ that
\[
\dbtn (s,t) \leq \begin{cases}
(1-\alpha) (D_n+1)^{-\alpha} \dbztn(s,t) &\text{ if } \alpha \leq 0, \\
2\dbztn(s,t)^{1-\alpha} &\text{ if } \alpha \in (0, 1).
\end{cases}
\]
\end{lemma}
\begin{proof}
By breaking at the most recent common ancestor, it is enough to prove this when $s \preceq t$. If $\alpha \leq 0$, we just write
\[
\dbtn(s,t) \leq (1-\alpha)(na_n^{-1})^{-1}d_n(x_{\lfloor 2ns \rfloor},x_{\lfloor 2nt \rfloor}) [(na_n^{-1})^{-1} d_n(O_n, x_{\lfloor 2nt \rfloor})+1]^{-\alpha} \leq (1-\alpha) (D_n+1)^{-\alpha} \dbztn(s,t).
\]
If instead $\alpha \in (0,1)$, we treat two cases:
\begin{enumerate}
    \item If $d_n(O_n, x_{\lfloor 2ns \rfloor}) \geq \frac{1}{2}d_n(O_n, x_{\lfloor 2nt \rfloor})$, then we obtain that
    \begin{align*}
\dbtn(s,t) \leq (1-\alpha) (na_n^{-1})^{-1}d_n(x_{\lfloor 2ns \rfloor},x_{\lfloor 2nt \rfloor}) [(na_n^{-1})^{-1} d_n(O_n, x_{\lfloor 2ns \rfloor})]^{-\alpha} &\leq (1-\alpha)\dbztn(s,t)^{1-\alpha}.
    \end{align*}
    \item If $d_n(O_n, x_{\lfloor 2ns \rfloor}) < \frac{1}{2}d_n(O_n, x_{\lfloor 2nt \rfloor})$, then we get
    \begin{align*}
\dbtn(s,t) \leq (na_n^{-1})^{-(1-\alpha)}d_n(O_n, x_{\lfloor 2nt \rfloor})^{1-\alpha} &\leq 2^{1-\alpha}\dbztn(s,t)^{1-\alpha}.
    \end{align*}
\end{enumerate}
\end{proof}

We are now ready to verify the tightness condition.

\begin{prop}[Kolmogorov's condition]\label{prop:Kolmogorov}
For every $\epsilon>0, n \geq 1$ there exist $p>0, q>1$, $C_{\epsilon,p,q}<\infty$ and an event $A_{n, \epsilon} \subset \mathbf{\Omega}$ with $\prb{A_{n, \epsilon}} \geq 1- \epsilon$, such that on the event $A_{n, \epsilon}$, we have for all $s,t \in [0,1]$ that
    \[
\econd{ \frac{|M_{2nt}^{(n,\alpha)} - M_{2ns}^{(n,\alpha)}|^p}{|s-t|^{q}}}{A_{n, \epsilon}} < C_{\epsilon,p,q}.
\]
In particular, Kolmogorov's tightness condition is satisfied on the event $A_{n, \epsilon}$ and for almost every $\omega \in \mathbf{\Omega}$, the convergence of \eqref{eqn:fd conv} holds in distribution on the space $C([0,1])$ equipped with the uniform topology.
\end{prop}
\begin{proof}
First note that it follows from \cite[Lemma 1.4]{marzouk2018snake} and \cite[Theorem 2]{Kortchemski2017sub} that for any $\gamma' < \frac{\gamma - 1}{\gamma}$ and $\epsilon>0$, we can choose $D_{\epsilon}, C_{\epsilon, \gamma'}<\infty$ such that with probability $1-\epsilon$, we have that
\begin{align}\label{eqn:marzouk holder result}
\begin{split}
\dbztn (s,t) = (na_n^{-1})^{-1}d_n(x_{2nt}, x_{2ns}) &\leq C_{\epsilon, \gamma'}|t-s|^{\gamma'} \text{ for all } s, t \in [0,1], \\
D_n := \sup_{t \in [0,1]} (na_n^{-1})^{-1} d_n(O_n, x_{2nt}) &\leq D_{\epsilon}-1.
\end{split}
\end{align}
(Note that \cite[Lemma 1.4]{marzouk2018snake} actually states a result using the lexicographical ordering rather than the contour ordering as we use, but since the difference in the labels of two fixed vertices can only decrease by at most a factor of $2$ in the contour ordering, this immediately implies the same result with the contour ordering and therefore really implies the first line of \eqref{eqn:marzouk holder result}).

\textbf{Case $\alpha<1$.} Assume for now that $2ns$ and $2nt$ are integers and that $x_{2ns}$ is an ancestor of $x_{2nt}$. This will then extend to the general case by breaking paths at the most recent common ancestor and since $M^{(n, \alpha)}$ is defined by interpolation. Note that by \eqref{eqn:recentred mgf}, we almost surely have for all $1 \leq k \leq (na_n^{-1})^{1/2}$ and all $x \in T_n$ that
\begin{align*}
\E{e^{k(\log \rho_{x} - \E{\log \rho_{x}})}}  &\leq \exp \left\{3\left(1+3\Delta\right) \left( \frac{2k\alpha}{|x|+n a_n^{-1}} + \frac{3k^2\Delta_n}{(|x|+n a_n^{-1})^{\alpha}} \right)  \right\}.
\end{align*}
Therefore, for all $1 \leq k \leq (na_n^{-1})^{1/2}$ we deduce that there exists $C_{\Delta} < \infty$ such that (using also that, given $T_n$, the sequence $(\rho_x)_{x \in T_n}$ is independent)
\begin{align}\label{eqn:recentred mgf*}
\begin {split}
    &\E{\exp \left\{\sum_{x_{ns} \prec x \preceq x_{nt}} (\log \rho_{x} - \E{\log \rho_{x}}) \right\}} \\
    &\qquad \leq \exp \left\{\sum_{x_{ns} \prec x \preceq x_{nt}} 3\left(1+3\Delta\right) \left(\frac{2k\alpha}{|x|+n a_n^{-1}} + \frac{3k^2\Delta_n}{(|x|+n a_n^{-1})^{\alpha}} \right)  \right\} \\
    &\qquad \leq \exp \left\{  \frac{C_{\Delta} k^2 (na_n^{-1})^{-(1-\alpha)}}{1-\alpha}\left((d_n(O_n, x_{nt})+na_n^{-1})^{1-\alpha} - (d_n(O_n, x_{ns})+na_n^{-1})^{1-\alpha}\right) \right\} \\
    &\qquad \qquad \times \exp \left\{ \frac{C_{\Delta}k\alpha}{1-\alpha} \left( \log ((d_n(O_n, x_{nt})+na_n^{-1})^{1-\alpha}) - \log ((d_n(O_n, x_{ns})+na_n^{-1})^{1-\alpha}) \right) \right\} \\
    &\qquad \le \exp \left\{ \frac{C_{\Delta }k(k+\alpha)}{1-\alpha} \dbtn(s,t)\right\},
    \end{split}
\end{align}

where $\dbtn (s,t)$ is given by \eqref{eqn:def distance discrete}.
Applying Lemma \ref{lem:alpha distance UB}, we deduce that this is upper bounded by 
\begin{align} \label{eqn:recentred mgf**}
\begin{cases}
\exp \left\{C_{\Delta} k(k+\alpha) D_n^{-\alpha} \dbztn(s,t)\right\} &\text{ if } \alpha \leq 0, \\
\exp \left\{\frac{2C_{\Delta} k (k+\alpha)}{1-\alpha} \dbztn(s,t)^{1-\alpha}\right\}   &\text{ if } \alpha\in (0, 1).
\end{cases}
\end{align}
Assume that $\dbztn(s,t) \leq 1$. Set
\[
k=\dbztn(s,t)^{-\frac{(1-\alpha) \wedge 1}{3}} \leq (na_n^{-1})^{\frac{1}{3}}.
\]


In the case $\alpha \leq 0$, we deduce from a Chernoff bound that there exists $C_{\epsilon, \Delta, \alpha} < \infty$ such that on the event $A_{n, \epsilon}$, we have for all $y > 0$ that
\begin{align}\label{eqn:short recentred mgf}
   \pr{ \sum_{x_{ns} \prec x \preceq x_{nt}} (\log \rho_{x} - \E{\log \rho_{x}}) > y} \leq e^{ C_{\Delta}k^3 D_n^{-\alpha}\dbztn(s,t) - y k } 
    \leq e^{C_{\Delta} D_n^{-\alpha}- y k}
    \leq C_{\epsilon, \Delta, \alpha} 
    e^{- y \dbztn(s,t)^{-\frac{(1-\alpha) \wedge 1}{3}} },
\end{align}
where \eqref{eqn:recentred mgf**} was used to provide the first bound in the inequality above. Moreover, by Claim \ref{claim:log mgf} the same result holds on replacing $\rho_x$ by $\rho_x^{-1}$, meaning that we can instead consider the absolute value of the sum in the bound above.

We obtain the same upper bound in the case $\alpha > 0$ (modifying the constant $C_{\epsilon, \Delta, \alpha}$ a bit if necessary). We deduce that, on the event $A_{n, \epsilon}$, for all $p>1$ there exists $C_{p, \epsilon, \Delta,\alpha} < \infty$ such that for all $s,t \in [0,1]$ with $\dbztn(s,t) \leq 1$,
\begin{align}\label{eqn:moment calc}
    \E{\left|  \sum_{x_{ns} \prec x \preceq x_{nt}} (\log \rho_{x} - \E{\log \rho_{x}}) \right|^p} &\leq \int_0^{\infty} \pr{ \sum_{x_{ns} \prec x \preceq x_{nt}} \left|\log \rho_{x} - \E{\log \rho_{x}} \right| >y^{1/p}} dy \nonumber \\
    &\leq \int_{0}^{\infty} 2 C_{\epsilon, \Delta, \alpha} e^{- y^{1/p}\dbztn(s,t)^{-\frac{(1-\alpha) \wedge 1}{3}}} dy \nonumber \\
    &\leq \int_{0}^{\infty} 2 C_{\epsilon, \Delta, \alpha}\dbztn(s,t)^{\frac{(1-\alpha) \wedge 1}{3}}pu^{p-1} e^{- u} du \nonumber \\
    &\leq 2 C_{p, \epsilon, \Delta, \alpha}\dbztn(s,t)^{\frac{(1-\alpha) \wedge 1}{3}}.
\end{align}
By replacing $C_{p,\epsilon, \Delta, \alpha}$ with $D_{\epsilon} C_{p, \epsilon, \Delta, \alpha}$ this extends to the case $\dbtn(s,t) > 1$ on the event $A_{n, \epsilon}$. Finally, this extends to the general case where $2ns$ and $2nt$ are not integers and $x_{2ns}$ is no longer an ancestor of $x_{2nt}$ by breaking paths at the most recent common ancestor and since $M^{(n, \alpha)}$ is defined by interpolation.

We now suppress the dependence on $\Delta$ and $\alpha$ since these are assumed to be fixed. We deduce that, on the event $A_{n, \epsilon}$, we have for all $\gamma' <\frac{\gamma - 1}{\gamma}$ that there exists $C_{p, \epsilon, \gamma'}$ (also applying \eqref{eqn:marzouk holder result}) such that for all $s,t \in [0,1]$,
\begin{align*}
    \E{|M_{2nt}^{(n,\alpha)} - M_{2ns}^{(n,\alpha)}|^p}
    &\leq 2 C_{p, \epsilon}\dbztn(s,t)^{\frac{(1-\alpha) \wedge 1}{3}} \leq C_{p, \epsilon, \gamma'} |s-t|^{\frac{[(1-\alpha) \wedge 1]\gamma'}{3}}.
\end{align*}
This proves the result by taking $p$ large enough.

\textbf{Case $\alpha=1$.} Again we can assume, without loss of generality, that $2ns$ and $2nt$ are integers and that $x_{2ns}$ is an ancestor of $x_{2nt}$. This time, repeating the arguments that led to \eqref{eqn:recentred mgf*} gives that there exists $C_{\Delta}<\infty$ such that
\begin{align*}
    &\E{\exp \left\{\sum_{x_{ns} \prec x \preceq x_{nt}} (\log \rho_{x} - \E{\log \rho_{x}}) \right\}} \\
    &\qquad \leq \exp \left\{\sum_{x_{ns} \prec x \preceq x_{nt}} 3\left(1+3\Delta\right) \frac{k(2+3 k \Delta)}{|x|+n a_n^{-1}}  \right\} \\
    &\qquad \leq \exp \left\{C_{\Delta} k(k+1)\left( \log ((d_n(O_n, x_{nt})+na_n^{-1})) - \log ((d_n(O_n, x_{ns})+na_n^{-1})) \right) \right\} \\
    &\qquad \leq \exp \left\{C_{\Delta} k(k+1) \dbztn(s,t) \right\}.
\end{align*}

Again we assume that $\dbztn(s,t) \leq 1$ and set
\[
k=\dbztn(s,t)^{-\frac{1}{3}} \leq (na_n^{-1})^{\frac{1}{3}}.
\]
As in \eqref{eqn:short recentred mgf}, this implies that there exists $C_{\epsilon, \Delta} < \infty$ such that for all $y>0$,
\begin{align*}
  \pr{\sum_{x_{ns} \prec x \preceq x_{nt}} (\log \rho_{x} - \E{\log \rho_{x}}) >y} \leq C_{\epsilon, \Delta} e^{- y\dbztn(s,t)^{-\frac{1}{3}}}.
\end{align*}
Again by Claim \ref{claim:log mgf} we can replace $\rho_x$ by $\rho_x^{-1}$ in this argument, so we can continue as in \eqref{eqn:moment calc} to deduce that on the event $A_{n, \epsilon}$, we have that for any $p>1$ there exists $C_{p, \epsilon, \Delta}<\infty$ such that for all $s,t \in [0,1]$ with $\dbztn(s,t) \leq 1$,
\begin{align*}
    \E{\left|\sum_{x_{ns} \prec x \preceq x_{nt}} (\log \rho_{x} - \E{\log \rho_{x}})\right|^p} 
    &\leq 2 C_{p, \epsilon, \Delta}\dbztn(s,t)^{\frac{1}{3}}.
\end{align*}
The proof then proceeds as in the case $\alpha<1$.
\end{proof}

\begin{proof}[Proof of Claim \ref{claim:Vbranch}]
We write the proof in the case $\alpha<1$. First set
\[
W^{(n, \alpha)}(2nt) = V^{(n, \alpha)}(2nt) - \frac{\Delta [(\dt(O,t)+1)^{1-\alpha}-1]}{1-\alpha} + \alpha \log \left(\dt(O, t) + 1\right).
\]
Then for every $t \in [0,1]$, 
\[
W^{(n, \alpha)}(2nt) - \E{W^{(n, \alpha)}(2nt)} = 
V^{(n, \alpha)}(2nt) - \E{V^{(n, \alpha)}(2nt)} =
M^{(n,\alpha)}_{2 n t},
\]
so it follows from \eqref{eqn:fd conv} and Proposition \ref{prop:Kolmogorov} that
\begin{equation}\label{eqn:conv summary 1}
\left(W^{(n, \alpha)}(2nt) - \E{W^{(n, \alpha)}(2nt)} \right)_{t \in [0,1]} \to \left(\sqrt{\frac{4\Delta}{1-\alpha}} \phib (t)\right)_{t \in [0,1]},
\end{equation}
uniformly on $C([0,1])$. Also, it follows from Claim \ref{claim:log sum exp} that
\begin{align}\label{eqn:conv summary 2}
\E{W^{(n, \alpha)}(2nt) } \to 0,
\end{align}
uniformly on $C([0,1])$. Claim \ref{claim:Vbranch} therefore follows on combining \eqref{eqn:conv summary 1} and \ref{eqn:conv summary 2}. 

The proof in the case $\alpha=1$ follows similarly.
\end{proof}

\subsection{Convergence of the metric measure spaces} \label{sctn:GHP convergence}

In Section \ref{sctn:potential convergence}, we showed that, almost surely on $\mathbf{\Omega}$, the relevant processes converge in distribution with respect to the Skorohod-$J_1$ topology and therefore with respect to the uniform topology since the limit process is continuous. Therefore, applying Skorohod's Representation theorem again (the space of continuous functions on $[0,1]$ is separable), we will work on a probability space $(\Omega', \mathcal{F}', \mathbb{P}')$ where these functionals additionally converge almost surely with respect to the uniform topology.

\begin{proof} [Proof of Proposition \ref{prop:mm space limit}]

\textbf{\textit{Part 1: convergence of metrics.}}
In the first part of the proof we will show that the distortion of the natural correspondence, defined by 
\begin{equation} \label{distort}
\textsf{dis} (\mathcal{R}_n)=\sup\{|n^{-1} a_n R_n(x_{\lfloor 2 n s\rfloor},x_{\lfloor 2 n t\rfloor})-\Rphi(t,s)|: (x_{\lfloor 2 n s\rfloor},p_{H^{(\gamma)}}(s)), (x_{\lfloor 2 n t\rfloor},p_{H^{(\gamma)}}(t))\in \mathcal{R}_n\},
\end{equation}
converges to 0 as $n\to \infty$, almost surely on $\Omega'$. It is enough to prove this just in the case $t=0$ since
\begin{align*}
|n^{-1} a_n R_n(x_{\lfloor 2n s\rfloor},x_{\lfloor 2 nt \rfloor})-\Rphi(p_{H^{(\gamma)}}(s),p_{H^{(\gamma)}}(t))|
&\le 2 |n^{-1} a_n R_n(x_0,x_{\lfloor 2n r\rfloor})-\Rphi(p_{H^{(\gamma)}}(0),p_{H^{(\gamma)}}(r))|
\\
&\quad +|n^{-1} a_n R_n(x_0,x_{\lfloor 2 n s\rfloor})-\Rphi(p_{H^{(\gamma)}}(0),p_{H^{(\gamma)}}(s))|
\\
&\quad +|n^{-1} a_n R_n(x_0,x_{\lfloor 2 n t\rfloor})-\Rphi(p_{H^{(\gamma)}}(0),p_{H^{(\gamma)}}(t))| \\
&\quad+ 2n^{-1} a_n R_n(x_{r_n},x_{\lfloor 2 n r\rfloor}),
\end{align*}
where $(x_{\lfloor 2nr\rfloor},p_{H^{(\gamma)}}(r))\in \mathcal{R}_n$, where $r\in [s,t]$ is any time between $s$ and $t$ at which the minimum of $H^{(\gamma)}$ is achieved, and where $r_n$ is the index of the most recent common ancestor of $x_{\lfloor 2n s\rfloor}$ and $x_{\lfloor 2 nt \rfloor}$. We will therefore first establish \eqref{distort} when $t=0$ and then show that $n^{-1} a_n R_n(x_{r_n},x_{\lfloor 2 n r\rfloor})\to 0$, uniformly over $s,t \in [0,1]$.

For the former, first note that by the definition in \eqref{eqn:discrete V, res, mu def}:
\begin{align} \label{changeofvar1}
n^{-1} a_n R_n(x_0,x_{\lfloor 2 n s\rfloor})
&=n^{-1} a_n \sum_{x_0\prec x\preceq x_{\lfloor 2 n s\rfloor}}  e^{V^{(n,\alpha)}(x)} 
=
2 a_n \int_{A_s^{(n)}} e^{V^{(n,\alpha)}(2 n r)} d \ell(r),
\end{align}
where $A^{(n)}_s=\{r < s: \inf_{u\in (\lfloor r \rfloor ,s]} C^{(n)}(u) > C^{(n)}( r ) \text{ and } x_{\lfloor 2nr \rfloor} \neq x_0\} \cup \{s\}$ (to avoid double counting repeat contour visits) and $\ell$ is Lebesgue measure on $\mathbb{R}$.

By the definition in \eqref{eqn:distorted res def}:
\begin{align} \label{changeofvar2}
\Rphi(0,s)=
\begin{cases}
\displaystyle \int_{(O, x_s]} (\dt(O,z)+1)^{-\alpha} e^{\sqrt{\frac{4 \Delta}{1-\alpha}} \phib(z)+\frac{\Delta }{1-\alpha} [(\dt(O,z)+1)^{1-\alpha}-1]} \lambda^{\gamma}(\text{d} z)  &\text{ if } \alpha<1,
\\
\displaystyle \int_{(O, x_s]} e^{\sqrt{4 \Delta} \phiz(z)+(\Delta-1) \log(\dt(O,z)+1)} \lambda^{\gamma}(\text{d} z)  &\text{ if } \alpha=1.
\end{cases}
\end{align}
For $r\in (0,1)$, let us set 
\begin{align*}
h^{(n,\alpha)}(r)&=
\displaystyle e^{V^{(n,\alpha)}(2nr)}
\\
h^{(\alpha)}(r)&=
\begin{cases}
\displaystyle (\dt(O,r)+1)^{-\alpha} e^{\sqrt{\frac{4 \Delta}{1-\alpha}} \phib(r)+\frac{\Delta }{1-\alpha} [(\dt(O,r)+1)^{1-\alpha}-1]}, &\text{ if } \alpha<1,
\\
\displaystyle e^{\sqrt{4 \Delta} \phiz(r)+(\Delta-1) \log(\dt(O,r)+1)} , &\text{ if } \alpha=1.
\end{cases}
\end{align*}
Combining \eqref{changeofvar1} and \eqref{changeofvar2}, we deduce that for any pair $(x_{\lfloor 2 n s\rfloor},p_{H^{(\gamma)}}(s))\in \mathcal{R}_n$,
\begin{align} \label{weakconv1}
\left|n^{-1} a_n R_n(x_0,x_{\lfloor 2 n s\rfloor})-R_{\phi}(0,s)\right|
&\le 
\left|2 a_n \int_{A_s^{(n)}} h^{(n,\alpha)}(r) d \ell(r)-2 a_n \int_{A_s^{(n)}} h^{(\alpha)}(r) d \ell(r)\right|
\nonumber \\
&\quad +
\left|2 a_n \int_{A_s^{(n)}} h^{(\alpha)}(r) d \ell(r)-\int_{(O, x_s]} h^{(\alpha)}(z) \lambda^{\gamma}(\text{d} z) \right|
\nonumber \\
&\le 2 a_n ||h^{(n,\alpha)}-h^{(\alpha)}|| \ell(A_s^{(n)})
\nonumber \\
&\quad +||h^{(\alpha)}|| \sup_{s\in [0,1]} \left|2 a_n \int_{A_s^{(n)}} d\ell(r)-\int_{(O, x_s]} \lambda^{\gamma}(\text{d} z)\right|.
\end{align}
There are two steps to show that almost surely on $\Omega'$, the right-hand side converges to 0 uniformly over $s\in [0,1]$. First note that $\ell(A_s^{(n)}) = \frac{1}{2n}d_n(O_n, x_{\lfloor 2ns \rfloor}) = \frac{1}{2a_n}C^{(n)}(s)$. Therefore, almost surely on $\Omega'$,
\begin{equation} \label{weakconv2}
2 a_n ||h^{(n,\alpha)}-h^{(\alpha)}|| \ell(A_s^{(n)})
=
||h^{(n,\alpha)}-h^{(\alpha)}||C^{(n)}(s)\xrightarrow{n\to \infty} 0,
\end{equation}
uniformly over $s\in [0,1]$, by \eqref{contourconv} and \eqref{wlimit} (recall that we are working on a probability space where the convergences of \eqref{contourconv} and \eqref{wlimit} hold almost surely).   Secondly, using \eqref{eqn:length measure prop} it is not hard to see that
\[
\sup_{s\in [0,1]} \left|2 a_n \int_{A_s^{(n)}} d\ell(r)-\int_{(O, x_s]} \lambda^{\gamma}(\text{d} z)\right|=||C^{(n)}-H^{(\gamma)}||,
\]
which also goes to 0 almost surely as $n\to \infty$, by \eqref{contourconv} (which holds almost surely on $\Omega'$).

This therefore shows that each individual term on the right hand side of \eqref{weakconv1} converges to 0 uniformly in $s\in [0,1]$, completing the first part of the proof about the convergence to 0 of the distortion of the correspondence $\mathcal{R}_n$ as defined in \eqref{distort} when $t=0$.

The second part of the proof is to show that
\begin{equation}\label{eqn:Rn common ancestor error}
\sup_{s,t \in [0,1]} n^{-1} a_n R_n(x_{r^{s,t} _n},x_{\lfloor 2 n r^{s,t} \rfloor}) \to 0,
\end{equation}
where $r^{s,t} = s \wedge t$ and $r^{s,t}_n$ is the index of the most recent common ancestor of $x_{\lfloor 2n s\rfloor}$ and $x_{\lfloor 2 nt \rfloor}$. However, note that (writing $r=r^{s,t}$ and $r_n = r^{s,t}_n$)
\begin{align*}
 n^{-1} a_n R_n(x_{r_n},x_{\lfloor 2 n r \rfloor}) &=n^{-1} a_n \sum_{x\in [x_{r_n} ,x_{\lfloor 2 n r\rfloor}] \setminus \{x_{r_n} \wedge x_{\lfloor 2 n r\rfloor}\}}  e^{V^{(n,\alpha)}(x)} \\
 &\leq  (n^{-1} a_n d_n(x_{r_n} ,x_{\lfloor 2 n r\rfloor}) \sup_{x \in T_n}  e^{V^{(n,\alpha)}(x)}.
\end{align*}
By \eqref{contourconv} (which holds almost surely on $\Omega'$), $n^{-1}a_nd_n(x_{r_n} ,x_{\lfloor 2 n r\rfloor})\to 0$ uniformly over $s,t \in [0,1]$, almost surely on $\Omega'$. Similarly, by \eqref{wlimit}, for almost every $\omega \in \Omega'$ we have that $\sup_{x \in T_n}  e^{V^{(n,\alpha)}(x)}$ is bounded by a constant (which may depend on $\omega$, but that can nevertheless be upper bounded on a set of probability $1-\epsilon$ for any $\epsilon > 0$). Therefore, we deduce that \eqref{eqn:Rn common ancestor error} holds almost surely on $\Omega'$, as required.

\textbf{\textit{Part 2: convergence of measures.}}

Recall from \eqref{eqn:discrete V, res, mu def} that for a non-root vertex $x \in T_n$ with $\# x$ offspring,
\begin{align*}
(2n)^{-1} \nu_n(x)=\frac{1}{2n} \sum_{i=0}^{\# x} e^{-V^{(n,\alpha)}(x_i)}
=\frac{1}{2n} \sum_{i=0}^{\# x} e^{-V^{(n,\alpha)}(2 n t_{x_i})},
\end{align*}
where $t_{x_i}$ is the minimal $t$ such that $(x_i,p_{H^{(\gamma)}}(t))\in \mathcal{R}_n$, where $x_0$ denotes the parent of $x$ and $(x_i)_{i=1}^{\# x}$ denotes its children. Therefore, for any set $A_n$ of vertices in $T_n$, we have that 
\[
(2n)^{-1} \nu_n(A_n)=\frac{1}{2n} \sum_{x\in A_n} \sum_{i=0}^{\# x} e^{-V^{(n,\alpha)}(2 n t_{x_i})}.
\]
We introduce an intermediate measure by setting:
\[
\tilde{\nu}_n(x)=\sum_{\substack{0 \le i\le 2n: \\ x_i=x}} e^{-V^{(n,\alpha)}(x_i)} = \begin{cases}
(\deg x)e^{-V^{(n,\alpha)}(x)}, \hspace{1cm} &\text{ if } x \neq O_n, \\
(\deg x + 1)e^{-V^{(n,\alpha)}(x)}, \hspace{1cm} &\text{ if } x = O_n. \\
\end{cases} 
\]
Therefore, for a subset $A_n \subset T_n$,
\[
(2n)^{-1} \tilde{\nu}_n(A_n)=
\frac{1}{2n} \sum_{x\in A_n} \sum_{\substack{0 \le i\le 2n: \\ x_i=x}} e^{-V^{(n,1)}(2 n u_x)},
\]
where $u_x$ is the minimal $u$ such that $(x,p_{H^{(\gamma)}}(u))\in \mathcal{R}_n$. We will now show that it suffices to consider $\tilde{\nu}_n$ in place of $\nu_n$ to obtain the Prokhorov limit. For $s\in (0,1)$, let us set
\[
g^{(n,\alpha)}(s)=
e^{-V^{(n,\alpha)}(2 n s)}.
\]
Note that, if $u_x$ and $t_{x_i}$ are defined as above, then 
\begin{equation}\label{eqn:Deltagn to 0}
\Delta_{g^{(n,\alpha)}} := \sup_{x \in T_n, x_i \sim x} \left\{g^{(n,\alpha)}(u_x) - g^{(n,\alpha)}(t_{x_i})\right\} \to 0,
\end{equation}
by the convergence of \eqref{wlimit} and since $[0,1]$ is compact. Therefore, if $A_n \subset T_n$, then
\begin{align*}
\left|(2n)^{-1} \nu_n(A_n)-(2n)^{-1} \tilde{\nu}(A_n)\right|&= \bigg|\frac{1}{2 n} \sum_{x\in A_n} \sum_{i=0}^{\# x} g^{(n,\alpha)}(t_{x_i})-\frac{1}{2 n} \sum_{x\in A_n} \sum_{i=0}^{\# x} g^{(n,\alpha)}(u_x)\bigg|
\\
&\leq \frac{1}{2 n} \sum_{x\in A_n} \sum_{i=0}^{\# x} \Delta_{g^{(n,\alpha)}}
\\
&=\frac{\Delta_{g^{(n,\alpha)}}}{2 n} \left( \sum_{x\in A_n} \text{deg}(x) + \mathbbm{1}\{O_n \in A_n\}\right) \leq \frac{\Delta_{g^{(n,\alpha)}}}{2 n}\cdot 2 n \to 0
\end{align*}
as $n \to \infty$. Therefore, letting
\[
r_n= \textsf{dis} (\mathcal{R}_n)
\]
as in \eqref{distort}, it is sufficient to show that under the canonical Gromov-Hausdorff embedding $F_n=T_n\sqcup \Tgc$, 
\[
d_P^{F_n}((2n)^{-1} \tilde{\nu}_n,\nu_{\phi})\le r_n,
\]
and hence converges to 0 as $n\to \infty$ (by Part 1 of this proof). We proceed as follows. First, take a subset $B\subset\Tgc$, and for $s \in (0,1)$ set
\[
g^{(\alpha)}(s)=
\begin{cases}
\displaystyle (\dt(O,s)+1)^{\alpha} e^{-\left[\sqrt{\frac{4 \Delta}{1-\alpha}} \phib(s)+\frac{\Delta}{1-\alpha}[(\dt(O,s)+1)^{1-\alpha}-1]\right]}, &\text{ if } \alpha<1,
\\
\displaystyle e^{-[\sqrt{4\Delta} \phiz(s)+(\Delta-1) \log(\dt(O,s)+1)]}, &\text{ if } \alpha=1.
\end{cases}
\]
Note that $||g^{(n,\alpha)}-g^{(\alpha)}|| \to 0$ as $n \to \infty$ by the convergence of \eqref{wlimit} and \eqref{contourconv}. Let $B'=p_{H^{(\gamma)}}^{-1}(B)$, $I_{n,i}=[\frac{i}{2 n},\frac{i+1}{2 n})$ for $0 \leq i < 2n$ and 
\[
B_n=\cup_{0 \leq i < 2n} \{x \in T_n: \exists s\in B' \text{ such that } s\in I_{n,i} \text{ and } x=x_i\},
\]
where $x_i$ corresponds to the $i$-th node of $T_n$ in contour exploration order. Clearly $B'\subset\cup_{i:x_i\in B_n} I_{n,i}$, and so
\begin{align*}
\nu_{\phi}(B)=\int_{B'} g^{(\alpha)}(s) ds&\le \sum_{i:x_i\in B_n} \int_{I_{n,i}} g^{(\alpha)}(s) ds
\\
&\le \sum_{i:x_i\in B_n} \int_{I_{n,i}} |g^{(n,\alpha)}(s)-g^{(\alpha)}(s)|ds+\sum_{i:x_i\in B_n} \int_{I_{n,i}} g^{(n,\alpha)}(s) ds
\\
&\le \sup_{s \in (0,1)} |g^{(n,\alpha)}(s)-g^{(\alpha)}(s)| +\frac{1}{2n}\sum_{x\in B_n}\sum_{0 \leq i \leq 2n} g^{(n,\alpha)}\left( \frac{i}{2n}\right)\mathbbm{1}\{x_i=x\}
\\
&=o(1)+(2n)^{-1} \tilde{\nu}(B_n).
\end{align*}
Now note that if $x \in B_n$, then there exists $i \leq 2n$ and $s\in B'$ with $s\in I_{n,i}$ so that $(x_{\lfloor 2 n s\rfloor},p_{H^{(\gamma)}(s)})\in \mathcal{R}_n$. This further entails that $B_n\subset B^{r_n}$, and so
\begin{equation}\label{eqn:dP first}
\nu_{\phi}(B)\le o(1)+(2n)^{-1} \tilde{\nu}(B^{r_n}). 
\end{equation}
We now prove the reverse statement. Let $A_n$ be a set of vertices in $T_n$, and let
\[
A_n'=\bigcup_{i: x_i\in A_n} I_{n,i}.
\]
Let $A_n''=p_{H^{(\gamma)}}(A_n')$. For any $x\in A_n''$, there exists $t\in A_n'$ with $x=p_{H^{(\gamma)}}(t)$ and $t\in I_{n,i}$ for some $i$ with $x_i\in A_n$. Moreover, this entails that $i=\lfloor 2 n t\rfloor$, and hence $(x_{\lfloor 2 n t\rfloor},p_{H^{(\gamma)}}(t))\in \mathcal{R}_n$. It follows that $x\in A_n^{r_n}$, i.e. $A_n''\subset A_n^{r_n}$; hence it follows from \eqref{eqn:Deltagn to 0} that
\begin{align*}
    (2n)^{-1} \tilde{\nu}(A_n)=\frac{1}{2n} \sum_{x\in A_n}  \sum_{\substack{0\le i\le 2n: \\ x_i=x}} g^{(n,\alpha)}(u_x)
    &=\sum_{x\in A_n}  \sum_{\substack{0 \le i\le 2n: \\ x_i=x}} \int_{I_{n,i}} g^{(n,\alpha)}\left(\frac{1}{2n}\lfloor 2 n s\rfloor\right) ds
    \\
    &=\sup_{s \in (0,1)} |g^{(n,\alpha)}(s)-g^{(\alpha)}(s)|+\int_{A_n'} g^{(\alpha)}(s) ds \\
    &=o(1)+\nu_{\phi}(A_n'')
    \le o(1)+\nu_{\phi}(A_n^{r_n}).
\end{align*}
Together with \eqref{eqn:dP first}, this entails that
\[
d_P^{F_n}((2n)^{-1} \tilde{\nu}_n,\nu_{\phi})\le r_n.
\]
The desired result follows.
\end{proof}
  
  In the following corollary, applying \cite[Theorem 7.2]{croydon2016scaling}, yields the quenched convergence of the LERRW to a mixture of diffusions, with the limit law of this process given by the annealed law $P_{O}$ as defined in \eqref{eqn:annealed X}. 

\begin{corollary}\label{cor:final RW convergence}
Let $X^{(n)}$ have the law of a discrete-time LERRW on $T_n$, with initial weights \eqref{scheme2} and reinforcement parameter $\Delta$. Let $X$ be as defined by \eqref{eqn:annealed X}. Then,
\[
P_{O_n}^{(\alpha^{(n)})}\left(\left(n^{-1} a_n X^{(n)}_{\lfloor 2n^2 a_n^{-1} t \rfloor}\right)_{t\ge 0}\in \cdot \right)\xrightarrow{(d)} P_{O}\left(\left(X_t\right)_{t\ge 0}\in \cdot\right).
\]
\end{corollary}
\begin{proof}
Recall that we are working on a probability space where the convergence of \eqref{contourconv} holds almost surely. It follows that almost surely on this probability space, there exists a canonical embedding into the common metric space $n^{-1} a_n T_n\sqcup \Tgc $ in which the GHP distance between the random elements in Proposition \ref{prop:mm space limit} goes to zero as $n \to \infty$. We work pointwise on this probability space so that we only need to consider the randomness of the Dirichlet weights.

We first consider a continuous-time LERRW with $\textsf{exp}(1)$ holding times at each vertex. By our choice of metric and measure, it follows that the stochastic process associated with the discrete space $(T_n, R_n, \nu_n)$ has the quenched law of a RWDE also with $\textsf{exp}(1)$ holding times at each vertex. Therefore, by Lemma \ref{lem:Dirichlet rep of LERRW} the continuous-time LERRW corresponds to the annealed law of this RWDE exactly as considered in \eqref{lerrwannealed}. It therefore follows directly from \cite[Theorem 7.2]{croydon2016scaling} (the non-explosion condition (39) appearing there is satisfied since all the spaces are compact) that this continuous-time LERRW converges in distribution to $X$ as defined by \eqref{eqn:annealed X}. The extension to the discrete-time LERRW in place of the continuous-time one then follows by a straightforward application of the law of large numbers on the time index.
\end{proof}

Clearly Theorem \ref{thm:scaling lim intro} follows directly from Corollary \ref{cor:final RW convergence}.

\section{Properties of $\Tg$ and its Gaussian potential}\label{sctn:Tgamma props}

In order to understand the long-time behaviour of a LERRW it is more natural to consider it on an infinite critical Galton-Watson tree $\Ti$ as introduced in Section \ref{sctn:infinite trees bckgrnd}. Analogously to \cref{sctn:LERRW and RWRE connection trees}, one can define a RWRE on $\Ti$ determined by $(\rho_v)_{v \in \Ti}, V, R$ and $\nu$ on $\Ti$, as defined in \eqref{eqn:discrete V, res, mu def}. Once again, the law of the LERRW can be represented as a RWRE as in \cref{lem:Dirichlet rep of LERRW}.

With such offspring distribution as in \eqref{eqn:dom of att def}, Kesten's tree $\Ti$ is well-known to converge to a non-compact version of a stable tree coded by two independent \Levy excursions and an appropriate immigration measure under rescaling \cite{DuqSinTree}; these are known as stable sin-trees and we will denote them by $\Tg$ for $\gamma \in (1,2)$ (in fact due to uniform re-rooting invariance one can also construct the same objects with only the \Levy processes, but we do not explore this here). Analogously to $\Tgc$, one can define a canonical metric and a measure on $\Tg$ using the canonical projection from the coding \Levy processes, and additionally define a snake process $\phib$ satisfying all the properties of \cref{def:snake}. Given these, we can define a resistance metric $\Rphi$ and a measure $\nuphi$ on $\Tg$ exactly as in \eqref{eqn:distorted res def} and \eqref{eqn:distorted meas def}. (To keep the notation light, we only write ``$\infty$'' explicitly on the spaces $\Ti$ and $\Tg$ and not on all the metrics and measures).

For non-compact mm-spaces, GHP convergence naturally extends to Gromov-Hausdorff-vague (GHv) convergence which is equivalent to GHP convergence of metric balls $B(O_n, r) \to B(O, r)$ for almost every ${r > 0}$ \cite[Definition 5.8]{athreya2016gap}. The results of \cref{prop:mm space limit} extend straightforwardly to these balls and therefore we deduce the following proposition. The proofs to extend to the infinite-dimensional setting are not illuminating, but can be carried out using the same strategy as in the proof of \cite[Theorem 1.2]{archer2020infinite}.

\begin{prop} \label{prop:mm space limit inf}
Let $\Ti$ be Kesten's tree as defined in \cref{def:Kesten's tree}, with offspring distribution satisfying \eqref{eqn:dom of att def}. As $n\to \infty$, with initial weights as in \eqref{scheme2} and reinforcement parameter $\Delta$,
\[
(\Ti, n^{-1}a_n R_n, (2n)^{-1} \nu_n,O_{\infty}) \overset{(d)}{\to} (\Tg, \Rphi, \nu_{\phi}, O)
\]
in the GHv topology.
\end{prop}

Exactly as in Corollary \ref{cor:final RW convergence}, it has the following immediate corollary, by \cite[Theorem 7.2]{croydon2016scaling}. Note that condition (40) of \cite[Theorem 7.2]{croydon2016scaling} is satisfied as a consequence of the unique spine to infinity.

\begin{corollary}\label{cor:LERRW scaling lim infinite}
Let $X^{(n)}$ be the discrete-time LERRW on $\Ti$, with initial weights as in \eqref{scheme2} and reinforcement parameter $\Delta$. Let $X$ be the diffusion on $\Tg$ defined analogously to \eqref{eqn:annealed X}. Then
\[
\left(n^{-1}a_n X^{(n)}_{\lfloor 2n^2 a_n^{-1} t \rfloor}\right)_{t\ge 0}\xrightarrow{(d)} \left(X_t\right)_{t\ge 0}.
\]
\end{corollary}

\cref{cor:LERRW scaling lim infinite} shows that the quenched scaling limit of LERRW, which we denote $(X_t)_{t \geq 0}$, can be represented as a mixture of diffusions in random environments parametrised by different realisations of the functional $\phib$. In particular, for a fixed realisation of $\Tg$ we do not have a pointwise correspondence between realisations of the law of $\phib$ and realisations of $(X_t)_{t \geq 0}$, but instead we must average over $\phib$ to get the equality in distribution. Nevertheless, we can transfer almost sure results for $\phib$ directly to $(X_t)_{t \geq 0}$ to prove Theorems \ref{thm:almost sure limsup} and \ref{thm:almost sure limsup critical}. In this section we establish some almost sure properties of $\Tg$ and $\phib$. For some of the following propositions, the stated results are well-known but we could not find them explicitly written in the literature (for example Proposition \ref{prop:CRT vol growth annulus}). Therefore, we have provided an outline of the proofs but omitted some details (which would be long to justify and somewhat tangential to the main purpose of this paper). 

Throughout this section, the notation $d$ and $\mu$ refer to the measure and metric on the non-compact tree $\Tg$.

\begin{prop}\label{prop:CRT vol growth}
$\bPb$-almost surely, for any $\epsilon>0$,
\[
\limsup_{r \to \infty} \left( \frac{\mu(B(O, r))}{r^{\frac{\gamma}{\gamma -1}} (\log r)^{\frac{1+\epsilon}{\gamma -1}}} \right) = 0, \hspace{5mm} \liminf_{r \downarrow 0} \left( \frac{\mu(B(\rho, r))}{r^{\frac{\gamma}{\gamma -1}} (\log \log r^{-1})^{\frac{1}{\gamma -1}}} \right) = \gamma - 1.
\]
\end{prop}
\begin{proof}
Set $r_m = 2^m$. By monotone convergence (applied twice), scaling invariance of $\Tg$, then monotone convergence again,
\begin{align*}
\prb{\limsup_{m \to \infty} \left( \frac{\mu(B(O, r_m))}{(r_m)^{\frac{\gamma}{\gamma -1}} (\log r_m)^{\frac{1+\epsilon}{\gamma -1}}} \right)\leq 1} &= 
\prb{\lim_{k \to \infty} \lim_{K \to \infty}\sup_{k \leq m \leq K} \left( \frac{\mu(B(O, r_m))}{(r_m)^{\frac{\gamma}{\gamma -1}} (\log r_m)^{\frac{1+\epsilon}{\gamma -1}}} \right)\leq 1} \\
&= \lim_{k \to \infty} \lim_{K \to \infty}\prb{\sup_{k \leq m \leq K} \left( \frac{\mu(B(O, r_m))}{(r_m)^{\frac{\gamma}{\gamma -1}} (\log r_m)^{\frac{1+\epsilon}{\gamma -1}}} \right) \leq 1} \\
&= \lim_{k \to \infty} \lim_{K \to \infty} \prb{\sup_{k \leq m \leq K} \left( \frac{\mu(B(O, r_m^{-1}))}{(r_m^{-1})^{\frac{\gamma}{\gamma -1}} (\log r_m)^{\frac{1+\epsilon}{\gamma -1}}} \right) \leq 1} \\
&=\prb{\limsup_{m \to \infty} \left( \frac{\mu(B(O, r_m^{-1}))}{(r_m^{-1})^{\frac{\gamma}{\gamma -1}} (\log r_m)^{\frac{1+\epsilon}{\gamma -1}}} \right) \leq 1}.
\end{align*}

By \cite[Theorem 1.4]{duquesnelegallhausdorff2006}, the final probability on the RHS is 1 when considered on the compact stable tree $\Tgc$ instead of $\Tg$. Combining again with scaling invariance, it therefore follows from \cite[Theorem 1.3]{DuqSinTree} (which shows that $\Tg$ is a local limit of compact stable trees) that this probability is also $1$ on $\Tg$. Finally, since $\epsilon > 0$ was arbitrary we can replace $1$ with and $\delta>0$ and then with $0$, and extend more generally to $r \to 0$ since $\mu (B(O, r)) \leq \mu (B(O, 2^{\lceil \log_2 r \rceil})$.

The second statement follows by the same proof using \cite[Proposition 1.1]{duquesne2011small} for $\Tgc$.
\end{proof}

It follows from \cite[Proposition 1.1]{DuqSinTree} that we can define a local time measure $(L^{(r)})_{r > 0}$ on $\Tg$ such that for any compactly supported continuous function $g:(0, \infty) \to \R$,
\begin{equation}\label{eqn:local time cts function def}
\int_{\Tg} g(\dt(O, x)) \mu (dx) = \int_0^{\infty} L^{(r)}g(r) dr.
\end{equation}

The following result is a direct consequence of \cite[Theorem 1.5]{duquesnelegallhausdorff2006} and \cite[Theorem 8.8]{mattila1999geometry}. (In fact the result of \cite[Theorem 1.5]{duquesnelegallhausdorff2006} is for the limit $r \downarrow 0$ on the compact $\Tgc$, but transfers to the limit $r \to \infty$ on $\Tg$ exactly as in the proof of Proposition \ref{prop:CRT vol growth}).

\begin{prop}\label{prop:local time limsup}
For any $\epsilon>0$,
\begin{equation*}
\limsup_{r \to \infty} \left( \frac{L^{(r)}}{r^{\frac{1}{\gamma -1}} (\log r)^{\frac{1+\epsilon}{\gamma-1}}} \right) = 0.
\end{equation*}
\end{prop}

Similarly, we have the following.

\begin{prop}\label{prop:CRT vol growth annulus}
$\bPb$-almost surely, for any $\epsilon>0$,
\[
\limsup_{r \to \infty} \left( \frac{\mu(B(O, r+1) \setminus B(O, r))}{r^{\frac{1}{\gamma -1}} (\log r)^{\frac{1+\epsilon}{\gamma - 1}}} \right) = 0. 
\]
\end{prop}
\begin{proof}

It follows from \eqref{eqn:local time cts function def} (and continuous approximation) that 
\[
\mu(B(O, r+1) \setminus B(O, r)) = \int_{r}^{r+1}L^{(s)} ds,
\]
for all $r$ almost surely. The result therefore follows directly from Proposition \ref{prop:local time limsup} on integrating $L^{(s)}$ over $s \in [r, r+1]$.

\end{proof}

The construction of $\Tg$ in \cite{DuqSinTree} codes $\Tg$ by two \Levy processes plus two stable subordinators that represent immigration. The result of \cite[Proposition 1.3]{DuqSinTree} shows that $\Tg$ is also the local limit of compact stable trees conditioned on their height going to infinity. Since compact stable trees satisfy the property of uniform re-rooting invariance, it is also possible to consider them to be coded by two halves of a \Levy bridge, rather than a \Levy excursion (by applying the Vervaat transform to the \Levy excursion that codes them). By taking a limit using this bridge representation, it follows that $\Tg$ can in fact be constructed solely by two \Levy processes, by imagining that the time $0$ corresponds to the ``tip'' rather than the ``base'' of $\Tg$. The argument proceeds exactly the same as in the proof in \cite[Section 5.1]{archer2020infinite}; we do not repeat it here, but just give the construction.

\vspace{.5cm}

\begin{tcolorbox}[colback=white]\label{box:Tg two sided construction}
\textbf{Construction of $\Tg$}
\begin{enumerate}
\item Let $X$ and $X'$ be independent $\gamma$-stable, spectrally positive \Levy processes.
\item Define a function $\Xi: \R \rightarrow \R$ by
\[
\Xi_t =
\begin{cases} X_{t} & \text{ if } t \geq 0, \\
-X'_{-t^-} & \text{ if } t<0.
\end{cases}
\]
\item Given $s < t$, let $I_{s,t} = \inf_{r \in [s,t)} \Xi_r$. Say that $s \preceq t$ if $\Xi_{s-} \leq I_{s,t}$ and that $s \prec t$ if $s \preceq t$ and $s \neq t$. Also set $s \wedge t = \sup \{r \in \R: r \preceq t \text{ and } r \preceq s\}$.
\item If $s \preceq t$ set
\begin{align*}
\dt(s,t) &= \lim_{\epsilon \downarrow 0} \epsilon^{-1} \int_{s}^{t} \mathbbm{1}\{X_r \leq I_{s,t} + \epsilon\} dr. 
\end{align*}
For all $s,t$ this limit exists in probability, see \cite[Lemma 1.1.3]{duquesne2002asterisque}.
Then, for general $s, t \in \R$, set 
\[
\dt(s,t)= \dt(s \wedge t, s) + \dt(s \wedge t, t).
\]
Finally, define an equivalence relation $\sim$ on $\R$ by setting $s \sim t$ if and only if $\dt(s,t)=0$. We define
\begin{align*}
\Tg &= (\R / \sim, \dt)
\end{align*}
and denote by $\pi$ the canonical projection $\R \to \Tg$, and set the root to be equal to $\pi (0)$.
\end{enumerate}
\end{tcolorbox}

\vspace{.5cm}

We will use this construction and in particular excursion theory for the \Levy processes to prove some results about $\Tg$. (This is more convenient than using the construction of \cite{DuqSinTree} since we don't have to deal with the immigration). In what follows we will let $N(\cdot)$ denote the \Ito excursion measure for $X$. This can be thought of as the ``law'' of an excursion of $X$, though it is not normalisable (see \cite[Section 3.1.1]{curien2014loop} for a concise explanation).

For a subset $A \subset \Tg$ and a pseudometric $\tilde{d}$ on $A$, let $D(A,\epsilon, \tilde{d})$ denote the $\epsilon$-packing number of $(A, \tilde{d})$; in other words, $D(A,\epsilon, \tilde{d})$ is the maximal size of a collection of points $(x_{i})_{i \leq D(A, \epsilon, \tilde{d})}$ contained in $A$ such that $\tilde{d}(x_i, x_j) \geq \epsilon$ for all $i \neq j$.

\begin{prop}\label{prop:sqrt packing number}
For any $p>\frac{2\gamma}{\gamma - 1}$, there exists a constant $C_p < \infty$ such that for any $r, t > 0$,
\begin{align*}
    \Eb{D(B(O, r), t, \dt)^{\frac{1}{p}}} \leq C_p(r t^{-1})^{\frac{\gamma}{\gamma - 1} (1+ \frac{1}{p})} \vee 1.
\end{align*}
\end{prop}
\begin{remark}
This is not an optimal result, but is sufficient for our purposes.
\end{remark}
\begin{proof}[Proof of Proposition \ref{prop:sqrt packing number}]
Without loss of generality we can assume that $r>t$; otherwise $D(B(O, r), t, \dt)=1$. By scaling invariance, it is sufficient to show the result when $r=1$.

We claim the following: for any $x \in \R$, $\delta > 0$, there exist $C<\infty$, $c>0$ such that
\begin{align}\label{eqn:diam prob bound}
\begin{split}
\prb{B(O, 1) \not\subset \pi ([-\lambda, \lambda])} &\leq C\lambda^{-\frac{(1-\delta)(\gamma -1)}{\gamma}},\\
\prb{\diam (\pi([x - t^{\frac{\gamma}{\gamma - 1}}\lambda^{-1}, x])) > t} &\leq e^{-c \lambda^{\frac{\gamma - 1}{\gamma^2}}}.
\end{split}
\end{align}
The result then follows since if $B(O, 1) \subset \pi([-\lambda, \lambda])$, we can divide the interval $[-\lambda, \lambda]$ up into $\lceil 2t^{\frac{-\gamma}{\gamma - 1}}\lambda^{2} \rceil$ covering intervals of length $t^{\frac{\gamma}{\gamma - 1}}\lambda^{-1}$, and with high probability the diameter of each of these intervals is at most $t$. In particular, the probability that this does not happen is upper bounded by 
\[
C\lambda^{-\frac{(1-\delta)(\gamma -1)}{\gamma}} + 4t^{\frac{-\gamma}{\gamma - 1}}\lambda^{2}e^{-c \lambda^{\frac{\gamma - 1}{\gamma^2}}}.
\]
Any $t$-packing can have at most one point in each of these intervals (in fact, we have bounded the covering number), so we deduce that 
\begin{align}\label{eqn:packing number tail bound}
\prb{D(B(O, 1), t, d) \geq 4t^{-\frac{\gamma}{\gamma - 1}}\lambda^{2}} \leq C\lambda^{-\frac{(1-\delta)(\gamma -1)}{\gamma}} + 4t^{\frac{-\gamma}{\gamma - 1}}\lambda^{2}e^{-c \lambda^{\frac{\gamma - 1}{\gamma^2}}}.
\end{align}
The result then follows by writing
\begin{align*}
    \Eb{D(B(O, r), t, \dt)^{\frac{1}{p}}} = \int_0^{\infty} \prb{D(B(O, r), t, \dt) \geq x^p} dx
\end{align*}
and performing the appropriate change of variables to apply the tail bound of \eqref{eqn:packing number tail bound}.

By scaling invariance, for the general claim we just need to replace $t$ with $r^{-1}t$, so we see that the claim holds for any $p > \frac{2\gamma}{\gamma - 1}$.

We therefore just need to prove \eqref{eqn:diam prob bound}. We will use excursion theory for the \Levy process $X'$ coding the left side of $\Tg$ (i.e. on the negative real line). We therefore let $\overline{X'}_t = \sup_{s \leq t} X'_s$ denote the running supremum process of $X'$, and let $(L(t))_{t \geq 0}$ denote the local time of $\overline{X'}-X'$ at zero, normalised so that $\Eb{\exp \{-\lambda \overline{X'}_{L^{-1}(t)}\}} = e^{-t\lambda^{\gamma - 1}}$ (this is well-defined, e.g. \cite[Section VIII]{BertoinLevy}). Moreover, we have by \cite[Section VIII, Lemma 1]{BertoinLevy} that $L^{-1}$ is a stable subordinator of index $1-\frac{1}{\gamma}$, and \cite[Proposition 3.1$(ii)$]{curien2014loop} that the measure $\sum_{\overline{X'}_s > \overline{X'}_{s^-}} \delta (L(s), \Delta_s(X'))$ is a Poisson point measure with intensity $dl \cdot C_{\gamma} x^{-\gamma} dx$.

\textbf{First bound.} To prove the first statement, first note that $X^{\infty}_{-t} = -X_{t^-}'$ for $t \geq 0$. It follows that new suprema of $X'$ correspond to backwards minima of $X^{\infty}$ from $t=0$. Therefore, $L = \inf \{s \geq 0: L(s) \geq 1\}$ corresponds to the ancestor of $0$ on the infinite backbone with $d(O, \pi(L)) = 1$, and setting $R= \inf\{t \geq 0: X_t \leq -X'_{L}\}$ it follows that 
\[
B(O, 1) \subset \pi([-L,R]).
\]
Let $S$ denote a stable subordinator with \Levy measure $C_{\gamma} x^{-\gamma} dx$. We then have that
\begin{enumerate}
\item $\prb{L > \lambda} = \prb{L(\lambda) \leq 1} = \prb{L^{-1}(1) \geq \lambda} \leq c\lambda^{\frac{-(\gamma - 1)}{\gamma}}$.
\item  By \cite[Section VIII, Propositions 3 and 4]{BertoinLevy},
\begin{align*}
\prb{R> \lambda} \leq \prb{X_{L}' \geq \lambda^{\frac{1-\delta}{\gamma}}} + \prb{\inf_{s \leq \lambda} X_s \geq -\lambda^{\frac{1-\delta}{\gamma}}} &\leq \prb{S_1 \geq \lambda^{\frac{1-\delta}{\gamma}}} + \prb{\inf_{s \leq 1} X_s \geq -\lambda^{\frac{-\delta}{\gamma}}} \\
&\leq C\lambda^{-\frac{(1-\delta)(\gamma -1)}{\gamma}} + Ce^{-c\lambda^{\delta}}.
\end{align*}
\end{enumerate}
This establishes the first bound in \eqref{eqn:diam prob bound}.

\textbf{Second bound.} We will prove the diameter bound for $x=0$; the proof is the same for arbitrary $x$. Set $x_t = \inf\{s \geq 0: d(0, -s) > t\}$. $\bPb$-almost surely, there is a unique path  $\Gamma_t$ from $O$ to $\pi(x_t)$ in $\Tg$ of length exactly $t$. Moreover, the interval $[-x_t, 0]$ codes all the subtrees grafted to one side of this path. Each of these complete subtrees are coded by the \Ito excursion measure $N$, and moreover, $x_t$ is equal to the sum of the lengths of all of these \Ito excursions. In fact, the subtrees grafted to this side of the path $\Gamma_t$ form a Poisson point process on this path. Let $M_{t, \lambda}$ denote the number of subtrees of lifetime at least $t^{\frac{\gamma}{\gamma - 1}}\lambda^{-1}$ grafted to $\Gamma_t$ at a point within distance $\frac{t}{2}$ of $O$. The subtrees grafted to $\Gamma_t$ are concentrated in groups at certain hubs along $\Gamma_t$, and it follows from \cite[Corollary 1]{BertoinPitmanExt} that $M_{t, \lambda}$ stochastically dominates a Poisson random variable with parameter $S_{\frac{t}{2}}N(\sigma > t^{\frac{\gamma}{\gamma - 1}}\lambda^{-1}, H<\frac{t}{2})$, where $S$ is a $(\gamma - 1)$-stable subordinator by \cite[Proposition 5.6]{goldschmidt2010extinctionstable} (here $\sigma$ denotes the lifetime of a \Levy excursion, and $H$ the height of a tree it codes). Moreover, since
\[
\prb{S_{t/2} \leq ct^{\frac{1}{\gamma - 1}}\lambda^{-p}} \leq e\Eb{e^{-c^{-1}\lambda^{p}S_1}} \leq e^{1-c'\lambda^{p(\gamma - 1)}} \text{   and    } N(\sigma > t^{\frac{\gamma}{\gamma - 1}}\lambda^{-1}, H<t) \geq ct^{-\frac{1}{\gamma - 1}}\lambda^{\frac{1}{\gamma}} - c't^{-\frac{1}{\gamma - 1}},
\]
this parameter is lower bounded by $c\lambda^{\frac{1}{\gamma}-p}$ with probability at least $1-Ce^{-c\lambda^{p(\gamma - 1)}}$. Therefore,
\[
\prb{M_{t, \lambda}=0} \leq Ce^{-c\lambda^{p(\gamma - 1)}} + e^{-c \lambda^{\frac{1}{\gamma}-p}}.
\]
On the event $M_{t, \lambda}>0$, it follows that $x_t \geq t^{\frac{\gamma}{\gamma - 1}}\lambda^{-1}$, so taking $p=\gamma^{-2}$, we deduce the second result of \eqref{eqn:diam prob bound}, which completes the proof.
\end{proof}

We can extend the definitions of $\dbt$ and $\phib$ from \cref{sctn:scaling limit candidate} to the infinite tree $\Tgc$. Just as we did there, for $\sigma, \sigma' \in \Tg$, we set
\[
\dbt(\sigma,\sigma') =
\begin{cases}
(\dt(O, \sigma)+1)^{1-\alpha} + (\dt(O,\sigma')+1)^{1-\alpha} - 2(\dt(O, \sigma\wedge \sigma')+1)^{1-\alpha}, &\text{ if } \alpha < 1, \\
\log (\dt(O, \sigma)+1) + \log (\dt(O, \sigma')+1) - 2\log (\dt(O, \sigma \wedge \sigma')+1) &\text{ if } \alpha = 1.
\end{cases}
\] 
We also define $\phib$ exactly as in \cref{def:snake}, but on $\Tg$ instead of $\Tgc$. This means that $\phib$ is a centred Gaussian process and that for all $s, t \in \mathbb{R}$, 
\[
\mathbb{E}{[|\phib(s) - \phib(t)|^2]} = \dbt(s,t).
\]
These are well-defined on $\Tg$ by the same considerations as in the compact case. We now transfer the result of Proposition \ref{prop:sqrt packing number} to the metric $\dbt$.

\begin{corollary}\label{cor:packing number bound}
For any $\alpha < 1$ and any $\eta>0$, there exists $P>0$ such that for all $p \geq P$ there exists $\tilde{C}_p< \infty$ such that
\begin{align*}
\Eb{D\left(B(O, r), \delta, \sqrt{\dbt}\right)^{\frac{1}{p}}} \leq \tilde{C}_p((r+1)^{1-\alpha} \delta^{-2})^{\eta} \vee 1.
\end{align*}
\end{corollary}
\begin{proof}

\textbf{Case 1: $0<\alpha<1$.} Note that, since $1-\alpha \in (0,1)$ in this case, we can assume that $r>\left(\frac{\delta^2}{2}\right)^{\frac{1}{1-\alpha}}$, otherwise $D\left(B(O, r), \delta, \sqrt{\dbt}\right) =1$.

Firstly, suppose that $p > \frac{2 \gamma q}{\gamma - 1}$ for some $q>1$. Then, by Jensen's inequality and Proposition \ref{prop:sqrt packing number}, there exists $\epsilon > 0$ such that
\begin{align}\label{eqn:packing number small power}
\Eb{D(B(O, r), t, \dt)^{\frac{1}{p}}} \leq \Eb{D(B(O, r), t, \dt)^{\frac{1}{\frac{2 \gamma }{\gamma - 1}+ \epsilon}}}^{\frac{1}{q}} \leq \left( C_p (r t^{-1})^{\frac{\gamma}{\gamma - 1} + \frac{1}{2}} \right)^{\frac{1}{q}}.
\end{align}
We now set $B=B(O, r)$, measured with respect to the metric $\dt$, and observe the following: if $x,y \in B$, then $\left(\frac{\dbt(x,y)}{2}\right)^{\frac{1}{1-\alpha}} \leq \dt(x,y)$. Therefore, for any $\delta > 0$, a $\delta$-packing of $B$ with respect to $\dbt$ is a $\left(\frac{\delta}{2}\right)^{\frac{1}{1-\alpha}}$-packing of $B$ with respect to $\dt$, so that
\[
D\left(B, \delta, \sqrt{\dbt}\right) \leq  D\left(B, \left(\frac{\delta^2}{2}\right)^{\frac{1}{1-\alpha}}, \dt\right).
\]
Taking $p$ and $q$ as above and combining with \eqref{eqn:packing number small power}, we therefore deduce that
\begin{align*}
\Eb{D\left(B, \delta, \sqrt{\dbt}\right)^{\frac{1}{p}}} \leq \Eb{D\left(B,\left(\frac{\delta^2}{2}\right)^{\frac{1}{1-\alpha}}, \dt\right)^{\frac{1}{p}}} \leq \tilde{C}_{p, \gamma} (r \delta^{\frac{-2}{1-\alpha}})^{\frac{1}{q}\left(\frac{\gamma}{\gamma - 1} + \frac{1}{2}\right)}.
\end{align*}
In particular, we can choose $p$ and therefore $q$ large enough so that this final exponent is less than $\eta$, which proves the result.

\textbf{Case 2: $\alpha < 0$.}
Again set $B=B(O, r)$, measured with respect to the metric $\dt$, and now observe the following: if $x,y \in B$, then $\dbt(x,y) \leq (1-\alpha)(r+1)^{|\alpha|} \dt(x,y)$. 
Therefore, for any $\delta > 0$, a $\delta$-packing of $B$ with respect to $\dbt$ is a $\frac{\delta}{(1-\alpha)(r+1)^{|\alpha|}}$-packing of $B$ with respect to $\dt$, so that
\[
D\left(B, \delta, \sqrt{\dbt}\right) \leq 
D\left(B, \frac{\delta^2}{(\alpha-1)(r+1)^{|\alpha|}}, \dt\right).
\]
Taking $p$ and $q$ as above and combining with \eqref{eqn:packing number small power}, we therefore deduce that
\begin{align*}
\Eb{D\left(B, \delta, \sqrt{\dbt}\right)^{\frac{1}{p}}} \leq \Eb{D\left(B, \frac{\delta^2}{(1-\alpha)(r+1)^{|\alpha|}}, \dt\right)^{\frac{1}{p}}} \leq \tilde{C}_p ((1-\alpha)(r+1)^{1-\alpha} \delta^{-2})^{\frac{1}{q}\left(\frac{\gamma}{\gamma - 1} + \frac{1}{2}\right)}.
\end{align*}
In particular, we can choose $p$ and $q$ large enough so that this final exponent is less than $\eta$, which proves the result.
\end{proof}

In the critical case the packing number instead grows logarithmically in $r$, as we see in the following proposition.

\begin{prop}\label{prop:sqrt packing number d0}
Take $\alpha = 1$. For any $p>\frac{1}{\gamma - 1}$ and any $\epsilon >0$, there exists an event $A_{\epsilon}$ with $\prb{A^c_{\epsilon}} < \epsilon$ and a constant $c_{\gamma, p, \epsilon} < \infty$ such that for any $r > 0$,
\begin{align*}
    \Eb{D(B(O, r), \delta, \sqrt{\dzt})^{\frac{1}{p}}\mathbbm{1}\{A_{\epsilon}\}} \leq c_{\gamma, p, \epsilon}\delta^{-\frac{2}{p}}(\lceil (\log (r+1))\rceil)^{\frac{\gamma +\epsilon}{p(\gamma-1)}}.
\end{align*}
\end{prop}
\begin{proof}
Given $\delta > 0$, we construct a $\delta$-covering of $B(O, r)$ with respect to the metric $\dzt$, the size of which gives an upper bound for the $\delta$-packing number. In particular, for $i \geq 0$ we let $S_{\delta}^{(i)}$ denote the set of vertices at distance $e^{\frac{i\delta}{2}}-1$ from the root that have descendants at distance $e^{\frac{(i+1)\delta}{2}}-1$ from the root. We claim that $\cup_{i = 0}^{\lceil2\delta^{-1} (\log (r+1))\rceil} S_{\delta}^{(i)}$ is a $\delta$-covering of $B(O, r)$ with respect to the metric $\dzt$. Indeed, if $x \in B(O, r)$ with $\dt(O, x) \geq e^{\delta}-1$ then set $i_x = \lfloor 2\delta^{-1} \log (\dt(O, x)+1) \rfloor$, so that $i_x$ is the maximal $i$ such that $e^{\frac{i\delta}{2}}-1 \leq d(O, x)$.  Note that $i_x \geq 2$ since we are assuming that $\dt(O, x) \geq e^{\delta}-1$. Then let $A_x$ be the ancestor of $x$ at distance $e^{\frac{\delta}{2}i_x} - 1$ from the root, and let $A_x'$ be the ancestor of $x$ at distance $e^{\frac{\delta}{2}(i_x-1)} - 1$ from the root. Then, since $d(O, x) \leq e^{\frac{\delta}{2}(i_x+1)} - 1$, we have that
\[
\dzt(A_x', x) \leq \dzt(A_x, A_x') + \dzt(A_x, x) \leq \delta,
\]
which implies that $x \in S_{\delta}^{(i_x -1)}$. If instead $\dt(O, x) \leq e^{\delta}-1$ then $x \in S_{\delta}^{(0)}$.

We therefore turn to bounding $|S_{\delta}^{(i)}|$ for each $i$. Recall that the size of generation $e^{\frac{i\delta}{2}}-1$ can be formally measured by the local time measure of \eqref{eqn:local time cts function def}. Moreover, conditionally on the total local time measure at level $e^{\frac{i\delta}{2}}-1$, denoted $L^{(e^{\frac{i\delta}{2}}-1)}$, being equal to $l^{(i)}$, it follows from \cite[Proposition 3.1$(ii)$]{curien2014loop} that the number of subtrees emanating from level $e^{\frac{i\delta}{2}}-1$ that reach level $e^{\frac{(i+1)\delta}{2}}-1$ is one more than a Poisson random variable with parameter $2l^{(i)} N(H>e^{\frac{i\delta}{2}}(e^{\frac{\delta}{2}}-1))$, where $N$ is the \Ito excursion measure, so that $N(H>t) = c_{\gamma} t^{-\frac{1}{\gamma - 1}}$ (see \cite[Proposition 5.6]{goldschmidt2010extinctionstable}).

By \cref{prop:local time limsup}, it follows that we can choose $K_{\epsilon}<\infty$ such that $A_{\epsilon} =\{\forall r > 0: L^{(r)} \leq K_{\epsilon}r^{\frac{1}{\gamma -1}} (\log r)^{\frac{1+\epsilon}{\gamma-1}}\}$ satisfies $\prb{A^c_{\epsilon}} < \epsilon$. On this event, $l^{(i)} \leq K_{\epsilon}e^{\frac{i\delta}{2(\gamma - 1)}}(i\delta)^{\frac{1+\epsilon}{\gamma-1}}$ for all $i \geq 0$, so that 
\[
|S^{(i)}| \overset{s.d.}{\preceq} \textsf{Poisson}\left( c_{\gamma}K_{\epsilon}e^{\frac{i\delta}{2(\gamma - 1)}}(i\delta)^{\frac{1+\epsilon}{\gamma-1}} \left(e^{\frac{i\delta}{2}}(e^{\frac{\delta}{2}}-1)\right)^{-\frac{1}{\gamma - 1}} \right) =  \textsf{Poisson}\left( c_{\gamma}K_{\epsilon}(i\delta)^{\frac{1+\epsilon}{\gamma-1}} (e^{\frac{\delta}{2}}-1)^{-\frac{1}{\gamma - 1}} \right),
\]
where the symbol $\overset{s.d.}{\preceq}$ denotes stochastic domination between random variables. Therefore, if $p>1$ and $\delta>0$ is sufficiently small,
\begin{align*}
\Eb{\left(\sum_{i = 0}^{\lceil2\delta^{-1} \log (r+1)\rceil} |S_{\delta}^{(i)}|\right)^{\frac{1}{p}}} \leq \Eb{\sum_{i = 0}^{\lceil2\delta^{-1} (\log (r+1))\rceil} |S_{\delta}^{(i)}|}^{\frac{1}{p}} 
&\leq \left( c_{\gamma}K_{\epsilon}(e^{\frac{\delta}{2}}-1)^{-\frac{1}{\gamma - 1}}\sum_{i = 0}^{\lceil2\delta^{-1} \log (r+1)\rceil}(i\delta)^{\frac{1+\epsilon}{\gamma-1}} \right)^{\frac{1}{p}}
\\
&\leq  c_{\gamma, p, \epsilon}\delta^{-\frac{1}{p}}(\lceil \log (r+1)\rceil)^{\frac{\gamma +\epsilon}{p(\gamma-1)}}.
\end{align*}

\end{proof}

\begin{prop}\label{prop:phi d comp 2}
\begin{enumerate}[(i)]
    \item 
$\bPb \times \Pb$-almost surely, for any $\alpha < 1,\beta > \frac{1}{2}$,
\[
\lim_{r \uparrow \infty} \sup_{x \in B(O, r)^c} \frac{|\phib(x)|}{\dt(O, x)^{(1-\alpha) \beta}} = 0.
\]
\item $\bPb \times \Pb$-almost surely, for any $\beta > \frac{1}{2}$,
\[
\lim_{r \uparrow \infty} \sup_{x \in B(O, r)^c} \frac{|\phiz(x)|}{(\log (\dt(O, x)+1))^{\beta}} = 0
\]
\end{enumerate}

\end{prop}
\begin{proof}

\textbf{Case $\alpha<1$.} The result is a consequence of \cite[Theorem 2.2.4 and Corollary 2.2.5]{VaartWellnerEmpiricalBook}. First, fix some $p > \frac{2\gamma}{\gamma - 1}$, and note that by Gaussianity, we have for any $s,t >0$ that $\bPb$-almost surely,
\[
\E{\frac{|\phib (s) - \phib (t)|^p}{\dbt(s,t)^{\frac{p}{2}}}} \leq C_p^p,
\]
where $C_p$ is a deterministic constant. Rearranging, we deduce that 
\[
\E{|\phib (s) - \phib (t)|^p}^{\frac{1}{p}} \leq C_p \sqrt{\dbt(s,t)},
\]
which verifies the first condition of \cite[Theorem 2.2.4]{VaartWellnerEmpiricalBook} in the case $\psi (x) = x^p$ (in which case the $\psi$-Orlicz norm in Theorem 2.2.4 there coincides with the usual $L^p$-norm). Now choose some $\delta \in (0, p - \frac{2\gamma}{\gamma - 1})$. Working with the pseudometric $\sqrt{\dbt(\cdot, \cdot)}$, in which case $\diam (B(O, r-1)) \leq 2\sqrt{r^{1-\alpha}}$ for all $r\geq 1$, we therefore deduce from \cite[Corollary 2.2.5]{VaartWellnerEmpiricalBook} and Fubini's theorem that for any $p>1$,
\begin{align*}
\Eb{\E{\sup_{s, t \in B(O, r-1)}|\phib (s) - \phib (t)|}} &\leq \Eb{\E{\sup_{s, t \in B(O, r-1)}|\phib (s) - \phib (t)|^p}^{\frac{1}{p}}} \\
&\leq \Eb{K_p \int_0^{2\sqrt{r^{1-\alpha}}} D\left(B(O, r-1), \epsilon, \sqrt{\dbt}\right)^{\frac{1}{p}} d \epsilon} \\
&= K_p \int_0^{2\sqrt{r^{1-\alpha}}} \Eb{D\left(B(O, r-1), \epsilon, \sqrt{\dbt}\right)^{\frac{1}{p}}} d \epsilon.
 \end{align*}
By Corollary \ref{cor:packing number bound}, we can also assume that $p$ is large enough that we can upper bound this by
 \begin{align*}
K_p \int_0^{2\sqrt{r^{1-\alpha}}} \tilde{C}_p(r^{1-\alpha} \epsilon^{-2})^{\frac{1}{3}} d \epsilon = 
{K}_p' (r^{1-\alpha})^{\frac{1}{3}} \left[ \epsilon^{\frac{1}{3}}\right]^{2\sqrt{r^{1-\alpha}}}_0  = \tilde{K}_p \sqrt{r^{1-\alpha}}.
\end{align*}
Therefore, if $\beta > \frac{1}{2}$, we can choose some $\eta < \beta - \frac{1}{2}$ and apply Markov's inequality to get that
\begin{align*}
\bPb \times \pr{\sup_{s, t \in B(O, r-1)}|\phib (s) - \phib (t)| \geq r^{({1-\alpha})(\beta - \eta)}} &\leq \Eb{\E{\sup_{s, t \in B(O, r)}|\phib (s) - \phib (t)|}} r^{-({1-\alpha})(\beta - \eta)} \\
&\leq \tilde{K}_p\sqrt{r^{1-\alpha}} r^{-({1-\alpha})(\beta - \eta)} \\
&= \tilde{K}_p r^{-({1-\alpha})(\beta - \eta - \frac{1}{2})}.
\end{align*}
Therefore, applying Borel-Cantelli along the subsequence $r_n = 2^n$, we deduce that
\[
\sup_{s, t \in B(O, r_n-1)}|\phib (s) - \phib (t)| < r_n^{({1-\alpha})(\beta - \eta)}
\]
for all sufficiently large $n$. Then, if $x \in B(O, r_{n+1}-1) \setminus B(O, r_n-1)$, and $d(O, x) = r \in [r_n-1, r_{n+1}-1]$, we have that
\begin{align*}
\frac{\phib(x)}{d(O, x)^{({1-\alpha})\beta}} \leq r^{-({1-\alpha})\beta} \sup_{s, t \in B(O, r_{n+1}-1)}|\phib (s) - \phib (t)| &< r^{-({1-\alpha})\beta} r_{n+1}^{({1-\alpha})(\beta - \eta)} 
\\
&\leq 2^{({1-\alpha})(\beta - \eta)} r^{- (1-\alpha) \eta} \to 0,
\end{align*}
as $r \to \infty$, which proves the result.

\textbf{Case $\alpha=1$.} As above, we can apply \cite[Theorem 2.2.4]{VaartWellnerEmpiricalBook} with the function $\psi (x) = x^p$. Again let $\beta > \frac{1}{2}, \epsilon>0$ and choose some $\eta < \frac{2}{3}\left(\beta - \frac{1}{2}\right)$. This time we instead apply Proposition \ref{prop:sqrt packing number d0}, work on the event $A_{\epsilon}$ and choose $p$ large enough that $\frac{1 +\epsilon}{p(\gamma-1)} < \eta/2$ to deduce that
 \begin{align*}
\Eb{\E{\sup_{s, t \in B(O, r)}|\phiz (s) - \phiz (t)|} \mathbbm{1}\{A_{\epsilon}\}} 
&\leq \int_0^{\sqrt{\log(r+1)}}c_{\gamma, p, \epsilon}\delta^{-\frac{2}{p}}\lceil (\log (r+1))\rceil^{\frac{\gamma +\epsilon}{p(\gamma-1)}} d \delta \\
&= c_{\gamma, p, \epsilon}\sqrt{\log(r+1)}\lceil (\log (r+1))\rceil^{\frac{\eta}{2}}.
\end{align*}
Therefore, if $\beta > \frac{1}{2}$, we can choose some $\eta < \frac{2}{3}\left(\beta - \frac{1}{2}\right)$ and apply Markov's inequality to get that
\begin{align*}
&\bPb \times \pr{\sup_{s, t \in B(O, r)}|\phiz (s) - \phiz (t)| \geq (\log(r+1))^{\beta - \eta}, \text{ and } A_{\epsilon}} \\ &\leq \Eb{\E{\sup_{s, t \in B(O, r)}|\phiz (s) - \phiz (t)|} \mathbbm{1}\{A_{\epsilon}\}} (\log(r+1))^{-(\beta - \eta)} \\
&\leq c_{\gamma, p, \epsilon} (\log(r+1))^{\frac{1}{2} + \frac{\eta}{2} - (\beta - \eta)} \\
&= c_{\gamma, p, \epsilon} (\log(r+1))^{-(\beta - \frac{3 \eta}{2} - \frac{1}{2})}.
\end{align*}
Therefore, applying Borel-Cantelli along the subsequence $r_n = 2^{2^n}-1$, we deduce that on the event $A_{\epsilon}$, i.e. with probability at least $1-\epsilon$, $\sup_{s, t \in B(O, r_n)}|\phiz (s) - \phiz (t)| < (\log(r_n+1))^{\beta - \eta}$ for all sufficiently large $n$. Then, if $x \in B(O, r_{n+1}) \setminus B(O, r_n)$, and $d(O, x) = r \in [r_n, r_{n+1}]$, we have that
\begin{align*}
\frac{\phiz(x)}{(\log (d(O, x)+1))^{\beta}} \leq (\log(r+1))^{-\beta} \sup_{s, t \in B(O, r_{n+1})}|\phiz (s) - \phiz (t)| &< (\log(r+1))^{-\beta} (\log(r_{n+1}+1))^{\beta - \eta} \\
&\leq 2(\log(r+1))^{-\beta} (\log(r_n+1))^{\beta - \eta} \\
&\leq 2(\log(r+1))^{-\eta}
\\
&\to 0,
\end{align*}
as $r \to \infty$. Since $\epsilon>0$ was arbitrary, this proves the result.
\end{proof}

\section{Reinforced strongly recurrent regime: $\alpha < 1, \Delta > 0$}\label{sctn:limsup asymp recurrent}

In this section we will prove \cref{thm:almost sure limsup}. Recall from \eqref{eqn:annealed X} that, given a realisation of $\Tg$, the limiting diffusion $X$ has the annealed law of a diffusion in a random potential $\phib$, i.e.
\begin{equation*}
P_{O} \left( (X_t)_{t \geq 0} \in \cdot \right)=\int \tilde{P}_{O,\phi} \left( ({X}_t)_{t \geq 0} \in \cdot \right) \mathbb{P}(d \phib).
\end{equation*}
We will in fact prove the analogue of \cref{thm:almost sure limsup} for the \textit{quenched} law of this diffusion, i.e. $\bPb \times \mathbb{P} \times \tilde{P}_{O,\phi}$- almost surely. This clearly implies the same result for the annealed process, and therefore for $X$.


\subsection{Volume and resistance growth in $\Tgphi$}
We start with some asymptotics for balls with respect to the distorted metric $R_{\phi}$. Here $B_{\phi}$ denotes the open ball measured with respect to $\Rphi$, as defined by \eqref{eqn:distorted res def}:
\begin{align*}
R_{\phi}(x,y)&=
\displaystyle \int_{[x,y]} e^{\BD \phib(z)+\AD [(\dt(O,z)+1)^{1-\alpha}-1] - \alpha \log (\dt (O, z)+1)} \lambda^{\gamma}(\text{d} z), \\
\nu_{\phi} (A)&=
 \displaystyle \int_A  e^{-\left[\BD \phib(x)+\AD [(\dt(O,z)+1)^{1-\alpha}-1] - \alpha \log (\dt(O,x)+1) \right]} \mu(\text{d}x),
\end{align*}
where $\BD = \sqrt{\frac{4 \Delta}{1-\alpha}}$ and $\AD = \frac{\Delta }{1-\alpha}$.

%
\begin{prop}\label{prop:ball comparison}
$\bPb \times \mathbb{P}$-almost surely, for any $\epsilon > 0$ we have that
\[
\Bphi (O, e^{(1-\epsilon)\AD r^{1-\alpha}}) \subset B(O, r) \subset \Bphi (O, e^{(1+\epsilon)\AD r^{1-\alpha}})
\]
for all sufficiently large $r$.
\end{prop}
\begin{proof}
Take some $\beta \in (\frac{1}{2},1)$ and some $\epsilon > 0$. Also set $b=1-\alpha$. By Proposition \ref{prop:phi d comp 2}$(i)$, there almost surely exists $R<\infty$ such that
\[
\sup_{x \in B( O, R)^c} \frac{|\phib(x)|}{\dt(O, x)^{b\beta}} \leq 1.
\]
Without loss of generality, we also assume that
\begin{itemize}
    \item $\AD [(R+1)^b-1] - \BD R^{b\beta} - \alpha \log (R+1) \geq (1-\frac{\epsilon}{2})\AD R^b$ ,
\item $\left(\frac{R-1}{R}\right)^b \geq \frac{1-\epsilon}{1-\frac{\epsilon}{2}}$,
    \item $\AD [(R+1)^b-1] + \BD R^{b\beta} - \alpha \log (R+1) \leq (1+\epsilon)\AD R^b$,
    \item $\AD bR^{b-1} \geq \mathbbm{1}\{b > 1\}$,
\item $\frac{b^{-1}R^{1-b}}{\AD} \leq e^{\epsilon \AD R^b}$, 
which we will just use in the case $b\leq 1$.
\end{itemize}
Now take some $r>R+1$. We start by showing that $ B(O, r)^c \subset \Bphi (O, e^{(1- \epsilon) \AD r^b})^c$ for all sufficiently large $r$. Indeed, assume that $\dt(O, x) > r$, and let $y$ be the point on the path from $O$ to $x$ such that $\dt(O, y) = R$, and $\dt(x,y)=\dt(O, x) - R$. Note that this almost surely exists and is unique, since $\Tg$ is a length space. Moreover, $\dt(O,z) > R$ for all $z\in [y,x]$ by our choice of $R$, which implies that $\AD\dt(O, z)^b- \BD\dt(O, z)^{b\beta} - \alpha \log \dt(O, z) \geq (1-\frac{\epsilon}{2})\AD \dt (O, z)^b$ for all such $z$. We therefore have that
\begin{align*}
\Rphi(O, x)
\geq \int_{[y, x]} e^{\AD[(\dt(O, z)+1)^b-1]- \BD\dt(O, z)^{b\beta} - \alpha \log (\dt(O, z)+1)} \lambda^{\gamma}(\text{d} z)
\geq \int_R^r e^{(1-\frac{\epsilon}{2})\AD s^b} ds 
&\geq e^{(1-\frac{\epsilon}{2})\AD (r-1)^b}
\\
&\geq e^{(1-\epsilon) \AD r^b}.
\end{align*}

For the second inclusion, we keep $\beta, R$ and $y$ as above, and now assume that $\dt(O, x) < r$. 
 Also set
\[
\tilde{r}_c = \inf\Bigg\{r \geq R: \sup_{w \in \partial B(O, R)} \int_{[O, w]} e^{\BD \phib(z)+ \AD[(\dt(O, z)+1)^b-1] - \alpha \log (\dt(O, z)+1)} \lambda^{\gamma}(\text{d} z) \leq \frac{\epsilon}{1+\epsilon} e^{(1+\epsilon)\AD r^b} \Bigg\}.
\] 
(This is always possible since $\phib$ is continuous with bounded sample paths, so is bounded on the compact set $B(O, R)$, and so $\sup_{w \in \partial B(O, R)} \int_{[O, w]} e^{\BD \phib(z)+ \AD \dt(O, z)^b-\alpha \log d(O,z)} \lambda^{\gamma}(\text{d} z)$ is almost surely finite).

We then have for all $r> \tilde{r}_c \vee R$ and all $x \in B(O, r)$ that
\begin{align*}
\Rphi(O, x) 
&\leq \int_{[O, y]} e^{\BD \phib(z)+ \AD[(\dt(O, z)+1)^b-1] - \alpha \log (\dt(O, z)+1)} \lambda^{\gamma}(\text{d} z) 
\\
&+ \int_{[y, x]} e^{\BD \phib(z)+ \AD[(\dt(O, z)+1)^b-1] - \alpha \log (\dt(O, z)+1)} \lambda^{\gamma}(\text{d} z)
\leq \frac{\epsilon}{1+\epsilon}e^{(1+\epsilon)\AD r^b} + \int_{R}^r e^{(1+\epsilon)\AD s^b} ds.
\end{align*}

In the case $b > 1$ we bound this above by (using that $\AD bR^{b-1} \geq 1$):
\begin{align*}
\frac{\epsilon}{1+\epsilon}e^{(1+\epsilon)\AD r^b} + \frac{1}{1+\epsilon} \int_{R}^r (1+\epsilon) \AD bs^{b-1}e^{(1+\epsilon)\AD s^b} ds 
&\leq e^{(1+\epsilon) \AD r^b}.
\end{align*}

If $b \leq 1$, we instead bound this above by (using that $\frac{b^{-1} R^{1-b}}{\AD}\leq e^{\varepsilon \AD R^b}$):
\begin{align*}
\frac{\epsilon}{1+\epsilon}e^{(1+\epsilon)\AD r^b} + \frac{b^{-1}r^{1-b}}{\AD (1+\epsilon)} \int_{R}^r (1+\epsilon) \AD bs^{b-1}e^{(1+\epsilon)\AD s^b} ds 
&\leq e^{(1+2\epsilon)\AD r^b},
\end{align*}
which completes the proof since $\epsilon$ was arbitrary.
\end{proof}

\begin{prop}\label{prop:Reff}
For any $\alpha<1$, we have that $\bPb \times \Pb$-almost surely, for any $\epsilon > 0$,
\[
r^{1-\epsilon} \leq \Rphi (O, \Bphi (O, r)^c) \leq r
\]
for all sufficiently large $r$.
\end{prop}
\begin{proof}
Fix $\epsilon > 0$, and let $\Mphi(r)$ be the smallest cardinality of a set of points in 
\[
\Bphi(O, r) \setminus \Bphi(O, r^{1-\epsilon})
\]
such that any path passing from $O$ to $\Bphi(O, r)^c$ must pass through one of the points in the set. Similarly to \cite[Lemma 4.5]{barlow2006cluster}, we note that since any such set is a cutset, it follows from the parallel law for resistance that
\begin{equation}\label{eqn:res cutset}
\Rphi (O, \Bphi (O, r)^c) \geq \frac{r^{1-\epsilon}}{\Mphi (r)}.
\end{equation}

Now choose $\he$ small enough that $(1+\he)(1-\epsilon) < 1-\he$. By Proposition \ref{prop:ball comparison}, 
it follows that any cutset separating 
\[
B\left( O, \left( \frac{(1-\epsilon) \log r}{\AD(1-\he)} \right)^{\frac{1}{1-\alpha}}\right) \text{ from } B\left( O, \left( \frac{\log r}{\AD(1+\he)} \right)^{\frac{1}{1-\alpha}}\right)^c
\]
is also a cutset separating $\Bphi(O, r^{1-\epsilon})$ from $\Bphi(O, r)^c$, so that $\Mphi(r) \leq M(r)$, where $M(r)$ is the smallest cardinality of a set of points separating 
\[
B\left( O, \left( \frac{(1-\epsilon) \log r}{\AD (1-\he)} \right)^{\frac{1}{1-\alpha}}\right) \text{ from } B\left( O, \left( \frac{\log r}{\AD(1+\he)} \right)^{\frac{1}{1-\alpha}}\right)^c
.
\]
By scaling invariance of the stable tree, $M(r) \overset{(d)}{=} M(e^{\AD (1+\he)})$ for all $r > 0$, so we bound the latter quantity, and set 
\[
\delta = 1 - \left(\frac{(1+\he)(1-\epsilon)}{(1-\he)}\right)^{\frac{1}{1-\alpha}},
\]
so that we are in fact counting subtrees from level $1-\delta$ to level $1$.

The ``size'' of generation $1-\delta$ can be formally measured by the local time measure of Proposition \ref{prop:local time limsup}. Moreover, conditional on the total local time measure at level $1-\delta$ being equal to $L$, it follows from the construction on page \pageref{box:Tg two sided construction} that the number of subtrees emanating from level $1-\delta$ that reach level $1$ is one more than a Poisson random variable with parameter $2L N(H>\delta)$, where $N$ is the \Ito excursion measure, so that $N(H>\delta) = c_{\gamma} \delta^{-\frac{1}{\gamma - 1}}$ for a deterministic $c_{\gamma} \in (0,\infty)$ \cite[Proposition 5.6]{goldschmidt2010extinctionstable}. Moreover, $M(e^{\AD (1+\he)})$ is equal to the number of such subtrees. Again letting $L$ denote the total local time measure at level $1-\delta$, we have from (a minor adaptation of) \cite[Proposition 5.2]{duquesnelegallhausdorff2006} that there exists $c>0$ such that $\prb{L > \lambda} \leq c\lambda^{-(\gamma - 1)}$, uniformly over $\delta \in [0,1]$. From a Chernoff bound, we therefore deduce that
\begin{equation}\label{eqn:1-delta M cutset bound}
\prb{M(e^{\AD (1+\he)}) > \lambda} \leq \prb{L>\frac{\delta^{\frac{1}{\gamma - 1}}}{2c_{\gamma}e}\lambda} + \prb{\textsf{Poisson}\left(\frac{\lambda}{e}\right) \geq \lambda} \leq c_{\delta} \lambda^{-(\gamma - 1)} + e^{-\lambda}.
\end{equation}

We therefore deduce from Borel-Cantelli and \eqref{eqn:res cutset} that if we take $r_n = 2^n$ and $\lambda_n=(\log r_n)^{\frac{1+\epsilon}{\gamma - 1}}$, then $\bPb \times \Pb$-almost surely, we have that
\[
\Rphi (O, \Bphi (O, r_n)^c) \geq \frac{r_n^{1-\epsilon}}{(\log r_n)^{\frac{1+\epsilon}{\gamma - 1}}} \geq r_n^{1-2 \epsilon}
\]
for all sufficiently large $n$. By monotonicity, this implies that 
$
\Rphi (O, \Bphi (O, r)^c) \geq r^{1-3\epsilon}
$
for all sufficiently large $r$. Since $\epsilon$ was arbitrary and the upper bound 
$
\Rphi (O, \Bphi (O, r)^c) \leq r
$
holds trivially, this proves the result.
\end{proof}

\begin{prop}\label{prop:tree vol bound}
$\bPb \times \Pb$-almost surely, $\nuphi (\Tg) < \infty$.
\end{prop}
\begin{proof}
Fix $\epsilon > 0$ and $\beta \in (\frac{1}{2}, 1)$. By Propositions \ref{prop:CRT vol growth}, \ref{prop:phi d comp 2}$(i)$ and since $\phib$ is continuous, we can $\bPb \times \Pb$-almost surely choose $R \geq 1$ and $V, \epsilon, p < \infty$ so that
\begin{itemize}
    \item $\sup_{x \in B(O, R)} \big\{ e^{-\left[\BD \phib(x)+\AD [(\dt(O,x)+1)^{1-\alpha}-1] - \alpha \log (\dt(O,x)+1) \right]}\big\} \leq p$,
    \item $\sup_{x \in B( O, R)^c} \frac{\BD |\phib(x)|+\alpha \log (\dt(O,x)+1)}{\dt(O, x)^{(1-\alpha)\beta}} < \epsilon$,
    \item $\sup_{r \geq 1} \Bigg\{ \frac{\mu \left(B(O, r+1)\right)}{r^{\frac{\gamma}{\gamma -1}} (\log r)^{\frac{1+\epsilon}{\gamma -1}}}\Bigg\} \leq V$.
\end{itemize}
We then have that
\begin{align*}
\nuphi (\Tg) &\leq p \mu \big(B(O, R)\big) 
\\
&+ \sum_{r=R}^{\infty} \sup_{x \in B(O, r+1) \setminus B(O, r)} \big\{ e^{-\left[\BD \phib(x)+\AD [(\dt(O,x)+1)^{1-\alpha}-1] - \alpha \log (\dt(O,x)+1) \right]}\big\} \mu \big(B(O, r+1) \setminus B(O, r)\big) \\
&\leq p V R^{\frac{\gamma}{\gamma -1}} (\log R)^{\frac{1+\epsilon}{\gamma -1}} + V\sum_{r=R}^{\infty} e^{\epsilon (r+1)^{(1-\alpha)\beta} - \AD (r^{1-\alpha}-1)} r^{\frac{\gamma}{\gamma -1}} (\log r)^{\frac{1+\epsilon}{\gamma -1}} =: C(p, V, R, \epsilon),
\end{align*}
where $C(p, V, R, \epsilon)$ is a finite constant.
\end{proof}

\subsection{Proof of Theorem \ref{thm:almost sure limsup}}

We now have all the tools to prove Theorem \ref{thm:almost sure limsup}.

\begin{proof}[Proof of Theorem \ref{thm:almost sure limsup}]
Again set $b = 1-\alpha$. We work pointwise on the probability space $\Omega'$ on which $\Tg$ and the environment $\phib$ are defined. Recall from \eqref{eqn:annealed X} that given $\Tgc$ and $\phib$, $\tilde{P}_{O,\phi}$ denotes the (quenched) law of the corresponding RWRE. We will show that the statement of Theorem \ref{thm:almost sure limsup} holds $\tilde{P}_{O,\phi}$ almost surely on $\Omega'$. This then transfers to the annealed law via \eqref{eqn:annealed X}. Since the quenched law of the LERRW limit is equal to the annealed law of the RWRE, this proves the result.

Throughout we assume that $\Tgc$ and $\phib$ are fixed and therefore just write $\tilde{P}$ in place of $\tilde{P}_{O, \phi}$.

We first show that, $\tilde{P}$-almost surely,
\[
\limsup_{t \rightarrow \infty} \frac{\dt(O, X_t)}{(\log t)^{{1}/{b}}} \leq \AD^{-1/b}.
\]
Letting $T_{B(O, r)}$ denote the exit time of $X$ from $B(O, r)$, we will in fact show that, almost surely,
\[
T_{B(O, r(1+\epsilon))} \geq e^{(1-\epsilon)\AD r^b},
\]
for all sufficiently large $r$. For notational convenience, set $t = e^{\AD r^b}$ and $\tau_r=T_{B(O, r(1+\epsilon))}$. Also let $T_r$ be a geometric random variable, with parameter $p_r$ given by
\[
p_r = \sup_{x \in \partial B (O, 1)} \ptstart{\tau_r < \tau_{O}}{x}{},
\]
where $\tau_{O}$ is the hitting time of $O$ ($p_r$ is well-defined by \cite[Proposition 1.9]{athreya2013brownian}).
We can then write
\begin{align*}
\pt{\tau_r \leq t} \leq \pt{\sum_{i=1}^{T_r-1} \xi_i \leq t},
\end{align*}
where $(\xi_i)_{i = 1}^{T_r-1}$ are the times to travel from $\partial B (O, 1)$ to $O$ and back to $\partial B (O, 1)$, conditional on not hitting $\partial B (O, r(1+\epsilon))$ first, from a ``best case" starting point on $\partial B (O, 1)$. Letting $L(t) = \log t$, we then have that
\begin{align}\label{eqn:sup UB geometric prob decomp}
\pt{\tau_r \leq t} \leq \pt{T_r \leq tL(t)} + \pt{\sum_{i=1}^{tL(t)-1} \xi_i \leq t}.
\end{align}
To bound the first term, note that since $\partial B (O, 1)$ is a cutset separating $O$ from $\partial B (O, r(1+\epsilon))$, we can apply Propositions \ref{prop:ball comparison} and \ref{prop:Reff} to deduce that $\bPb \times \Pb$-almost surely for any $\epsilon > 0$, we have for all sufficiently large $r$ that
\begin{align*}
p_r \leq \frac{1}{1 + \Rphi( \partial B (O, 1), \partial B (O, (1+\epsilon)r))} \leq \frac{1}{\Rphi(O, B (O, (1+\epsilon)r)^c)} &\leq \frac{1}{\Rphi(O, \Bphi (O, e^{\AD r^b(1+\frac{1}{2}\epsilon)})^c)} 
\\
&\leq e^{-\AD r^b(1+\frac{1}{4}\epsilon)}.
\end{align*}
Recalling that $t=e^{\AD r^b}$, we therefore have for all sufficiently large $r$ that $p_rtL(t) \leq \AD r^b e^{-\frac{1}{4} \AD \epsilon r^b}$, and hence converges to zero as $r \rightarrow \infty$. It follows that for all sufficiently large $r$,
\begin{align} \label{geombound}
\pt{T_r \leq tL(t)} = 1 - (1-p_r)^{tL(t)} \leq 1 - e^{-2 p_r tL(t)} \leq (1+ \delta) p_r  t L(t) \leq 2\AD r^b e^{-\frac{\varepsilon}{4}\AD r^b}.
\end{align}

Now take any $s>0$. To bound the second term of \eqref{eqn:sup UB geometric prob decomp}, we follow the approach of \cite[Section 4]{kumagai2004harnack} and use the Markov property to write
\begin{equation}\label{eqn:exit expectation relation}
\estartpt{T_{B(O, 1)}}{O}{} \leq s + \estartpt{T_{B(O, 1)} \mathbbm{1}\{T_{B(O, 1) > s}\}}{X_s}{} \leq s + \sup_{x \in B(O,1)}\estartpt{T_{B(O, 1)}}{x}{}\pt{T_{B(O, 1)} > s}.
\end{equation}
Letting $g_B(\cdot, \cdot)$ denote the Green's function for the diffusion killed on exiting the ball $B(O, 1)$ (e.g. as in \cite[Equation (4.7)]{kumagai2004harnack}), we can then write (also using \cite[Equations (4.5) and (4.6)]{kumagai2004harnack})
\begin{align*}
\sup_{x \in B(O,1)}\estartpt{T_{B(O, 1)}}{x}{} = \sup_{x \in B(O, 1)} \int_{B(O, 1)} g_B(x,y) \nuphi(dy) &\leq \sup_{x \in B(O, 1)} \int_{B(O, 1)} R_{\phi} (x, B(O, 1)^c) \nuphi(dy) 
\\
&\leq \sup_{x \in B(O, 1)} R_{\phi} (x, B(O, 1)^c) \cdot \nuphi (B(O, 1)).
\end{align*}
Note that $\nuphi (B(O, 1))$ is $\bPb \times \Pb$-almost surely finite and non-zero by Proposition \ref{prop:tree vol bound}, and the same holds for $\sup_{x \in B(O, 1)} R_{\phi} (x, B(O, 1)^c)$ by Proposition \ref{prop:ball comparison}. It therefore follows from \eqref{eqn:exit expectation relation} that there $\bPb \times \Pb$-almost surely exist constants $c_1 < c_2 \in (0, \infty)$ such that for all $s \geq 0$,
\[
c_1 \leq s + c_2\pt{T_{B(O, 1)} > s}. 
\]
Rearranging, we see that 
\[
\pt{T_{B(O, 1)} \leq s} \leq \frac{1}{c_2}s + \frac{c_2 - c_1}{c_2}.
\]
To return to \eqref{eqn:sup UB geometric prob decomp}, note that the sequence $(\xi_i)_{i=1}^{tL(t)}$ in the sum there stochastically dominates a sequence of independent copies of $T_{B(O,1)}$ (since the above bounds hold $\bPb \times \Pb$-almost surely, we can work pointwise on the probability space so that the tree and the environment are fixed, so we are only considering randomness under $P(\cdot)$). Therefore, letting $\left(\tilde{\xi_i}\right)_{i=1}^{tL(t)}$ denote independent copies of $T_{B(O, 1)}$ (conditional on our particular realisations of $\Tg$ and $\phib$), and recalling again that $t=e^{\AD r^b}$, we have from \cite[Lemma 3.14]{barlow1998fractals} that $\bPb \times \Pb$-almost surely there exist $c_3, c_4 \in (0, \infty)$ such that
\begin{align} \label{barlowbound}
\pt{\sum_{i=1}^{tL(t)-1} \xi_i \leq t} \leq \pt{\sum_{i=1}^{tL(t)-1} \tilde{\xi}_i \leq t} \leq \exp \left\{- tL(t)\left( c_3  -c_4 L(t)^{-\frac{1}{2}}\right)\right\} \leq e^{-c_3 e^{\AD r^b}}
\end{align}
for all sufficiently large $r$. In particular, substituting the bounds in \eqref{geombound} and \eqref{barlowbound} back into \eqref{eqn:sup UB geometric prob decomp} we have that, $\bPb \times \Pb$-almost surely,
\begin{align*}
\pt{\tau_r \leq e^{\AD r^b}} \leq 2 e^{-\frac{\varepsilon}{4}\AD r^b} + e^{-c_3 e^{\AD r^b}}
\end{align*}
for all sufficiently large $r$. Therefore, applying Borel-Cantelli along the sequence of integer $r$, we deduce that $\bPb \times \Pb \times \tilde{P}$-almost surely,
\[
T_{B(O, r(1+\epsilon))} \geq e^{\AD \lfloor r^b \rfloor} \geq e^{(1-\epsilon)\AD r^b}
\]
for all sufficiently large $r$, or in other words, 
\[
\dt (O,X_t) \leq \frac{1+\epsilon}{\AD^{1/b}(1-\epsilon)^{{1}/{b}}} (\log t)^{\frac{1}{b}}
\]
for all sufficiently large $t$. Since $\epsilon$ was arbitrary, this proves the upper bound. To prove the lower bound, i.e. that 
\[
\limsup_{t \rightarrow \infty} \frac{\dt(O, X_t)}{(\log t)^{\frac{1}{b}}} \geq \AD^{-1/b},
\]
we take $\epsilon > 0$ and set $f(t) = \AD^{-1/b}(1 - \epsilon)(\log t)^{\frac{1}{b}}$. Then, again using \cite[Equations (4.5), (4.6) and (4.7)]{kumagai2004harnack} and Markov's inequality, we have that
\begin{align*}
\tilde{P}_{O}(T_{B(O, f(t))} \geq t) \leq \tilde{E}_{O}[T_{B(O, f(t))}] t^{-1} &\leq \Rphi (O, B(O, f(t))^c) \nuphi (B(O, f(t))) t^{-1} 
\\
&\leq \Rphi (O, B(O, f(t))^c) \nuphi (\Tg) t^{-1}.
\end{align*}
By Proposition \ref{prop:ball comparison}, $\nuphi (\Tg)$ is almost surely finite. Now choose some $\delta = \delta (\epsilon) > 0$ small enough that $(1+\delta)(1-\epsilon)^b < 1$. By Propositions \ref{prop:ball comparison} and \ref{prop:Reff}, we therefore have that, almost surely for all sufficiently large $t$,
\begin{align*}
\tilde{P}_{O}(T_{B(O, f(t))} \geq t) \leq \nuphi (\Tg) e^{(1+\delta)\AD f(t)^b} t^{-1} = \nuphi (\Tg) e^{((1+\delta)(1-\epsilon)^b - 1)\log t},
\end{align*}
which vanishes as $t \rightarrow \infty$.
Applying this along with Fatou's Lemma, we therefore deduce that
\begin{align}\label{eqn:Fatou} \ptstart{\frac{d(O,X_t)}{(\log t)^{1/b}} > \AD^{-1/b} (1 - \epsilon) \text{ i.o.}}{O}{} = \ptstart{T_{B(O, f(t))} < t \text{ i.o.}}{O}{} \geq \limsup_{t \rightarrow \infty} \ptstart{T_{B(O, f(t))} < t}{O}{} = 1.
\end{align}
Since $\epsilon$ was arbitrary, this completes the proof.
\end{proof}

\section{Reinforced critical regime: $\alpha = 1, \Delta>0$}
In this section we prove \cref{thm:almost sure limsup critical}. The strategy is the same as in \cref{sctn:limsup asymp recurrent}, but due to the logarithmic factors we obtain different estimates.

In contrast to the previous section, this time
\begin{align*}
R_{\phi}(x,y)&=
\displaystyle \int_{[x,y]} e^{\BD \phiz(z)+\AD \log (\dt(O,z)+1)} \lambda^{\gamma}(\text{d} z), \\
\nu_{\phi} (A)&=
 \displaystyle \int_A  e^{-\left[\BD \phiz(z)+\AD \log (\dt(O,z)+1) \right]} \mu(\text{d}x),
\end{align*}
where $\BD = \sqrt{4 \Delta}$ and $\AD = \Delta-1$.

\subsection{Volume and resistance growth in $\Tgphiz$}
Again it is natural to start with some volume and resistance estimates.
\begin{prop}\label{prop:ball comparison 0}
$\bPb \times \Pb$-almost surely, for any $\epsilon > 0$ we have that
\[
\Bphi (O, r^{\AD+1} e^{-(\log (r+1))^{\frac{1}{2}+ \epsilon}}) \subset B(O, r) \subset \Bphi (O, r^{\AD+1} e^{(\log (r+1))^{\frac{1}{2}+ \epsilon}})
\]
for all sufficiently large $r$.
\end{prop}
\begin{proof}
Take some $\beta \in (\frac{1}{2},1), \beta' \in (\frac{1}{2}, \beta)$ and some $\epsilon > 0$. By Proposition \ref{prop:phi d comp 2}$(ii)$, there $\bPb \times \Pb$-almost surely exists $R<\infty$ such that $(\log R)^{\beta} < \frac{\epsilon}{2} \log R$ and
\[
\sup_{x \in B( O, R)^c} \frac{\BD|\phiz(x)|}{(\log (\dt(O, x)+1)^{\beta'}} \leq 1.
\]
Take some $r>2R+1$. We start by proving the first inclusion by showing the reverse inclusion for the complements. Indeed, assume that $\dt(O, x) > r$, and let $y$ be the point on the path from $O$ to $x$ such that $\dt(O, y) = R$, and $\dt(x,y)=\dt(O, x) - R$. Note that this almost surely exists and is unique. Moreover, $\dt(O,z) > R$ for all $z\in [y,x]$ by our choice of $R$, which implies that $\log (\dt(O, z)+1)-(\log (\dt(O, z)+1))^{\beta} \geq (1-\frac{\epsilon}{2})\log (\dt(O, z)+1)$ for all such $z$.

Similarly to \cref{prop:ball comparison}, provided that $r$ is sufficiently large (depending only on $R$ and $\Delta$) we have that
\begin{align*}
\Rphi(O, x) 
\geq \int_{[y, x]} e^{\BD\phiz(z)+\AD \log (\dt(O, z)+1)} \lambda^{\gamma}(\text{d} z) 
&\geq e^{-(\log (r+1))^{\beta'}} \int_{[y, x]} e^{\AD\log (\dt(O, z)+1)} \lambda^{\gamma}(\text{d} z) \\
&\geq e^{-(\log (r+1))^{\beta}} r^{\AD+1}.
\end{align*}

For the second inclusion, we keep $R, y, \beta$ and $\beta'$ as above, and now assume that $\dt(O, x) < r$. 
Also set
\[
\tilde{r}_c = \inf\Bigg\{r \geq R: \sup_{w \in \partial B(O, R)} \int_{[O, w]} e^{\BD\phiz(z)+\AD \log (\dt(O, z)+1)} \lambda^{\gamma}(\text{d} z) \leq \left(1 - \frac{1}{\AD + 1}\right)r^{\AD+1}e^{(\log (r+1))^{\frac{1}{2}}} \Bigg\}.
\] 
(This is always possible since $\phiz$ is continuous with bounded sample paths, so for any fixed $R > 0$, $\phiz$ is bounded on the compact set $B(O, R)$, and so the supremum is almost surely finite).

We then have for all sufficiently large (depending on $R$ and $\Delta$) $r$ and all $x \in B(O, r)$ that 
\begin{align*}
\Rphi(O, x) 
&\leq \sup_{w \in \partial B(O, R)} \int_{[O, w]} e^{\BD\phiz(z)+\AD \log (\dt(O, z)+1)} \lambda^{\gamma}(\text{d} z) + \int_{[y, x]} e^{\BD\phiz(z)+\AD \log (\dt(O, z)+1)} \lambda^{\gamma}(\text{d} z) \\
&\leq \left(1 - \frac{1}{\AD + 1}\right)r^{\AD+1}e^{(\log (r+1))^{\frac{1}{2}}} + e^{(\log (r+1))^{\beta'}} \int_{[y, x]} e^{\AD \log (\dt(O, z)+1)} \lambda^{\gamma}(\text{d} z) \\
&\leq r^{\AD+1}e^{(\log (r+1))^{\beta}},
\end{align*}
which completes the proof.
\end{proof}

\begin{prop}\label{prop:Reff 0}
For $\alpha=1$, we have for any $\epsilon > 0$ that $\bPb \times \Pb$-almost surely,
\[
re^{-(\log r)^{\frac{1}{2}+\epsilon}} \leq \Rphi (O, \Bphi (O, r)^c) \leq r
\]
for all sufficiently large $r$.
\end{prop}

\begin{proof}
The proof is similar to that of \cref{prop:Reff}. Fix $\epsilon > 0$, and let $\Mphi(r)$ be the smallest cardinality of a set of points in 
\[
\Bphi(O, r) \setminus \Bphi(O, re^{-3(\log r)^{\frac{1}{2}+\epsilon}}),
\]
such that any path passing from $O$ to $\Bphi(O, r)^c$ must pass through one of the points in the set. Since any such set is a cutset, it follows from the parallel law for resistance that
\[
\Rphi (O, \Bphi (O, r)^c) \geq \frac{re^{-3(\log r)^{\frac{1}{2}+\epsilon}}}{\Mphi (r)}.
\]
It follows from Proposition \ref{prop:ball comparison 0} that $\Mphi(r) \leq M(r)$ almost surely for all sufficiently large $r$, where $M(r)$ is the smallest cardinality of a set of points separating 
\[
B\left( O, r^{\frac{1}{\AD+1}}e^{-2(\log (r+1))^{\frac{1}{2}+\epsilon}}\right)
\text{ from }
B\left( O, r^{\frac{1}{\AD+1}} e^{-(\log (r+1))^{\frac{1}{2}+ \epsilon}} \right)^c,
\]
almost surely for all sufficiently large $r$.


Let $(L^{(s)})_{s \geq 0}$ be as in \eqref{eqn:local time cts function def}. The same arguments that led to \eqref{eqn:1-delta M cutset bound} entail that
\begin{align*}
\prb{M(r) > 1} &\leq \prb{L^{(r^{\frac{1}{\AD+1}}e^{-2(\log (r+1))^{\frac{1}{2}+\epsilon}})}>\left( r^{\frac{1}{\AD+1}}e^{-2(\log (r+1))^{\frac{1}{2}+\varepsilon}} \right)^{\frac{1}{\gamma-1}}e^{\frac{k}{(\gamma - 1)}(\log (r+1))^{\frac{1}{2}+\varepsilon}}} \\
& + \prb{\textsf{Poisson}\left(c_{\gamma}e^{\frac{-(\log (r+1))^{\frac{1}{2}+ \varepsilon}}{2(\gamma - 1)}}\right) \geq 1} \\
&\leq c_{\gamma}' e^{-(\log r)^{\frac{1}{2}+\varepsilon}/2(\gamma - 1)}.
\end{align*}

We therefore deduce from Borel-Cantelli that if we take $r_n = 2^n$, then $\bPb \times \Pb$-almost surely, we have that
\[
\Rphi (O, \Bphi (O, r_n)^c) \geq r_ne^{-3(\log r_n)^{\frac{1}{2}+\epsilon}}
\]
for all sufficiently large $n$. This then implies that $\Rphi (O, \Bphi (O, r)^c) \geq re^{-(\log r)^{\frac{1}{2}+2\epsilon}}$ for all sufficiently large $r$ by monotonicity. Since $\epsilon$ was arbitrary and the upper bound $ \Rphi (O, \Bphi (O, r)^c) \leq r$ holds trivially, this is enough to prove the result.
\end{proof}

\begin{prop}\label{prop:tree vol bound 0}
Take any $\beta > \frac{1}{2}$. Then $\bPb \times \Pb$-almost surely,
\begin{itemize}
    \item If $\AD \leq \frac{\gamma}{\gamma - 1}$, then $e^{ - (\log (r+1))^{\beta}} r^{\frac{\gamma}{\gamma - 1}-\AD} \leq \nuphi (B(O, r)) \leq r^{\frac{\gamma}{\gamma - 1}-\AD} e^{(\log (r+1))^{\beta}}$ for all sufficiently large $r$.
    \item If $\AD > \frac{\gamma}{\gamma - 1}$ then $\nuphi (\Tg) < \infty$.
\end{itemize}
\end{prop}
\begin{proof}
Again the proof is similar to that of \cref{prop:tree vol bound}. We treat the case $\AD< \frac{\gamma}{\gamma - 1}$ first, starting with the upper bound. Take some $\beta \in (\frac{1}{2}, 1)$. By Proposition \ref{prop:phi d comp 2}$(ii)$ and since $\phiz$ is continuous, we can $\bPb \times \Pb$-almost surely choose $R \geq 1$ and $\epsilon, p < \infty$ so that
\begin{itemize}
    \item $\sup_{x \in B(O, R)} \big\{ e^{-\BD \phiz(x) - \AD \log (\dt(O, x)+1)}\big\} \leq p$,
    \item $\sup_{x \in B( O, R)^c} \frac{\BD |\phiz(x)|}{(\log (\dt(O, x)+1)^{\beta}} < \epsilon$.
\end{itemize}
We then have from Proposition \ref{prop:CRT vol growth annulus} that for any $\delta > 0$, almost surely for all sufficiently large $r$,
\begin{align}\label{eqn:volume calc 0}
\begin{split}
\nuphi (B(O, r)) &\leq \sup_{x \in B(O, R)} \big\{ e^{-\BD\phiz(x) - \AD \log ( \dt(O, x)+1)}\big\} \mu\big(B(O, R)\big) \\
&\qquad + \sum_{n=R}^{r} \sup_{x \in B(O, n) \setminus B(O, n-1)} \big\{ e^{-\BD\phiz(x) - \AD \log ( \dt(O, x)+1)}\big\} \mu\big(B(O, n) \setminus B(O, n-1)\big) \\
&\leq p \mu\big(B(O, R)\big) + \sum_{n=R}^{r} n^{-\AD}e^{\varepsilon (\log n)^{\beta}} n^{\frac{1}{\gamma -1}} (\log n)^{\frac{1+\epsilon}{\gamma - 1}} \\
&\leq C_{p,R,\varepsilon} + r^{\frac{\gamma}{\gamma - 1}-\AD}e^{(\log r)^{\beta + \delta}}.
\end{split}
\end{align}
Since $\beta, \delta > 0$ were arbitrary this implies the upper bound. If $\AD=\frac{\gamma}{\gamma - 1}$ we obtain a factor of $\log r$ from the sum, but this is dominated by the $e^{(\log r)^{\beta + \delta}}$ term, so the result still holds.

The lower bound is simpler since
\[
\nuphi (B(O, r)) \geq \inf_{x \in B(O, r)} \big\{e^{-\BD\phiz(x) - \AD\log (\dt(O, x)+1)}\big\} \mu(B(O, r)) \geq  (r+1)^{-\AD} e^{- \varepsilon (\log (r+1))^{\beta}} \frac{r^{\frac{\gamma}{\gamma - 1}}}{\log r}
\]
almost surely for all sufficiently large $r$ by \cref{prop:CRT vol growth}. This proves the result since $\beta$ was arbitrary. Clearly the result holds trivially if $A=\frac{\gamma}{\gamma - 1}$.

If $\AD > \frac{\gamma}{\gamma - 1}$ the calculation in \eqref{eqn:volume calc 0} shows that $\nuphi (\Tg) < \infty$ almost surely.
\end{proof}

\subsection{Proof of Theorem \ref{thm:almost sure limsup critical}}

We now have all the tools to prove Theorem \ref{thm:almost sure limsup critical}.

\begin{proof}[Proof of Theorem \ref{thm:almost sure limsup critical}]
As in the proof of Theorem \ref{thm:almost sure limsup},  we assume that $\Tgc$ and $\phiz$ are fixed and therefore just write $\tilde{P}$ or $\tilde{P}_O$ in place of $\tilde{P}_{O, \phi}$.

Take $\beta ' \in (\frac{1}{2}, 1)$ and assume first that $\AD \leq \frac{\gamma}{\gamma - 1}$. Letting $g_B(\cdot, \cdot)$ denote the Green's function for the diffusion killed on exiting the ball $B(O, r)$ we firstly have from Propositions \ref{prop:ball comparison 0} and \ref{prop:tree vol bound 0} that, $\bPb \times \Pb$-almost surely for all sufficiently large $r$,
\begin{align*}
\estartpt{T_{B(O, r)}}{O}{} =  \int_{B(O, r)} g_B(O,y) \nuphi(dy) \leq \int_{B(O, r)} \Rphi (O, B(O, r)^c) \nuphi(dy) 
\leq r^{\frac{2\gamma - 1}{\gamma - 1}} e^{2(\log (r+1))^{\beta'}}.
\end{align*}
Therefore, for any $\beta > \frac{1}{2}$ we can choose $\beta' \in (\frac{1}{2}, \beta)$ so that Markov's inequality gives
\begin{align*}
    \ptstart{T_{B(O, r)} \geq r^{\frac{2\gamma - 1}{\gamma - 1}} e^{(\log (r+1))^{\beta}}}{O}{} \leq e^{-(\log (r+1))^{\beta}/2}
\end{align*}
for all sufficiently large $r$. 
In particular, Fatou's Lemma gives that
\begin{align*}
\ptstart{d(O, X_t) \geq t^{\frac{\gamma - 1}{2\gamma - 1}}e^{-(\log (t+1))^{\beta}} \text{ i.o.}}{O}{}
\geq \limsup_{r \to \infty} \ptstart{T_{B(O, r)} \leq r^{\frac{2\gamma - 1}{\gamma - 1}} e^{(\log (r+1))^{\beta}}}{O}{}=1.
\end{align*}
If instead $\AD > \frac{\gamma}{\gamma - 1}$, letting $V=\nuphi (\Tg)$ the same calculation gives
\begin{align*}
\estartpt{T_{B(O, r)}}{O}{} &\leq V r^{\AD+1} e^{(\log (r+1))^{\beta'}},
\end{align*}
so we similarly deduce from Markov's inequality and Fatou's Lemma that
\begin{align*}
\ptstart{d(O, X_t) \geq t^{\frac{1}{\AD+1}}e^{-(\log (t+1))^{\beta}} \text{ i.o.}}{O}{}=1.
\end{align*}

For the upper bound, we use \cite[Lemma 4.2(b)]{croydon2016scaling} which states that for any $t > 0$ and any $\delta \in (0, \Rphi (O, \Bphi(O, r)^c))$,
\begin{equation}\label{eqn:Croydon exit tail bound}
\ptstart{T_{\Bphi(O, r)} \leq t}{O}{} \leq 4\left[\frac{\delta}{\Rphi (O, \Bphi(O, r)^c)} + \frac{t}{\nuphi \left(\Bphi (O, \delta)\right) (\Rphi (O, \Bphi(O, r)^c) - \delta)}\right].
\end{equation}
Set $\delta = re^{-(\log r)^{\beta}}$ for some $\beta > \frac{1}{2}$, and choose $\beta' \in (\frac{1}{2}, \beta)$. Note that
\begin{itemize}
\item $\Bphi (O, r) \subset B(O, r^{\frac{1}{\AD+1}}e^{(\log (r+1))^{\beta'}})$ eventually almost surely by Proposition \ref{prop:ball comparison 0},
\item $\Rphi (O, \Bphi (O, r)^c) \geq re^{-(\log r)^{\beta '}}$ eventually almost surely by Proposition \ref{prop:Reff 0},
\item If $\AD \leq \frac{\gamma}{\gamma - 1}$, then almost surely, $\nuphi (\Bphi (O, \delta)) \geq  \delta^{\frac{1}{\AD+1}\left(\frac{\gamma}{\gamma-1}-\AD\right)} e^{ - (\log (\delta+1))^{\beta}}$ for all sufficiently large $\delta$ by Propositions \ref{prop:ball comparison 0} and \ref{prop:tree vol bound 0}. If $\AD \geq \frac{\gamma}{\gamma - 1}$, then there exists (a random) $V>0$ such that $\nuphi (\Bphi (O, \delta)) \geq V$ for all $\delta \geq 1$.
\end{itemize}
In the case $\AD \leq \frac{\gamma}{\gamma - 1}$, we take some $\eta>0$ and set $t=r^{\frac{2\gamma - 1}{(\gamma - 1)(\AD+1)}}e^{-(\log r)^{\beta+2 \eta}}$. We then deduce from \eqref{eqn:Croydon exit tail bound} that, on the almost sure events above and provided $r$ is sufficiently large,
\begin{align*}
\ptstart{T_{B\left(O, r^{\frac{1}{\AD+1}}e^{(\log (r+1))^{\beta}}\right)} \leq t}{O}{} &\leq \ptstart{T_{\Bphi(O, r)} \leq t}{O}{} \\
&\leq 4\left[\frac{re^{-(\log r)^{\beta}}}{re^{-(\log r)^{\beta'}}} + \frac{r^{\frac{2\gamma - 1}{(\gamma - 1) (\AD+1)}}e^{-(\log r)^{\beta + 2\eta}}}{r^{\frac{\gamma}{(\gamma - 1)(\AD+1)}-\frac{\AD}{\AD+1}}e^{-(\log r)^{\beta + \eta}} re^{-(\log r)^{\beta '}}}\right] \\
&\leq 8e^{-(\log r)^{\beta'}}.
\end{align*}

In particular, if $r_n = 2^n$ we have from Borel-Cantelli that
\[
\ptstart{T_{B\left(O, r_n^{\frac{1}{\AD+1}}e^{(\log (r_n+1))^{\beta}}\right)} \leq r_n^{\frac{2\gamma - 1}{(\gamma - 1) (\AD+1)}}e^{-(\log r_n)^{\beta + 2\eta}} \text{ i.o. }}{O}{}=0.
\]
If $r \in [r_n, r_{n+1}]$, this implies that 
\[
T_{B\left(O, r^{\frac{1}{\AD+1}}e^{(\log (r+1))^{\beta}}\right)} \geq T_{B\left(O, r_n^{\frac{1}{\AD+1}}e^{(\log (r_n+1))^{\beta}}\right)} \geq r^{\frac{2\gamma - 1}{(\gamma - 1) (\AD+1)}}e^{-(\log r)^{\beta + 3\eta}}
\]
for all sufficiently large $r$. By inverting this relation we deduce the result. If instead $\AD \geq \frac{\gamma}{\gamma - 1}$ we instead take $t=re^{-(\log r)^{\beta + 2\eta}}$ and \eqref{eqn:Croydon exit tail bound} instead gives that \begin{align*}
\ptstart{T_{B\left(O, r^{\frac{1}{\AD+1}}e^{(\log (r+1))^{\beta}}\right)} \leq r e^{-(\log r)^{\beta + 2\eta}}}{O}{} &\leq \ptstart{T_{\Bphi(O, r)} \leq re^{-(\log r)^{\beta + 2\eta}}}{O}{} \\
&\leq 4\left[\frac{re^{-(\log r)^{\beta}}}{re^{-(\log r)^{\beta'}}} + \frac{re^{-(\log r)^{\beta + 2\eta}}}{Ve^{-(\log r)^{\beta + \eta}} re^{-(\log r)^{\beta '}}}\right] \\
&\leq C_{V} e^{-(\log r)^{\beta'}},
\end{align*}
where $C_{V}>0$ is a (random) constant that depends on $V$. The final result can be deduced from Borel-Cantelli exactly as above.
\end{proof}

\section{Reinforced transient regime: $\alpha>1, \Delta>0$} \label{sctn:trans results}

In the transient regime resistance does not provide the right framework to characterise the scaling limit of LERRW, since the resistance between pairs of points in the correspondence collapses to zero under any rescaling. However, we can still use resistance when working directly with the unrescaled process to understand its exit times from a ball of radius $r$.

Accordingly, we let $(Y_t)_{t \geq 0}$ denote a constant speed continuous-time LERRW on $\Ti$ as given by Definition \ref{def:Kesten's tree} with offspring distribution satisfying \eqref{eqn:dom of att def} (the final result transfers directly to a discrete-time LERRW by the strong law of large numbers applied to the time index). For convenience, we assume that the function $a_n$ appearing in $\eqref{eqn:dom of att def}$ is of the form $a_n = cn^{1/\gamma}$; however one can also incorporate a slowly-varying correction and ``pull it through'' in all the proofs that follow. More precisely, $(Y_t)_{t \geq 0}$ has the annealed law of a continuous-time random walk in the Dirichlet random environment of \eqref{eqn:rescaled Dirichlet parameters} on the infinite tree $T_{\infty}$ with $\textsf{exp}(1)$ holding time at each vertex and initial weights given by
\begin{equation}\label{eqn:initial weights trans}
\alpha_{\{\overleftarrow{x},x\}}=
d_{\Ti}(O_{\infty}, x)^{\alpha}
\end{equation}
(so we just take $\alpha_{\{\overleftarrow{x},x\}}$ instead of $\alpha^{(n)}_{\{\overleftarrow{x},x\}}$ in \eqref{eqn:rescaled Dirichlet parameters}). We will work primarily with the RWRE in this section and denote the law of the environment by $\bP$. The discrete setup of Section \ref{sctn:GHP convergence} still holds, so that for $y \in \Ti$, cf. \eqref{eqn:discrete V, res, mu def}, we can write
\begin{align}\label{eqn:trans mm def}
V(y) = \sum_{O_{\infty} \prec v \preceq y} \log \rho_{v}, \hspace{1cm}
R(\overleftarrow{y},y) = e^{V(y)}, \hspace{1cm}
\nu(y) &= e^{-V(y)}\mathbbm{1}\{y \neq O_{\infty}\} + \sum_{z: \overleftarrow{z}=y} e^{-V(z)}.
\end{align}

Let $O_{\infty} = b_0, b_1, \ldots$ denote the backbone vertices of $\Ti$ (these are the special vertices of Definition \ref{def:Kesten's tree}), ordered by their distance from the root. Given $i \geq 1$, also let $T^i$ denote the (infinite) subtree of $\Ti$ rooted at $b_i$, and for any integer $r>0$ let $B_{T^i}(b_i, r)$ denote the ball of radius $r$ around $b_i$ \textit{in this subtree}, defined with respect to the graph metric. We use the notation $B^R$ to denote a ball defined with respect to the metric $R$ defined above. 

We start with a brief lemma on the structure of $\Ti \setminus T^{r}$.

\begin{prop}\label{prop:A_r diam bound}
$\bPb$-almost surely, for any $\epsilon > 0$,
\[
\limsup_{r \to \infty} \frac{\diam (\Ti \setminus T^{r})}{r (\log r)^{\gamma + \epsilon}}< \infty.
\]
\end{prop}
\begin{proof}
Let $T$ denote an \textit{unconditioned} Galton-Watson tree with the same offspring distribution as $\Ti$. By conditioning on the heights of the subtrees grafted to the infinite backbone, it follows from Definition \ref{def:Kesten's tree}, \cite[Lemma A.2]{archer2020random}, \cite[Theorem 2]{SlackInfVar} and a union bound that there exists $c<\infty$ such that for any $p>0, r \geq 1, \lambda \geq 1$,
\begin{align*}
    \prb{\diam (\Ti \setminus T^{r}) \geq 2 r \lambda} &\leq \prb{ \sum_{v \prec b_r} \deg v \geq r^{\frac{1}{\gamma-1}} \lambda^p} + r^{\frac{1}{\gamma-1}} \lambda^p\prb{\Height (T) \geq r\lambda} \\
    &\leq c\lambda^{-p(\gamma - 1)} + cr^{\frac{1}{\gamma-1}} \lambda^p (r\lambda)^{-\frac{1}{\gamma - 1}}.
\end{align*}
Taking $p=\frac{1}{\gamma(\gamma - 1)}$ gives an upper bound of order $\lambda^{-\frac{1}{\gamma}}$. Now set $\hat{\epsilon} = \frac{\epsilon}{2}$, apply Borel-Cantelli with $\lambda_r = (\log r)^{\gamma + \hat{\epsilon}}$ along the subsequence $r_n=2^n$, then use monotonicity and divide through by $(\log r)^{\hat{\epsilon}}$ for the result. 
\end{proof}

The difference with the recurrent regime is that typical terms of the form $\log \rho_{v}$ will now be negative. In particular, we show in the Appendix in Lemmas \ref{lem:log sum exp app 2} and \ref{cor:log var app 2} that, as $d_{\Ti}(O_{\infty}, y) \to \infty$,
\begin{align}\label{eqn:transient var and exp sum}
\begin{split}
    \E{\sum_{O_{\infty}\prec v \preceq y} \log \rho_v} & = -\alpha \log (d_{\Ti}(O_{\infty}, y)) + O(1), \hspace{1cm} 
    \Var{\sum_{O_{\infty}\prec v \preceq y} \log \rho_v} = O(1).
\end{split}
\end{align}

We start with the following bound on the behaviour of the potential on $\Ti$.

\begin{lemma}\label{lem:transient potential control}
For any $\delta > 0$, $\bPb$-almost surely,
\begin{enumerate}[(i)]
    \item 
    \[
    \pr{\left|\sum_{O_{\infty}\prec v \preceq b_n} \log \rho_v + \alpha \log (d_{\Ti}(O_{\infty}, b_n))\right| \geq (\log d_{\Ti}(O_{\infty}, b_n))^{\frac{1}{2}+\delta} \text{ i.o.}} = 0.
    \]
    \item As $m\to \infty$, 
    \[
    \inf_{y: b_m \preceq y} \pr{\left|\sum_{O_{\infty}\prec u \preceq v} \log \rho_u + \alpha \log (d_{\Ti}(O_{\infty}, v))\right| < (\log d_{\Ti}(O_{\infty}, v))^{\frac{1}{2}+\delta} \text{ for all } b_m \preceq v \prec y} \to 1.
    \]
\end{enumerate}

\end{lemma}
\begin{proof}
\begin{enumerate}[(i)]
    \item For $y \in \Ti$, let $W_y = \sum_{O_{\infty}\prec v \preceq y} \log \rho_v + \alpha \log (d_{\Ti}(O_{\infty}, y))$. It follows from Chebyshev's inequality and the bounds of \eqref{eqn:transient var and exp sum} that there exists a constant $c$ such that for any $y \in T_{\infty}$ and any $A>0$,
\begin{equation}\label{eqn:Cheb trans 2}
\pr{\left|W_y\right| \geq (\log d_{\Ti}(O_{\infty}, y))^{A}} \leq c(\log d_{\Ti}(O_{\infty}, y))^{-2A}.
\end{equation}
Applying Borel-Cantelli, we therefore deduce that for any $\delta > 0$,
\begin{equation} \label{eqn:transient Borel cantelli calc}
\pr{\left|W_{b_{2^m}}\right| \geq (\log d_{\Ti}(O_{\infty}, b_{2^m}))^{\frac{1+2\delta}{2}} \text{ i.o.}}=0.
\end{equation}
Moreover, if $r \in [2^m, 2^{m+1}]$ and $\delta < 1$, then $|(\log d_{\Ti}(O_{\infty}, b_{2^{m+1}}))^{\frac{1+\delta}{2}} - (\log d_{\Ti}(O_{\infty}, b_{r}))^{\frac{1+\delta}{2}}| \leq \log 2$, so that
\begin{align*}
&\prcond{ \left|W_{b_{2^{m+1}}}\right| \geq (\log (d_{\Ti}(O_{\infty}, b_{2^{m+1}})))^{\frac{1+\delta}{2}} }{\left|W_{b_{r}}\right| \geq (\log (d_{\Ti}(O_{\infty}, b_{r})))^{\frac{1+2\delta}{2}}}{} \\
&\geq \pr{\sum_{b_r \prec v \preceq b_{2^{m+1}}} \log \rho_v \geq -\frac{1}{2}(\log (d_{\Ti}(O_{\infty}, b_{r})))^{\frac{1+2\delta}{2}}}
\to 1,
\end{align*}
as $m\to \infty$ by \eqref{eqn:transient var and exp sum} and Chebyshev's inequality. It therefore follows that
\begin{align*}
&\frac{1}{2}\pr{\left|W_{b_{r}}\right| \geq (\log (d_{\Ti}(O_{\infty}, b_{r})))^{\frac{1+2\delta}{2}} \text{ i.o.}} \leq \pr{\left|W_{b_{2^{m+1}}}\right| \geq (\log (d_{\Ti}(O_{\infty}, b_{2^{m+1}})))^{\frac{1+\delta}{2}} \text{ i.o.}}=0.
\end{align*}
The last equality holds from \eqref{eqn:transient Borel cantelli calc}. This proves $(i)$ since $\delta > 0$ was arbitrary.
\item We first prove $(ii)$ 
when $y$ is on the backbone. First take $m \geq 1$ and note that 
\[
\pr{\left|\sum_{O_{\infty}\prec v \preceq b_k} \log \rho_v + \alpha \log (d_{\Ti}(O_{\infty}, b_k))\right| < (\log d_{\Ti}(O_{\infty}, b_k))^{\frac{1}{2}+\delta} \text{ for all } k \geq m} \to 1,
\]
as $m \to \infty$ by part $(i)$. 
Now note that if $y \in T_{\infty}$ with $b_k \preceq y$ for some $k \geq m$ and $y$ is not on the backbone, then it follows from the Dirichlet construction of Section \ref{sctn:LERRW and RWRE connection trees} that $R(b_k, y)$ and $(\sum_{u \prec v} \log \rho_u)_{v\in [b_m,y]}$ are distributed as $R(b_k, b_{d_{\Ti}(O_{\infty}, y)})$ and $(\sum_{u \prec v} \log \rho_u)_{v\in [b_m, b_{d_{\Ti}(O_{\infty},y)}]}$ respectively. Part $(ii)$ follows.
\end{enumerate}
\end{proof}

Recall that our aim in this section is to prove Theorem \ref{thm:almost sure limsup transient}. Our strategy to do this will be to divide the ball $B(O_{\infty}, r)$ up into smaller sets which must all be traversed by the RWRE of \eqref{eqn:rescaled Dirichlet parameters} before it can exit $B(O_{\infty}, r)$. We then apply a well-known chaining technique to bound the sum of all the exit times from these smaller sets. We then transfer this back to the LERRW using \eqref{lerrwannealed}.

Before doing this, we give some definitions of good and typical vertices and establish some of their properties.

\begin{defn}\label{def:good typical}
Fix some $\delta > 0$ and then fix some $A' > A > \frac{1}{2} + \delta$. To ease notation in the rest of this section, we view these parameters as fixed and do not record them as indices.
\begin{enumerate}
\item  Given $r, \lambda>1$, we define a sequence $(r_i)_{i=0}^{\infty}$ by $r_i = r_i(r,\lambda) = \frac{r}{2}(1+i \lambda^{-1})$ and set $\tilde{r} = \tilde{r} (r,\lambda) = \frac{r}{2}e^{-(\log r)^{A'}}\lambda^{-1}$.
\item We say that a vertex $y \in \Ti$ is \textbf{\mgood}if $b_m \preceq y$ and 
\[
\left|\sum_{O_{\infty}\prec u \preceq v} \log \rho_u + \alpha \log (d_{\Ti}(O_{\infty}, v))\right| < (\log d_{\Ti}(O_{\infty}, v))^{\frac{1}{2}+\delta} \ \ \ \ \ \mathrm{ for \ all }  \ b_m \preceq v \prec y,
\]
and that $y$ is \textbf{\mbad}otherwise.
\item For $c, C, m \in [0, \infty), r>2m$ and $\lambda >1$, we say that an index $0 \leq i < \lambda$ is \textbf{\typical}if 
\[
\frac{c}{2} r^{\frac{\gamma}{\gamma -1}} \lambda^{\frac{-\gamma}{\gamma -1}}e^{-\frac{\gamma}{\gamma -1}(\log r)^{A'}} \leq |\{y \in B_{T^{r_i}}(b_{r_i},\tilde{r}): y \ \text{\mgood}\}| < C r^{\frac{\gamma}{\gamma -1}} \lambda^{\frac{-\gamma}{\gamma -1}}e^{-\frac{\gamma}{\gamma -1}(\log r)^{A'}}.
\]
\end{enumerate}
\end{defn}

We will use the fact that for typical indices $i$, we can control (in probability) the time for the RWRE to exit a ball of the form $B_{T^{r_i}}(b_{r_i},\tilde{r})$. Moreover, it is highly likely that most indices are typical. We make this precise in the following lemma.

\begin{lemma}\label{lem:trans good props}
\begin{enumerate}[(i)]
Take $A'>A$ as in Definition \ref{def:good typical}, and $A'' > A'$.
    \item $\bPb \times \bP$-almost surely, there exist deterministic $C_{\alpha}>0$, $\tilde{C}_{\alpha}<\infty$ such that 
    \begin{align*}
C_{\alpha} \lambda^{-1} r^{1-\alpha} e^{-(\log r)^{A}} \leq R(b_{r_i}, b_{r_{i+1}}) &\leq \tilde{C}_{\alpha} \lambda^{-1} r^{1-\alpha} e^{(\log r)^{A}},
\end{align*}
for all $\lambda>1$, all $1 \leq i < \lambda$ and all sufficiently large $r$.
\item Moreover, let $B^R$ denote the ball with respect to the metric $R$ defined in \eqref{eqn:trans mm def}. Then for any $\epsilon >0$, it holds for all sufficiently large $r$, any $\lambda > 1$ and any $0 \leq i < \lambda$ that with $\bPb \times \bP$-probability at least $1-\epsilon$,
\begin{align*}
\nu\left(B_{T^{r_i}}(b_{r_i},\tilde{r}) \cap B^R(b_{r_i},\tilde{r}e^{(\log r)^A})\right) &\geq r^{\alpha} e^{-(\log r)^{A}} r^{\frac{\gamma}{\gamma -1}} \lambda^{\frac{-\gamma}{\gamma -1}}e^{-\frac{\gamma}{\gamma -1}(\log r)^{A''}}.
\end{align*}
\item Let $N=N(m,r,c,C,\lambda)$ denote the number of \typical $i$ in $\{0, \ldots, \lambda -1\}$. Then for any $\epsilon>0$, there exist deterministic $c > 0, C< \infty, m< \infty$ and an event $A_{\epsilon}$ such that, $\bPb$-almost surely, $\pr{A_{\epsilon}^c}<\epsilon$ and
$
\pr{N \leq \frac{\lambda}{2} \text{ and } A_{\epsilon}} \leq e^{-\frac{\lambda}{6}}.
$
\end{enumerate}
\end{lemma}
\begin{proof}
\begin{enumerate}[(i)]
\item To prove point $(i)$, note from Lemma \ref{lem:transient potential control}$(i)$ that we almost surely have for all sufficiently large $r$ and all $\lambda>1$ that
\begin{align}\label{eqn:trans case res calc}
\begin{split}
R(b_{r_i}, b_{r_{i+1}}) = \sum_{b_{r_i} \prec y \preceq b_{r_{i+1}}} e^{V(y)} &\geq \sum_{b_{r_i} \prec y \preceq b_{r_{i+1}}} e^{-\alpha \log (d_{\Ti}(O_{\infty}, y)) - (\log d_{\Ti}(O_{\infty}, y))^{A}} \\
&\geq \frac{{r_i}^{1-\alpha} - r_{i+1}^{1-\alpha}}{\alpha - 1} e^{-(\log r_{i+1} )^{A}} \\
&\geq C_{\alpha} r^{1-\alpha} \lambda^{-1} e^{-(\log r)^{A}}.
\end{split}
\end{align}
The calculation for the upper bound follows similarly.

\item For the second point, take some $i\geq 1$ and some $y \in B_{T^{r_i}}(b_{r_i},\tilde{r})$ and choose some $0 < \epsilon \ll \frac{1}{96}$. Note that, if $y$ is $m$-good, then similarly to \eqref{eqn:trans case res calc}, it holds that
\begin{align}\label{eqn:res in ball mgood}
\begin{split}
&R(b_{r_i}, y) \leq \frac{r_i^{1-\alpha}-d_{\Ti}(O_{\infty},y)^{1-\alpha}}{\alpha - 1} e^{(\log d_{\Ti}(O_{\infty},y) )^{A}} \leq C_{\alpha}'  r^{1-\alpha} \lambda^{-1} e^{-(\log r)^{A'}}e^{(\log d_{\Ti}(O_{\infty},y) )^{A}} \\
& \hspace{7.2cm} \leq r^{1-\alpha}\lambda^{-1}e^{-(\log r)^{\frac{1}{2}(A+A')}}, \\
&\nu (y) \geq e^{-V(y)} \geq d_{\Ti}(O_{\infty},y)^{\alpha}e^{-(\log d_{\Ti}(O_{\infty}, y)^{\frac{1}{2}+\delta}} \geq r^{\alpha}e^{-(\log r)^A}.
\end{split}
\end{align} 
Therefore, by Lemma \ref{lem:transient potential control}$(ii)$, which says that $y$ is $m$-good whp provided $d_{\Ti}(O_{\infty}, y)$ and $m$ are sufficiently large, we $\bPb$-almost surely have that for all sufficiently large $m$, if we fix some $r>2m, i \geq 1$ and $y \in B_{T^{r_i}}(b_{r_i},\tilde{r})$, then
\begin{align*}
\pr{\left|\sum_{O_{\infty}\prec u \preceq v} \log \rho_u + \alpha \log (d_{\Ti}(O_{\infty}, v))\right| \geq (\log d_{\Ti}(O_{\infty}, v))^{\frac{1}{2}+\delta} \text{ for some } b_m \preceq v \prec y} &\leq \epsilon, \\
    \pr{ R(b_{r_i}, y) \geq r^{1-\alpha}\lambda^{-1}e^{-(\log r)^{\frac{1}{2}(A+A')}}} &\leq \epsilon \\
    \pr{ \nu (y) \leq r^{\alpha}e^{-(\log r)^A}} &\leq \epsilon.
\end{align*}
Now fix $m$ large enough satisfying the above, and take any $r>2m$. By Markov's inequality and the inequalities above, it $\bPb$-almost surely holds for all sufficiently large $r$ and any $i \geq 1$ that
\begin{align}\label{eqn:transient proportion good}
 \pr{|\{y \in B_{T^{r_i}}(b_{r_i},\tilde{r}): y \text{ is \mbad}\}| \geq \frac{1}{2}|B_{T^{r_i}}(b_{r_i},\tilde{r})|} \leq 6\epsilon.
\end{align}
Moreover, since $B_{T^{r_i}}(b_{r_i},\tilde{r}) \overset{(d)}{=} B(O_{\infty},\tilde{r})$ and $A'' > A'$, we have from \cite[Lemmas 2.5 and 2.7]{croydon2008random} that
\begin{align*}
\begin{split}
    &\pr{|B_{T^{r_i}}(b_{r_i},\tilde{r})| \notin \left[r^{\frac{\gamma}{\gamma -1}} \lambda^{\frac{-\gamma}{\gamma -1}}e^{-\frac{\gamma}{\gamma -1}(\log r)^{A''}}, r^{\frac{\gamma}{\gamma -1}} \lambda^{\frac{-\gamma}{\gamma -1}}e^{-\frac{\gamma}{\gamma -1}(\log r)^{A''}}\right]}  \leq \epsilon.
\end{split}
\end{align*}
Therefore the result follows by a union bound over the relevant events above.
\item As above, we have from \cite[Lemmas 2.5 and 2.7]{croydon2008random} that there exist deterministic $c>0$, $C< \infty$ such that, for any $r>1, \lambda > 1, i \geq 1$,
\begin{align}\label{eqn:transient vol LB}
\begin{split}
    &\pr{|B_{T^{r_i}}(b_{r_i},\tilde{r})| \notin \left[c r^{\frac{\gamma}{\gamma -1}} \lambda^{\frac{-\gamma}{\gamma -1}}e^{-\frac{\gamma}{\gamma -1}(\log r)^{A'}}, C r^{\frac{\gamma}{\gamma -1}} \lambda^{\frac{-\gamma}{\gamma -1}}e^{-\frac{\gamma}{\gamma -1}(\log r)^{A'}}\right]} \\
    &\leq \pr{|B_{T^{r_i}}(b_{r_i},\tilde{r})| \notin \left[2^{-\frac{\gamma}{\gamma - 1}}c \tilde{r}^{\frac{\gamma}{\gamma -1}}, 2^{\frac{\gamma}{\gamma - 1}}C \tilde{r}^{\frac{\gamma}{\gamma -1}}\right]} \leq \frac{1}{8} - 6\epsilon.
\end{split}
\end{align}
 We deduce from a union bound over the events in \eqref{eqn:transient proportion good} and \eqref{eqn:transient vol LB} that, provided $m$ is sufficiently large,
\begin{align*}
    \pr{i \text{ is not \typical}} \leq \frac{1}{8}.
\end{align*}
These events are not independent for distinct $i$. However, by Lemma \ref{lem:transient potential control} we can choose $m$ so that the probability appearing in Lemma \ref{lem:transient potential control}$(ii)$ is at least $1-\epsilon$, and define the event $A_{\epsilon}$ to be the corresponding high probability event. Then on the event $A_{\epsilon}$ we can control $V(b_{r_i})$ and $R(O_{\infty}, b_{r_i})$ for all $i$ (for sufficiently large $r$), conditionally on which the tail bounds above hold independently for each $i$, so if $N = \left|\left\{ i:  i \text{ is \typical}\right\}\right|$, then $N$ stochastically dominates a $\textsf{Bin}$($2\lambda, \frac{7}{8}$) random variable. In particular this establishes the claim.
\end{enumerate}
\end{proof}

In the proof of Theorem \ref{thm:almost sure limsup transient} we assume that $\epsilon>0$ has been fixed and $c, C$ are as in Lemma \ref{lem:trans good props}$(iii)$.

\begin{proof}[Proof of Theorem \ref{thm:almost sure limsup transient}]
We proceed in three steps:
\begin{enumerate}
    \item Firstly we couple $T_{\infty}$ and its Dirichlet weights with a modified model $\hat{T}_{\infty} = \hat{T}_{\infty}(r, \lambda)$ such that (using a hat to denote analogous quantities in this model), letting $\hat{\tau}_i$ denote the time to hit $b_{r_{i+1}}$ starting from $b_{r_i}$, each $\hat{\tau}_i$ is stochastically dominated by its analogous quantity in $\Ti$.
    We then show that, with probability at least $1-\epsilon$, we can obtain comparable upper and lower bounds for $\E{\hat{\tau}_i}$ in the modified model for all typical indices $i$. We use these bounds to show that there $\bPb$-almost surely exists a constant $\kappa_{\alpha} \in (0, \infty)$ and an event $A_{\epsilon}$ such that $\pr{A_{\epsilon}^c}<\epsilon$ (the same event as in Lemma \ref{lem:trans good props}$(iii)$) and such that, provided $ r^{\frac{\gamma}{\gamma -1}} \lambda^{\frac{-\gamma}{\gamma -1}}e^{-\frac{\gamma}{\gamma -1}(\log r)^{A'}} \geq r\lambda^{-1}$, we have $\bPb$-almost surely have for all $0 \leq i < \lambda$ that
    \begin{equation}\label{eqn:exit time trans prob bound}
\ptstart{\hat{\tau_i} \leq s}{b_{r_i}, \omega}{}\mathbbm{1}\{A_{\epsilon}\} \leq 1-\kappa_{\alpha} e^{-\frac{3\gamma-2}{\gamma -1}(\log r)^{A'}} + \kappa_{\alpha} \frac{s}{r^{\frac{2\gamma - 1}{\gamma -1}} \lambda^{\frac{-(2\gamma - 1)}{\gamma -1}}e^{-\frac{\gamma}{\gamma -1}(\log r)^{A'}}}
    \end{equation}
for each \typical $i$.
\item We then apply a chaining result of \cite{barlow1998fractals} to deduce that, $\bPb$-almost surely,
\begin{align*}
 \ptstart{\sum_{i=1}^N \hat{\tau}_i < t, N \geq \frac{\lambda}{2}}{\omega}{}\mathbbm{1}\{A_{\epsilon}\} &\leq \exp \left\{ -e^{(1-\frac{\delta}{4})(\log r)^{A' + \delta}} \right\}.
\end{align*}
    \item Using the stochastic domination, we then transfer this result to the analogous sum $\sum_{i=1}^N {\tau}_i$ on the original tree. Using this tail bound, we bound the hitting time of $b_{r}$ with a Borel-Cantelli argument.
\end{enumerate}
\textbf{Step 1.} Fix $\epsilon>0$ and $\delta>0$, and choose $m'$ large enough that the probability in Lemma \ref{lem:transient potential control} is at least $1-\epsilon$. Let $A_{\epsilon}$ be the corresponding high probability event. For any $r>2m'$, we can proceed as follows.

Set $m = \lfloor \frac{r}{2}\rfloor$. The bound of Lemma \ref{lem:transient potential control}$(ii)$ still holds since this can only increase the value of $m$, and for the rest of the proof we work on the event $A_{\epsilon}$, which in particular implies that
\[
\left|\sum_{O_{\infty}\prec u \preceq v} \log \rho_u + \alpha \log (d_{\Ti}(O_{\infty}, v))\right| < (\log d_{\Ti}(O_{\infty}, v))^{\frac{1}{2}+\delta} \text{ for all } b_m \preceq v \prec b_r.
\]
Given $r>1, \lambda > 1$ and $T_{\infty}$, we define $\hat{T}_{\infty} = \hat{T}_{\infty}(r, \lambda)$ from $\Ti$ as follows:
\begin{itemize}
    \item First remove all \mbad vertices and their incident edges from $T_{\infty}$.
    \item Then also remove all vertices in $\bigcup_{i=0}^{\lambda-1} \left(T^{r_i} \setminus T^{r_{i+1}} \right) \setminus \left(\bigcup_{i=0}^{\lambda-1} B_{T^{r_i}}(b_{r_i},\tilde{r}) \cup \bigcup_{j=r_i}^{r_{i+1}} b_j \right)$.
\end{itemize}  
(For the vertices and edges that are retained in $\hat{T}_{\infty}$, they retain the measures and resistances that they previously had in $T_{\infty}$). The resulting structure is $\hat{T}_{\infty}$. Note that the definition of \mgood and Lemma \ref{lem:transient potential control} ensures that $\hat{T}_{\infty}$ is really a connected tree for all sufficiently large $r$.

We use hat notation to denote all analogous quantities in $\hat{T}_{\infty}$. For each $i < \lambda$, we consider the subtree $\hat{T}^{r_i}$ (rooted at $b_{r_i}$) and let $\hat{\tau}_i$ denote the exit time from $\hat{T}^{r_i} \setminus \hat{T}^{r_{i+1}}$ \emph{in the subtree} $\hat{T}^{r_i}$ for a random walk started at $b_{r_i}$ (this means that the random walk can only exit $\hat{T}^{r_i} \setminus \hat{T}^{r_{i+1}}$ through $b_{r_{i+1}}$; since the random walk eventually has to pass through $b_{r_i}$ to exit $\hat{T}_{\infty}\setminus \hat{T}^{r_i}$ and further excursions back towards the root and in subtrees can only slow the random walk down, it is clearly sufficient to restrict to this set, and clearly $\hat{\tau}_i$ is stochastically dominated by $\tau_i$ for all $i$).

If $r$ is sufficiently large and $y \in B_{\hat{T}^{r_i}}(b_{r_i},\tilde{r})$ is \mgoodd, then by \eqref{eqn:res in ball mgood} and Lemma \ref{lem:trans good props}$(i)$, we have that

\begin{equation}\label{eqn:transient res ratio good}
\frac{R(b_{r_i}, y)}{R(b_{r_i}, b_{r_{i+1}})} \leq \frac{r^{1-\alpha} \lambda^{-1} e^{-(\log r)^{\frac{1}{2}(A+A')}} }{C_{\alpha}  r^{1-\alpha} \lambda^{-1} e^{-(\log r)^{A}}} \leq \frac{1}{2},
\end{equation}
since $A' > A$.  Now let $\hat{A}= \{y \in B_{T^{r_i}}(b_{r_i},\tilde{r}):y \ \mgood\}$, let $\hat{B}=\hat{T}^{r_i}\setminus \hat{T}^{r_{i+1}}$ and let $g_{\hat{A}}(\cdot, \cdot)$, $g_{\hat{B}}(\cdot, \cdot)$ respectively denote the Green's function for the diffusion killed on exiting ${\hat{A}}$ and $\hat{B}$ (e.g. as in \cite[Equation (4.7)]{kumagai2004harnack}). It then follows as in \cite[Equation (4.8)]{kumagai2004harnack} (also using \eqref{eqn:transient res ratio good}) that
\[
\frac{g_{\hat{B}}(b_{r_i}, y)}{g_{\hat{B}}(b_{r_i}, b_{r_i})} \geq 1 - \sqrt{2}.
\]
Recall that $\hat{\tau}_{i}$ is the exit time of $\hat{B}$. Consequently, since $g_{\hat{B}}(b_{r_i}, b_{r_i}) = R(b_{r_i}, b_{r_{i+1}})$ by \cite[Equation (4.5)]{kumagai2004harnack}, if $i$ is $(m,r,c,C,\lambda)$-typical, we have from Definition \ref{def:good typical}, Lemma \ref{lem:trans good props}$(i)$ and  \eqref{eqn:res in ball mgood} that
\begin{align}\label{eqn:Etaui UB}
\begin{split}
    \estartpt{\hat{\tau}_{i}}{b_{r_i}, \omega}\mathbbm{1}\{A_{\epsilon}\} &\geq \sum_{y \in \hat{A}} g_{\hat{B}}(b_{r_i}, y) \nu(y) \\
    &\geq (1 - \sqrt{2}) \sum_{y \in \hat{A}} R(b_{r_i}, b_{r_{i+1}}) \nu(y) \\
    &\geq (1 - \sqrt{2}) \frac{c}{2} r^{\frac{\gamma}{\gamma -1}} \lambda^{\frac{-\gamma}{\gamma -1}}e^{-\frac{\gamma}{\gamma -1}(\log r)^{A'}} \cdot C_{\alpha} r^{1-\alpha} \lambda^{-1}  e^{-(\log r)^{A}} \cdot r^{\alpha}e^{-(\log r)^A} \\
    &\geq r^{\frac{2\gamma - 1}{\gamma -1}} \lambda^{\frac{-(2\gamma-1)}{\gamma -1}}e^{-\frac{3\gamma-2}{\gamma -1}(\log r)^{A'}}.
\end{split}
\end{align}
We would like to obtain a similar upper bound. First note that for any $z \in \cup_{i=0}^{\lambda-1} B_{\hat{T}^{r_i}}(b_{r_i},\tilde{r})$, it also follows from \eqref{eqn:transient res ratio good} and Lemma \ref{lem:trans good props}$(i)$ that $R(z, b_{r_{i+1}}) \leq \frac{3}{2} R(b_{r_{i}}, b_{r_{i+1}}) \leq \frac{3}{2} \tilde{C}_{\alpha} r^{1-\alpha} \lambda^{-1} e^{(\log r)^{A}}$. Similarly if $z \in \cup_{j=r_i}^{r_{i+1}} b_j$, we have that $R(z, b_{r_{i+1}}) \leq R(b_{r_{i}}, b_{r_{i+1}})$ and therefore the same upper bound holds. Therefore, since $g_{\hat{B}}(z,y)\le g_{\hat{B}}(y,y)$ for all such $z$ and $g_{\hat{B}}(y,y) = R(y, b_{r_{i+1}})$ by \cite[Equation (4.6)]{kumagai2004harnack} and by \cite[Equation (4.5)]{kumagai2004harnack} respectively, if $i$ is \typical we can use Definition \ref{def:good typical} and \eqref{eqn:res in ball mgood} to deduce that there exists a constant $K_{\alpha}<\infty$ such that, whenever $ r^{\frac{\gamma}{\gamma -1}} \lambda^{\frac{-\gamma}{\gamma -1}}e^{-\frac{\gamma}{\gamma -1}(\log r)^{A'}} \geq r\lambda^{-1}$,
\begin{align}\label{eqn:Etaui LB}
\begin{split}
  \estartpt{\hat{\tau}_{i}}{z, \omega}\mathbbm{1}\{A_{\epsilon}\} &= \sum_{y \in \hat{B}}g_{\hat{B}}(z,y) \nu(y) 
  \\
 &\leq  \sum_{y\in \hat{B}}g_{\hat{B}}(y,y) \nu(y) \\
  &\leq (|\hat{A}|+(r_{i+1} - r_i)) \sup_{y \in \hat{B}} \{ R(y, b_{r_{i+1}}) \cdot \nu(y): y \in \hat{A}\} \\
  &\leq C( r^{\frac{\gamma}{\gamma -1}} \lambda^{\frac{-\gamma}{\gamma -1}}e^{-\frac{\gamma}{\gamma -1}(\log r)^{A'}} + r\lambda^{-1}) \cdot \frac{3}{2} \tilde{C}_{\alpha} r^{1-\alpha} \lambda^{-1} e^{(\log r)^{A}} \cdot r^{\alpha}e^{(\log r)^A} \\
  &\leq K_{\alpha}r^{\frac{2\gamma - 1}{\gamma -1}} \lambda^{\frac{-(2\gamma - 1)}{\gamma -1}}e^{-\frac{\gamma}{\gamma -1}(\log r)^{A'}},
  \end{split}
\end{align}
By the arguments of \cite[Lemma 4.2]{kumagai2004harnack} and applying \eqref{eqn:Etaui UB} and \eqref{eqn:Etaui LB}, we therefore have on the event $A_{\epsilon}$ that for any $s>0$,
\begin{align*}
r^{\frac{2\gamma - 1}{\gamma -1}} \lambda^{\frac{-(2\gamma-1)}{\gamma -1}}e^{-\frac{3\gamma-2}{\gamma -1}(\log r)^{A'}} \leq  \estartpt{\hat{\tau}_{i}}{b_{r_i},\omega} &\leq s + \ptstart{\hat{\tau_i} > s}{b_{r_i}, \omega}{} \sup_{z \in B_{T^{r_i}}(b_{r_i},\tilde{r})} \estartpt{\hat{\tau}_{i}}{z, \omega} \\
&\leq s + \ptstart{\hat{\tau_i} > s}{b_{r_i},\omega}{} K_{\alpha} r^{\frac{2\gamma - 1}{\gamma -1}} \lambda^{\frac{-(2\gamma - 1)}{\gamma -1}}e^{-\frac{\gamma}{\gamma -1}(\log r)^{A'}}.
\end{align*}
Rearranging we deduce \eqref{eqn:exit time trans prob bound}.

\textbf{Step 2.} Since \eqref{eqn:exit time trans prob bound} holds for all \typical indices $i$, and recalling that $N = N(m,r,c,C,\lambda)$ denotes the number of \typical indices, we therefore have from \cite[Lemma 3.14]{barlow1998fractals} that on the event $A_{\epsilon}$, it holds for all $t>0$ that
{\scriptsize \[
\log \left( \ptstart{\sum_{i=1}^N \hat{\tau}_i \leq t}{\omega}{} \right) \leq 2\left( \frac{Nt\kappa_{\alpha}}{r^{\frac{2\gamma - 1}{\gamma -1}} \lambda^{\frac{-(2\gamma - 1)}{\gamma -1}}e^{-\frac{\gamma}{\gamma -1}(\log r)^{A'}}(1-\kappa_{\alpha} e^{-\frac{3\gamma-2}{\gamma -1}(\log r)^{A'}})} \right)^{\frac{1}{2}} - N\log \left( (1-\kappa_{\alpha} e^{-\frac{3\gamma-2}{\gamma -1}(\log r)^{A'}})^{-1} \right).
\]}

In particular, taking $\lambda = e^{(\log r)^{A' + \delta}}$ (which satisfies the conditions required in step 1 provided that $r$ is sufficiently large) and $t = r^{\frac{2\gamma - 1}{\gamma -1}} \lambda^{\frac{-(2\gamma - 1)}{\gamma -1}}e^{-\frac{\gamma}{\gamma -1}(\log r)^{A'}}(1- e^{-\frac{3\gamma-2}{\gamma -1}(\log r)^{A'}})\lambda^{1-\delta}$ gives that, for all sufficiently large $r$,
\begin{align*}
 \ptstart{\sum_{i=1}^N \hat{\tau}_i < t, N \geq \frac{\lambda}{2}}{\omega}{}\mathbbm{1}\{A_{\epsilon}\} \leq \exp \left\{ 4\kappa_{\alpha} \lambda^{1-\frac{\delta}{2}} - \frac{\kappa_{\alpha}}{2}\lambda e^{-\frac{3\gamma-2}{\gamma -1}(\log r)^{A'}} \right\} 
 &\leq \exp \left\{ -e^{(1-\frac{\delta}{4})(\log r)^{A' + \delta}} \right\}.
\end{align*}
Consequently, combining with Lemma \ref{lem:trans good props}$(iii)$ and a union bound we deduce that, on the event $A_{\epsilon}$,
\begin{align*}
 \ptstart{\sum_{i=1}^N \hat{\tau}_i < t}{\omega}{} &\leq \ptstart{N < \frac{\lambda}{2}}{\omega}{} +  \ptstart{\sum_{i=1}^N \hat{\tau}_i < t, N \geq \frac{\lambda}{2}}{\omega}{} \\
 &\leq \exp\left\{\frac{-e^{(\log r)^{A' + \delta}}}{12}\right\} + \exp \left\{ -e^{(1-\frac{\delta}{4})(\log r)^{A' + \delta}} \right\}.
\end{align*}
\textbf{Step 3.} Since adding back the \mbad vertices and then adding the times to traverse balls corresponding to non-typical $i$ can only increase the total exit time, the same tail bound clearly holds for the exit time of $T_{\infty}\setminus T^{r_i}$. By Proposition \ref{prop:A_r diam bound}, we can also choose $R$ large enough that the event
\[
B_{\epsilon} := \left\{\sup_{r \geq R}\frac{\diam (\Ti \setminus T^{r})}{r (\log r)^{\gamma + \epsilon}} < 1\right\}
\]
satisfies $\prb{B_{\epsilon}} \geq 1-\epsilon$.

We deduce from steps 1 and 2 that, on the events $A_{\epsilon}$ and $B_{\epsilon}$,
\begin{align*}
    \ptstart{T_{B(O_{\infty},r( \log r)^{\gamma+\epsilon})} \leq r^{\frac{2\gamma}{2\gamma - 1}} e^{-(\log r)^{A' + 2 \delta}}}{\omega}{} \leq \ptstart{\sum_{i=1}^N \hat{\tau}_i \leq r^{\frac{2\gamma}{2\gamma - 1}} e^{-(\log r)^{A' + 2 \delta}}}{\omega}{} \leq 2e^{-e^{(1-\frac{\delta}{4})(\log r)^{A' + \delta}}}.
\end{align*}
Consequently, it follows from Borel-Cantelli along the subsequence $r_n = 2^n$ and monotonicity of $T_r$ that almost surely,
\[
d_{\Ti}\left(O_{\infty}, Y_{r^{\frac{2\gamma}{2\gamma - 1}} e^{-(\log r)^{A' + 2 \delta}}}\right)\mathbbm{1}\{A_{\epsilon}, B_{\epsilon}\} \leq r( \log r)^{\gamma +\epsilon},
\]
for all sufficiently large $r$. Since $A_{\epsilon}$ and $B_{\epsilon}$ both have $\bPb \times \Pb$-probability at least $1-\epsilon$, this implies that with overall probability at least $1-2\epsilon$, we have that
\[
d_{\Ti}\left(O_{\infty}, Y_{t}\right) \leq t^{\frac{2\gamma - 1}{2\gamma}} e^{(\log t)^{A' + 3 \delta}}
\]
for all sufficiently large $t$. Since $A' > \frac{1}{2}$, $\delta > 0$ and $\epsilon>0$ were arbitrary, this implies the result for the annealed law $\Pbt$, and therefore for the LERRW.
\end{proof}

\begin{remark}[Application to LERRW on $\Z_{+}$]\label{remark:result on Z}
We can also apply these results to LERRW on the infinite half line with these initial weights. Note that Lemma \ref{lem:transient potential control}$(i)$ also applies in this setting, to give that, for any $\delta > 0$,
    \[
    \pr{\left|\sum_{0 < i \leq r} \log \rho_i + \alpha \log r \right| \geq (\log r)^{\frac{1}{2}+\delta} \text{ i.o.}} = 0.
    \]
    In particular, this implies that $\bPb \times
    \Pb$-almost surely, there exist constants $c>0, C<\infty$ such that for all sufficiently large $r$,
    \begin{align*}
    r^{\alpha + 1} e^{-(\log r)^{\frac{1}{2} + \delta}} \leq \nu ([0,r]) &\leq r^{\alpha + 1} e^{(\log r)^{\frac{1}{2} + \delta}} \\
    c \leq R (0, [0,r]^c) &\leq C.
    \end{align*} 
    Moreover, again using \cite[Equation (4.5)]{kumagai2004harnack}, we deduce that, if $T_r$ is the exit time from $[0,r]$, then on these events it holds that
\begin{align*}
cr^{\alpha+1} e^{-(\log r)^{\frac{1}{2} + \delta}} \leq \sum_{y \in [0,\frac{r}{2}]}R(y, [0,r]^c) \nu(y)  \leq \estartpt{T_r}{0,\omega} &\leq \sum_{y \in [0,r]}R(0, [0,r]^c) \nu(y) \leq Cr^{\alpha+1} e^{(\log r)^{\frac{1}{2} + \delta}}.
\end{align*}
By Markov's inequality, Borel-Cantelli along the sequence $r_n=2^n$ and monotonicity we therefore get that $T_r \leq r^{2} e^{(\log r)^{\frac{1}{\alpha+1} + \delta}}$ eventually almost surely, i.e. that $\sup_{s \leq t} d(O_{\infty}, Y_s) \geq t^{\frac{1}{2}} e^{-(\log t)^{\frac{1}{\alpha+1} + \delta}}$ eventually almost surely.\end{remark}

\begin{appendix}
\section{Appendix}
In the appendix we prove Claims \ref{claim:log sum exp}, \ref{claim:log var} and \ref{claim:log mgf}, regarding the expectation, variance and moment generating function of the random variables $(\log \rho_x)_{x \in T_n}$.

In the appendix we always work conditionally on $T_n$. In particular, we will work pointwise on the probability space $(\mathbf{\Omega}, \mathbf{\F}, \bPb)$ (on which we defined $(T_n, d_n, \mu_n)$), and on which the convergence of \eqref{eqn:stable tree scaling limit def} holds almost surely. This means that most of the statements that follow should really be written conditionally on $T_n$. To make the arguments clearer to follow, we have not written this explicitly in the statements or the proofs, and instead ask the reader to keep this in mind throughout.

Moreover, a statement of the form $a_n = a +o(1)$ almost surely on $\Omega$ means that $a_n \to a$ almost surely and therefore that the $o(1)$ term is not necessarily bounded uniformly on $\Omega$. However, it will always be true that for any $\epsilon > 0$, the $o(1)$ term can be bounded uniformly on a set of probability at least $1-\epsilon$. The same holds for $O(\cdot)$ terms.

Recall that we simplified notation by writing
\[
\Delta_n=
\begin{cases}
\Delta (n a_n^{-1})^{-(1-\alpha)}, &\text{ if } \alpha<1,
\\
\Delta, &\text{ if } \alpha=1,
\end{cases}
\qquad |x|=d_n(O_n,x).
\]

\subsection{Properties of the digamma function}

The mean and variance of terms of the form $\log \rho_x$ can be expressed in terms of the digamma function, defined for $z>0$ by $\psi(z)=\Gamma'(z)/\Gamma(z)$, where $\Gamma$ is the gamma function.

Note that, for each $n \geq 1$ and each $x \in T_n$, we have that $\rho_x = \frac{1-p_x}{p_x}$, where $p_x$ has the beta distribution with some positive parameters $(a_x, b_x)$. This entails that $\rho_x \sim \beta'(a_x,b_x)$ and that $\frac{1}{\rho_x} \sim \beta'(b_x, a_x)$, where $\beta'$ is the beta prime distribution.

For a random variable $\rho$ with a beta prime distribution with positive parameters $(a, b)$, the following formulae were respectively established in \cite[Equations (4.1) and (4.5)]{takeshima2000behavior}:
\begin{align} \label{psimeanvar}
\E{\log \rho}&=\psi(b)-\psi(a),
\hspace{1cm}
\text{Var}\left(\log \rho \right) = \psi'(a)+\psi'(b).
\end{align}

In our case, we have for each $n \geq 1$ and each $x \in T_n$, that 
\[
\rho_{x} 
\sim 
\displaystyle \beta' \left( \frac{\left((|x|-1) + n a_n^{-1} \right)^{\alpha} + \Delta_n}{2\Delta_n},  \frac{(|x|+n a_n^{-1})^{\alpha}}{2\Delta_n} \right).
\]
and $(\rho_{x})_{x \in T_n\setminus \{O_n\}}$ is a sequence of independent random variables. Next, we record three properties of the digamma function that we will use throughout this section.

\begin{itemize}
    \item Proved in \cite[Equation (6.6)]{takei2020almost}:
    \begin{equation} \label{asymp1}
    \psi\left(z+\frac{1}{2}\right)-\psi(z)\sim 
    \displaystyle \frac{1}{2z} \text{ as } z\to \infty.
    \end{equation}

    \item Proved in \cite[Equation (4.11)]{takeshima2000behavior}: for $s, t>0$, 
    \begin{equation} \label{asymp2}
    -\frac{1}{t} \le [\psi(t)-\psi(s)]-\log\left(\frac{t}{s}\right)\le \frac{1}{s}.
    \end{equation}
    
    \item Proved in \cite[Equation (4.12)]{takeshima2000behavior} and \cite[Equation (6.5)]{takei2020almost}:
    \begin{equation} \label{asymp3}
    \psi'(z)\sim 
    \displaystyle \frac{1}{z} \text{ as } z\to \infty.
    \end{equation}
\end{itemize}

\subsection{Claim \ref{claim:log sum exp}: expectation of the potential when $\alpha \leq 1$}

\begin{lemma}\label{lem:log sum exp app}
\begin{enumerate}[(i)]
    \item For all $x \in T_n$, we have that
    \[
    |\E{\log \rho_x}| \leq \frac{2\alpha}{\left(|x| + n a_n^{-1} \right)} + \frac{4\Delta_n}{\left(|x| + n a_n^{-1} \right)^{\alpha}}.
    \]
    \item For almost every $\omega\in \mathbf{\Omega}$, it holds uniformly over $t \in [0,1]$ as $n \to \infty$ that:
    \begin{align*}
\sum_{O_n\prec x \preceq x_{\lfloor 2nt \rfloor}} &\E{\log \rho_{x} }\to \frac{ \Delta [(\dt(O,t)+1)^{1-\alpha}-1]}{1-\alpha} - \alpha \log \left(\dt(O, t) + 1\right) \hspace{1cm} &\text{if }  \alpha<1,
\\
\sum_{O_n\prec x \preceq x_{\lfloor 2nt \rfloor}} &\E{\log \rho_{x}}\to (\Delta-1) \log(\dt(O,t)+1),  \hspace{1cm} &\text{if } \alpha = 1.
\end{align*}
\end{enumerate}
\end{lemma}
\begin{proof}
\begin{enumerate}[(i)]
    \item By \eqref{psimeanvar} and \eqref{asymp2}, we have that 
    \begin{align*}
        |\E{\log \rho_x}| &= \left|\psi\left(\frac{(|x|+n a_n^{-1})^{\alpha}}{2\Delta_n}\right) - \psi\left( \frac{\left((|x|-1) + n a_n^{-1} \right)^{\alpha} + \Delta_n}{2\Delta_n} \right) \right| \\
        &\leq \log \left( \left| \frac{(|x|+n a_n^{-1})^{\alpha}}{\left((|x|-1) + n a_n^{-1} \right)^{\alpha} + \Delta_n} \right| \right) + \frac{2\Delta_n}{\left((|x|-1) + n a_n^{-1} \right)^{\alpha} + \Delta_n} \\
        &\leq \frac{2\left(|x| + n a_n^{-1} \right)^{\alpha} - \left((|x|-1) + n a_n^{-1} \right)^{\alpha} +  4\Delta_n}{\left(|x| + n a_n^{-1} \right)^{\alpha} + \Delta_n} \leq \frac{2\alpha}{\left(|x| + n a_n^{-1} \right)} + \frac{4\Delta_n}{\left(|x| + n a_n^{-1} \right)^{\alpha}}.
    \end{align*}
    \item First note that it follows from \eqref{psimeanvar} that for all $t \in [0,1]$,
\begin{align} \label{takeimean}
\sum_{O_n\prec x\preceq x_{\lfloor 2nt \rfloor}} \E{\log(\rho_x)} = &\sum_{O_n\prec x\preceq x_{\lfloor 2nt \rfloor}} \left\{ \psi\left(\frac{\left((|x|-1) + n a_n^{-1} \right)^{\alpha} + \Delta_n}{2\Delta_n}\right)-\psi\left(\frac{(|x|+n a_n^{-1})^{\alpha}}{2\Delta_n}\right)\right\}
 \nonumber \\
 =& 
 \psi\left(\frac{\left(n a_n^{-1} \right)^{\alpha} + \Delta_n }{2\Delta_n}\right)-\psi\left(\frac{(|x_{\lfloor 2 n t\rfloor}|+n a_n^{-1})^{\alpha}}{2\Delta_n}\right)
 \nonumber \\
 &\qquad +
 \sum_{i=1}^{|x_{\lfloor 2n t\rfloor}|-1} \left\{ \psi\left(\frac{\left(i + n a_n^{-1} \right)^{\alpha}}{2\Delta_n}+\frac{1}{2}\right)-\psi\left(\frac{(i+n a_n^{-1})^{\alpha}}{2\Delta_n}\right)\right\}.
\end{align}

First, we assume that $\alpha\le 1$. For the first term in \eqref{takeimean}, we use \eqref{asymp2} and note that for almost every $\omega \in \Omega$,
\begin{align*}
\log\left(\frac{\left(n a_n^{-1} \right)^{\alpha}+\Delta_n}{2\Delta_n}\right)-\log\left(\frac{(|x_{\lfloor 2 n t\rfloor}|+n a_n^{-1})^{\alpha}}{2\Delta_n}\right)
&=\log\left(\frac{(n a_n^{-1})^{\alpha} + \Delta_n}{(|x_{\lfloor 2 n t\rfloor}|+n a_n^{-1})^{\alpha}}\right)
\\
&= -\alpha \log(d(O,t)+1) + o(1),
\end{align*}
as $n\to \infty$ by \eqref{eqn:stable tree scaling limit def 1}. Thus, by \eqref{asymp2}, for almost every $\omega \in \mathbf{\Omega}$ the first term in \eqref{takeimean} satisfies
\[
\left|\psi\left(\frac{\left(n a_n^{-1} \right)^{\alpha}+\Delta_n}{2\Delta_n}\right)-\psi\left(\frac{(|x_{\lfloor 2 n t\rfloor}|+n a_n^{-1})^{\alpha}}{2\Delta_n}\right)+\alpha \log(d(O,t)+1)\right| \le \frac{2\Delta_n}{(n a_n^{-1})^{\alpha}} + o(1),
\]
as $n \to \infty$, and therefore
\[
\psi\left(\frac{\left(n a_n^{-1} \right)^{\alpha}+\Delta_n}{2\Delta_n}\right)-\psi\left(\frac{(|x_{\lfloor 2 n t\rfloor}|+n a_n^{-1})^{\alpha}}{2\Delta_n}\right) = -\alpha \log(d(O,t)+1) + o(1),
\]
as $n\to \infty$. 

For the second term, since $\inf_{i \geq 0} \frac{(n^{-1} a_n  + i)^{\alpha}}{2 \Delta_n}\to \infty$ as $n\to \infty$, it follows from \eqref{asymp1} that for all $i \geq 0$,
\[
\psi\left(\frac{\left(i + n a_n^{-1} \right)^{\alpha}}{2\Delta_n}+\frac{1}{2}\right)-\psi\left(\frac{(i+n a_n^{-1})^{\alpha}}{2\Delta_n}\right) =  \frac{\Delta_n}{(i + n a_n^{-1})^{\alpha}} (1+o(1)),
\]
where the $o(1)$ term is uniform over all $i \geq 0$ (but may depend on $\omega$).

This implies that for almost every $\omega \in \mathbf{\Omega}$ it holds for all $t \in [0,1]$ that
\begin{align*}
& \sum_{i=n a_n^{-1}+1}^{|x_{\lfloor 2 n t\rfloor}|-1+n a_n^{-1}} \left\{\psi\left(\frac{i^{\alpha}}{2\Delta_n}+\frac{1}{2}\right)-\psi\left(\frac{i^{\alpha}}{2\Delta_n}\right)\right\}
\\ \\
&\quad \sim 
\sum_{i=n a_n^{-1}+1}^{|x_{\lfloor 2 n t\rfloor}|-1+n a_n^{-1}} \frac{\Delta_n}{i^{\alpha}}
\\ \\
&\quad \sim 
\begin{cases}
\displaystyle \frac{\Delta_n}{1-\alpha} \left[(|x_{\lfloor 2 n t\rfloor}|-1+n a_n^{-1})^{1-\alpha}-(n a_n^{-1})^{1-\alpha}\right] &\text{if } \alpha<1,
\\
\displaystyle \Delta \left[\log(|x_{\lfloor 2 n t\rfloor}|-1+n a_n^{-1})-\log(n a_n^{-1})\right] &\text{if } \alpha=1,
\end{cases}
\\ \\
& \quad \sim  
\begin{cases}
\displaystyle \frac{\Delta}{1-\alpha} \left[\left[\frac{(|x_{\lfloor 2 n t\rfloor}|-1)+n a_n^{-1}}{n a_n^{-1}}\right]^{1-\alpha}-1\right] &\text{ if } \alpha<1,
\\
\displaystyle \Delta \log(n^{-1} a_n  |x_{\lfloor 2 n t\rfloor}|-n^{-1} a_n+1) &\text{ if } \alpha=1,
\end{cases}
\\ \\
& \quad \sim  
\begin{cases}
\displaystyle \frac{\Delta\left[(d(0,t)+1)^{1-\alpha}-1\right]}{1-\alpha} &\text{ if } \alpha <1,
\\
\displaystyle \Delta \log(d(0,t)+1) &\text{ if } \alpha=1.
\end{cases}
\end{align*}
Since $\sup_{t \in [0,1]} d(0,t)$ is bounded for almost every $\omega \in \mathbf{\Omega}$, this implies the result.
\end{enumerate}
\end{proof}

\subsection{Claim \ref{claim:log var}: variance of the potential when $\alpha \leq 1$}
\begin{prop}\label{cor:log var app}
For almost every $\omega\in \mathbf{\Omega}$, it holds uniformly over $t \in [0,1]$ as $n \to \infty$ that:
\begin{align*}
\sum_{O_n\prec x \preceq x_{\lfloor 2nt \rfloor}} &\textnormal{Var}[\log (\rho_{x}) ]\to \frac{4 \Delta [(\dt(O,t)+1)^{1-\alpha}-1]}{1-\alpha} \hspace{1cm} &\text{if }  \alpha<1,
\\
\sum_{O_n\prec x \preceq x_{\lfloor 2nt \rfloor}} &\textnormal{Var}[\log (\rho_{x})]\to 4 \Delta \log(\dt(O,t)+1),  \hspace{1cm} &\text{if } \alpha = 1.
\end{align*}
\end{prop}
\begin{proof}
It follows from \eqref{psimeanvar} that for all $t \in [0,1]$,
\[
\sum_{O_n\prec x\preceq x_{\lfloor 2nt \rfloor}} \text{Var}(\log \rho_x)
 =
 \sum_{i=1}^{|x_{\lfloor 2n t\rfloor}|} \left\{ \psi'\left(\frac{\left((i-1) + n a_n^{-1} \right)^{\alpha}}{2\Delta_n}+\frac{1}{2}\right)+\psi'\left(\frac{(i+n a_n^{-1})^{\alpha}}{2\Delta_n}\right)\right\}.
\]
First we assume that $\alpha\le 1$. Since both arguments in the derivatives of the digamma function blow up uniformly over $i \geq 0$ as $n \to \infty$, it follows from \eqref{asymp3} that for all $i \geq 1$,
\begin{align*}
   \psi'\left(\frac{(i+n a_n^{-1})^{\alpha}}{2\Delta_n}\right) + \psi'\left(\frac{\left((i-1) + n a_n^{-1} \right)^{\alpha}}{2\Delta_n}+\frac{1}{2}\right) 
   &= \frac{4 \Delta_n}{(i + n a_n^{-1})^{\alpha}}(1+o(1)),
\end{align*}
as $n \to \infty$, where the $o(1)$ term is uniform over all $i \geq 1$. This implies that 
\begin{align*}
\sum_{O_n\prec x\preceq x_{\lfloor 2nt \rfloor}} \text{Var}(\log \rho_x) &\sim \sum_{i=1+n a_n^{-1}}^{|x_{\lfloor 2n t\rfloor}|+n a_n^{-1}} \frac{4 \Delta_n}{i^{\alpha}}
\\ \\
&\sim 
\begin{cases}
\displaystyle \frac{4 \Delta_n}{1-\alpha} \left[(|x_{\lfloor 2 n t\rfloor}|+n a_n^{-1})^{1-\alpha}-(1+n a_n^{-1})^{1-\alpha}\right] &\text{if } \alpha<1,
\\
\displaystyle 4\Delta \left[\log(|x_{\lfloor 2 n t\rfloor}|+n a_n^{-1})-\log(1+n a_n^{-1})\right] &\text{if } \alpha=1,
\end{cases}
\\ \\
&\sim  
\begin{cases}
\displaystyle \frac{4 \Delta}{1-\alpha} \left[\left[\frac{|x_{\lfloor 2 n t\rfloor}|+n a_n^{-1}}{n a_n^{-1}}\right]^{1-\alpha}-\left[\frac{1+n a_n^{-1}}{n a_n^{-1}}\right]^{1-\alpha}\right] &\text{if } \alpha<1,
\\
\displaystyle 4\Delta \left[\log(n^{-1} a_n  |x_{\lfloor 2 n t\rfloor}|+1)-\log (n^{-1}a_n+1)\right] &\text{if } \alpha=1,
\end{cases}
\\ \\
&\sim  
\begin{cases}
\displaystyle \frac{4 \Delta \left[(d(0,t)+1)^{1-\alpha}-1\right]}{1-\alpha} &\text{if } \alpha <1,
\\
\displaystyle 4\Delta \log(d(0,t)+1) &\text{if } \alpha=1.
\end{cases} 
\end{align*}
Again, since $\sup_{t \in [0,1]} d(0,t)$ is bounded for almost every $\omega \in \mathbf{\Omega}$, this implies the result.
\end{proof}

\subsection{Claim \ref{claim:log mgf}: exponential moments when $\alpha \leq 1$}

\begin{lemma}\label{claim:log mgf app}
For almost every $\omega \in \mathbf{\Omega}$, it holds for all $x \in T_n$ and all $1 \leq k \leq (na_n^{-1})^{1/2}$ that
\begin{align*}
    \E{e^{k\log \rho_{x}}}  
    &\leq \exp \left\{\left(1+3\Delta\right) \left( \frac{2k\alpha}{|x|+n a_n^{-1}} + \frac{3k^2\Delta_n}{(|x|+n a_n^{-1})^{\alpha}} \right)  \right\},
\end{align*}
\begin{align*}
    \E{e^{k\log (\rho_{x}^{-1})}}  
    \leq \exp \left\{\left(1+3\Delta\right) \left( \frac{2k\alpha}{|x|+n a_n^{-1}} + \frac{3k^2\Delta_n}{(|x|+n a_n^{-1})^{\alpha}} \right)  \right\}.
\end{align*}
\end{lemma}
\begin{proof}
First recall that, since 
\[
\rho_x\sim \beta ' \left(  \frac{(|x|-1 + n a_n^{-1})^{\alpha} + \Delta_n}{2\Delta_n},  \frac{(|x|+na_n^{-1})^{\alpha}}{2\Delta_n} \right),
\]
we also have that 
\[
\frac{1}{\rho_x}\sim \beta ' \left(  \frac{(|x|+na_n^{-1})^{\alpha}}{2\Delta_n} ,\frac{(|x|-1+na_n^{-1})^{\alpha} + \Delta_n}{2\Delta_n}\right).
\]
Consequently (using the formula for the $k^{th}$ moment of the beta prime distribution), it holds for any $k \geq 1$ that
\begin{equation}\label{eqn:moment formula}
    \E{\rho_x^k} = \prod_{j=0}^{k-1} \frac{(|x|-1+n a_n^{-1})^{\alpha} + (2j+1)\Delta_n}{(|x|+na_n^{-1})^{\alpha} - (2j+2)\Delta_n},
\end{equation}
\begin{equation} \label{eqn:negative moment formula}
\E{\rho_x^{-k}} = \prod_{j=0}^{k-1} \frac{(|x|+na_n^{-1})^{\alpha} + 2j\Delta_n}{(|x|-1+n a_n^{-1})^{\alpha} - (2j+1)\Delta_n}.
\end{equation}
Now note from \eqref{eqn:moment formula} that
\begin{align*}
    \E{e^{k\log \rho_{x}}} &= \E{\rho_{x}^k} \\
    &= \prod_{j=0}^{k-1} \frac{(|x|-1+n a_n^{-1})^{\alpha} + (2j+1)\Delta_n}{(|x|+n a_n^{-1})^{\alpha} - (2j+2)\Delta_n} \\
    &\leq \prod_{j=0}^{k-1} \left( 1 + \left[ \frac{|(|x|-1+n a_n^{-1})^{\alpha} - (|x|+n a_n^{-1})^{\alpha}|}{(|x|+n a_n^{-1})^{\alpha}} + \frac{(4j+3) \Delta_n}{(|x|+n a_n^{-1})^{\alpha}}\right] \left[1+\frac{2(2j+2)\Delta_n}{(|x|+n a_n^{-1})^{\alpha}}\right]
    \right\} \\
    &\leq \exp \left\{\left(1+3\Delta\right) \left( \frac{2k\alpha}{|x|+n a_n^{-1}} + \frac{3k^2\Delta_n}{(|x|+n a_n^{-1})^{\alpha}} \right)  \right\}.
\end{align*}
Similarly, note from \eqref{eqn:negative moment formula} that 
\begin{align*}
    \E{e^{k\log (\rho_{x}^{-1})}} &= \E{(\rho_{x})^{-k}} \\
    &= \prod_{j=0}^{k-1} \frac{(|x|+na_n^{-1})^{\alpha} + 2j\Delta_n}{(|x|-1+n a_n^{-1})^{\alpha} - (2j+1)\Delta_n} \\
    &\leq \prod_{j=0}^{k-1} \left( 1 + \left[\frac{|(|x|+n a_n^{-1})^{\alpha} - (|x|-1+n a_n^{-1})^{\alpha}|}{(|x|-1+n a_n^{-1})^{\alpha}} +\frac{(4 j+1) \Delta_n}{(|x|-1+n a_n^{-1})^{\alpha}}\right]\left[ 1+\frac{2(2j+1)\Delta_n}{(|x|+n a_n^{-1})^{\alpha}}\right] \right) \\
    &\leq \exp \left\{\left(1+3\Delta\right) \left( \frac{2k\alpha}{|x|+n a_n^{-1}} + \frac{3k^2\Delta_n}{(|x|+n a_n^{-1})^{\alpha}} \right)  \right\}.
\end{align*}
\end{proof}

\subsection{Expectation of the potential when $\alpha > 1$}

In the transient regime $\alpha > 1$ considered in Section \ref{sctn:trans results}, we do not rescale the initial weights so for each $x \in \Ti$ we have
\[
\rho_{x} 
\sim 
\displaystyle \beta' \left( \frac{\left((|x|-1) \right)^{\alpha} + \Delta}{2\Delta},  \frac{|x|^{\alpha}}{2\Delta} \right).
\]

\begin{lemma}\label{lem:log sum exp app 2}
Assume $\alpha >1$. Then for almost every $\omega\in \mathbf{\Omega}$, the following holds.
As $|y| \to \infty$, we have that
    \begin{align*}
\sum_{O_{\infty}\prec x \preceq y} &\E{\log \rho_{x} } = - \alpha \log |y| + O(1).
\end{align*}
\end{lemma}
\begin{proof}
First note that it follows from \eqref{psimeanvar} that for any $y \in \Ti$,
\begin{align} \label{takeimean2}
\sum_{O_{\infty}\prec x\preceq y} \E{\log(\rho_x)} = &\sum_{O_{\infty}\prec x\preceq y} \left\{ \psi\left(\frac{\left(|x|-1\right)^{\alpha} + \Delta}{2\Delta}\right)-\psi\left(\frac{|x|^{\alpha}}{2\Delta}\right)\right\}
 \nonumber \\
 =& 
 \psi\left(\frac{1}{2}\right)-\psi\left(\frac{|y|^{\alpha}}{2\Delta}\right)
 \nonumber \\
 &\qquad +
 \sum_{i=1}^{|y|-1} \left\{ \psi\left(\frac{i^{\alpha}}{2\Delta}+\frac{1}{2}\right)-\psi\left(\frac{i^{\alpha}}{2\Delta}\right)\right\}.
\end{align}

For the first term in \eqref{takeimean2}, we use \eqref{asymp2} which implies that, as $|y| \to \infty$,
\[
 \psi\left(\frac{1}{2}\right)-\psi\left(\frac{|y|^{\alpha}}{2\Delta}\right) = \log \left( \frac{\Delta}{|y|^{\alpha}} \right) + O(1) = -\alpha \log |y| + O(1).
\]

For the second term, since $\frac{ i^{\alpha}}{2 \Delta }\to \infty$ as $i\to \infty$, it follows from \eqref{asymp1} that as $i \to \infty$,
\[
\psi\left(\frac{i^{\alpha}}{2\Delta}+\frac{1}{2}\right)-\psi\left(\frac{i^{\alpha}}{2\Delta}\right) =  \frac{\Delta}{i^{\alpha}} (1+o(1)).
\]
This implies that
\begin{align*}
\sum_{i=1}^{|y|-1} \left\{\psi\left(\frac{i^{\alpha}}{2\Delta}+\frac{1}{2}\right)-\psi\left(\frac{i^{\alpha}}{2\Delta}\right)\right\}
&\sim 
\sum_{i=1}^{|y|-1} \frac{\Delta}{i^{\alpha}} = O(1)
\end{align*}
as $|y| \to \infty$.
\end{proof}

\subsection{Variance of the potential when $\alpha > 1$}
\begin{lemma}\label{cor:log var app 2}
Assume $\alpha > 1$. Then for almost every $\omega\in \mathbf{\Omega}$, as $|y| \to \infty$ we have that
\begin{align*}
\sum_{O_n\prec x \preceq y} &\textnormal{Var}[\log (\rho_{x}) ] = O(1).
\end{align*}
\end{lemma}
\begin{proof}
It follows from \eqref{psimeanvar} that for all $t \in [0,1]$,
\[
\sum_{O_{\infty}\prec x\preceq y} \text{Var}(\log \rho_x)
 =
 \sum_{i=1}^{|y|} \left\{ \psi'\left(\frac{(i-1)^{\alpha}}{2\Delta}+\frac{1}{2}\right)+\psi'\left(\frac{i^{\alpha}}{2\Delta}\right)\right\}.
\]
Since both arguments in the derivatives of the digamma function blow up as $i \to \infty$, it follows from \eqref{asymp3} that as $i \to \infty$,
\begin{align*}
    \psi'\left(\frac{(i-1)^{\alpha}}{2\Delta}+\frac{1}{2}\right)+\psi'\left(\frac{i^{\alpha}}{2\Delta}\right)
   &= \frac{4 \Delta}{i^{\alpha}}(1+o(1)).
\end{align*}
This implies that 
\begin{align*}
\sum_{O_{\infty}\prec x\preceq y} \text{Var}(\log \rho_x) &\sim \sum_{i=1}^{|y|} \frac{4 \Delta}{i^{\alpha}} = O(1)
\end{align*}
as $|y| \to \infty$.
\end{proof}

\end{appendix}

\bibliographystyle{alpha}
\bibliography{references}

\end{document}